\documentclass[11pt]{amsart}
\usepackage{latexsym,amscd,amssymb, graphicx, shuffle}  
\usepackage[margin=1.0in]{geometry}

\numberwithin{equation}{section}
\newtheorem{theorem}{Theorem}[section]
\newtheorem{proposition}[theorem]{Proposition}
\newtheorem{corollary}[theorem]{Corollary}
\newtheorem{lemma}[theorem]{Lemma}

\newtheorem{problem}[theorem]{Problem}

\newtheorem*{strongc}{Strong Conjecture}
\newtheorem*{intermediatec}{Intermediate Conjecture}
\newtheorem*{weakc}{Weak Conjecture}

\newtheorem{defn}[theorem]{Definition}
\theoremstyle{definition}
\newtheorem{example}[theorem]{Example}

\newcommand{\Cat}{{{\sf Cat}}}
\newcommand{\Park}{{{\sf Park}}}

\newcommand{\mult}{{\mathrm{mult}}}

\newcommand{\Shi}{{{\sf Shi}}}
\newcommand{\Cox}{{{\sf Cox}}}

\newcommand{\leftexp}[2]{{\vphantom{#2}}^{#1}{#2}}

\newcommand{\symm}{{\mathfrak{S}}}

\newcommand{\CC}{{\mathbb {C}}}
\newcommand{\ZZ}{{\mathbb {Z}}}
\newcommand{\RR}{{\mathbb {R}}}

\newcommand{\AAA}{{\mathbb {A}}}
\newcommand{\DD}{{\mathbb {D}}}

\newcommand{\xx}{{\mathbf{x}}}

\newcommand{\LLL}{{\mathcal{L}}}

\begin{document}

\date{September 2012}

\title[Parking structures: Fuss analogs]
{Parking structures: Fuss analogs}

\author{Brendon Rhoades}
\address
{Deptartment of Mathematics \newline \indent
University of California, San Diego \newline \indent
La Jolla, CA, 92093-0112, USA}
\email{bprhoades@math.ucsd.edu}

\thanks{The author is partially supported by NSF grant DMS - 1068861.}



\begin{abstract}
For any irreducible real reflection group $W$ with Coxeter number 
$h$, Armstrong, Reiner, and the author introduced a pair of 
$W \times \ZZ_h$-modules which deserve to be called {\sf $W$-parking spaces} 
which generalize the type A notion of parking functions and conjectured a 
relationship between them.
In this paper we give a 
Fuss analog of their constructions.  

For a Fuss parameter $k \geq 1$, we  define a pair
of $W \times \ZZ_{kh}$-modules which deserve to be 
called {\sf $k$-$W$-parking spaces}
and conjecture a relationship between them.
 We prove 
the weakest version of our conjectures for each of the infinite families ABCDI
of finite reflection groups, together with proofs of stronger versions in special cases.
Whenever our weakest conjecture holds for $W$, we have the following
corollaries.
\begin{itemize}
\item There is a simple formula for the character of 
either $k$-$W$-parking space.
\item We recover a cyclic sieving result due to Krattenthaler 
and M\"uller
which gives the cycle structure of a generalized
rotation action on $k$-$W$-noncrossing partitions.
\item When $W$ is crystallographic, the restriction of either $k$-$W$-parking
space to $W$ 
isomorphic to the action of 
$W$ on the finite torus $Q / (kh+1)Q$, where $Q$ is the root lattice.
\end{itemize}
\end{abstract}
\maketitle

\section{Introduction}
\label{Introduction}

The set $\Park_n$ of 
parking functions 
of size $n$ plays a major role in Catalan combinatorics
with its connections to the Shi hyperplane arrangement and diagonal 
coinvariant modules.
In \cite{ARR} Armstrong, Reiner, and 
the author define two new generalizations 
$\Park^{NC}_W$ and $\Park^{alg}_W$
of $\Park_n$
to any (finite, real, irreducible) reflection group $W$.  The generalization 
$\Park^{NC}_W$ is defined combinatorially using
$W$-noncrossing partitions and the generalization $\Park^{alg}_W$ is defined 
algebraically using homogeneous systems of parameters.  When $W$ is a Weyl group with root system $\Phi$, a third generalization 
$\Park^{NN}_{\Phi}$ is defined combinatorially in \cite{ARR} using 
$\Phi$-nonnesting partitions.

Let $k \in \ZZ_{> 0}$ be a Fuss parameter.  The parking functions $\Park_n$ have a natural
Fuss generalization $\Park_n(k)$.
In this paper, we `complete the diamonds' and present a natural Fuss generalization
of the constructions in \cite{ARR}.  

\begin{equation*}
\begin{matrix}
&& \Park^{NC}_W(k) && \\
& \diagup && \diagdown & \\
\Park^{NC}_W &&&& \Park_n(k) \\
& \diagdown && \diagup & \\
&& \Park_n && 
\end{matrix} \hspace{1in}
\begin{matrix}
&& \Park^{alg}_W(k) && \\
& \diagup && \diagdown & \\
\Park^{alg}_W &&&& \Park_n(k) \\
& \diagdown && \diagup & \\
&& \Park_n && 
\end{matrix} 
\end{equation*}
\begin{equation*}
\begin{matrix}
&& \Park^{NN}_{\Phi}(k) && \\
& \diagup && \diagdown & \\
\Park^{NN}_{\Phi} &&&& \Park_n(k) \\
& \diagdown && \diagup & \\
&& \Park_n && 
\end{matrix}
\end{equation*}

We begin by recalling the classical notion of a parking function.

\subsection{Classical parking functions} A sequence of positive integers 
$(a_1, \dots, a_n)$ is called a {\sf parking function of size $n$} 
 \footnote{The term `parking function' arises from the following situation.  Suppose
 that $n$ cars wish to park in a parking lot consisting of $n$ linearly ordered parking
 spaces.  For $1 \leq i \leq n$, car $i$ wants to park in the space $a_i$.  Cars $1, 2, \dots, n$
 (in that order) try to park in the lot.  
Car $i$ parks in the first available space $\geq i$, unless there are no such available spaces,
in which case car $i$ leaves the lot.
 The sequence $(a_1, \dots, a_n)$ is a parking function if and only if every car can 
 park in the lot.}
if its 
nondecreasing rearrangement $(b_1 \leq \dots \leq b_n)$ satisfies
$b_i \leq i$ for all $i$. 
We denote by $\Park_n$ the set of parking functions of size $n$.
The set $\Park_n$ is famously enumerated by $|\Park_n| = (n+1)^{n-1}$.
Table~\ref{park} shows that sets $\Park_n$ for $n = 1,2,3$.
 Parking functions were originally introduced by 
Konheim and Weiss \cite{KonheimWeiss} 
in computer science, but have since received a great deal of attention in 
algebraic combinatorics.

\begin{table}
\label{park}
\caption{$\Park_n$ for $n = 1,2,3$}

\begin{tabular} {| p{1cm} |  p{7cm} |}
\hline
$n$ & $\Park_n$ \\ \hline
$1$ & $1$ \\ \hline
$2$ & $11, 12, 21$ \\ \hline
$3$ & $111, 112, 121, 211, 113, 131, 311, 122, 212,$ \\ & $221, 123, 213, 132, 312, 231, 321$ \\ \hline
\end{tabular}
\end{table}

In this paper we will focus on a module structure carried by the set of parking functions.
The symmetric group $\symm_n$ acts on  $\Park_n$ by coordinate permutation,
viz. $w.(a_1, \dots, a_n) := (a_{w(1)}, \dots, a_{w(n)})$ for $w \in \symm_n$ and 
$(a_1, \dots, a_n) \in \Park_n$.
If 
$\chi: \symm_n \rightarrow \CC$ is the character of this 
representation, it can be shown that
\begin{equation}
\chi(w) = (n+1)^{r(w) - 1},
\end{equation}
where $r(w)$ denotes the number of cycles in the permutation $w \in \symm_n$.
Specializing to $w = 1$, we recover the fact that $| \Park_n | = (n+1)^{n-1}$.

Stanley \cite{StanleyParking} discovered a connection between the $\symm_n$-module
$\Park_n$ and noncrossing set partitions.  
Given a partition
$\lambda = (\lambda_1 \geq \dots \geq \lambda_r) \vdash n$, let $NC(\lambda)$ denote
the number of noncrossing partitions of $[n]$ with block sizes $\lambda_1, \dots, \lambda_r$.
Stanley proved that the $\symm_n$-module $\Park_n$ decomposes as a direct
sum of coset representations:
 \begin{equation}
 \label{iso}
 \Park_n \cong_{\symm_n} \bigoplus_{\lambda \vdash n} NC(\lambda) {\bf 1}_{\symm_{\lambda}}^{\symm_n},
 \end{equation}
where $\symm_{\lambda} = \symm_{\lambda_1} \times \cdots \times \symm_{\lambda_r}$ is the 
Young subgroup of $\symm_n$ corresponding to the partition 
$\lambda = (\lambda_1, \dots, \lambda_r)$.  
Kreweras \cite{Kreweras} proved 
a product formula for the multiplicities $NC(\lambda)$:
\begin{equation}
\label{krew}
NC(\lambda) = \frac{n!}{(n - \ell(\lambda) + 1)! m_1(\lambda)! \cdots m_n(\lambda)!},
\end{equation}
where $\ell(\lambda)$ is the number of parts of $\lambda$ and $m_i(\lambda)$ is the multiplicity
of $i$ as a part of $\lambda$.  

\subsection{Fuss analogs of classical parking functions}  The Fuss analog of a classical parking 
function depends on a {\sf Fuss parameter} $k \in \ZZ_{> 0}$.  
A {\sf $k$-parking function of size $n$} is a 
sequence $(a_1, \dots, a_n)$ of length $n$ whose nondecreasing rearrangement
$(b_1 \leq \dots \leq b_n)$ satisfies $b_i \leq (k-1)i + 1$ for all $i$.  The set of $k$-parking functions of size
$n$ is denoted $\Park_n(k)$ and carries an action of $\symm_n$ via coordinate permutation.  
We have
that $\Park_n(1) = \Park_n$.  Table~\ref{fpark} shows $\Park_n(k)$ for $k = 2$ and $n = 1,2,3$.
The set $\Park_n(k)$ labels 
 regions of extended Shi arrangement and is related to the representation
theory of modules which generalize diagonal harmonics.  It can be shown that
$|\Park_n(k)| = (kn+1)^{n-1}$ and, more generally, if $\chi : \symm_n \rightarrow \CC$ is the 
character of $\Park_n(k)$,
\begin{equation}
\label{acharacter}
\chi(w) = (kn+1)^{r(w) - 1},
\end{equation}
where $r(w)$ is the number of cycles of $w$.

Stanley \cite{StanleyParking} related the representation theory of $\Park_n(k)$ to $k$-divisible noncrossing
partitions of $[kn]$.  
Given a partition $\lambda = (\lambda_1, \dots, \lambda_r) \vdash n$, let
$NC^k(\lambda)$ be the number of noncrossing partitions of $[kn]$ with block sizes
$k \lambda_1, \dots, k \lambda_r$.
The module $\Park_n(k)$ is isomorphic to the following sum of coset representations:
\begin{equation}
\label{fussiso}
 \Park_n(k) \cong_{\symm_n} \bigoplus_{\lambda \vdash n} NC^k(\lambda) {\bf 1}_{\symm_{\lambda}}^{\symm_n}.
\end{equation} 
The Kreweras product formula (\ref{krew}) can be used to calculate the multiplicities 
$NC^k(\lambda)$.

Armstrong, Reiner, and the author generalized $\Park_n$ to arbitrary real reflection
groups \cite{ARR}.  In this paper we will generalize $\Park_n(k)$.  We begin 
by discussing a generalization of $\Park_n(k)$ to arbitrary crystallographic type.

\begin{table}
\label{fpark}
\caption{$\Park_n(k)$ for $k = 2$ and $n = 1,2,3$}
\begin{tabular} {| p{1cm} |  p{9cm} |}
\hline
$n$ & $\Park_n(2)$ \\ \hline
$1$ & $1$ \\ \hline
$2$ & $11, 12, 21, 13, 31$ \\ \hline
$3$ & $111, 112, 121, 211, 113, 131, 311, 122, 212, 221, 133, 313,$ \\
& $331, 123, 213, 132, 231, 312, 321, 124, 214, 142, 241, 412,$ \\
& $421, 125, 215, 152, 251, 512, 521, 134, 314, 143, 413, 341,$ \\
& $431, 135, 315, 153, 351, 513, 531$ \\ \hline
\end{tabular}
\end{table}

\subsection{Crystallographic generalizations of $k$-parking functions}  Let
\begin{equation*}
\Pi \subset \Phi^+ \subset \Phi \subset Q 
\end{equation*}
denote a crystallographic root system $\Phi$, a choice of simple system $\Pi \subset \Phi$, the corresponding
set $\Phi^+$ of positive roots, and the root lattice $Q = \ZZ[\Pi]$.  
Let $W = W(\Phi)$ be the Weyl group corresponding to $\Phi$ and let $h$ be the Coxeter 
number of $W$.  The group $W$ acts on the lattice $Q$ and on its dilation $(kh+1)Q$.  
Hence, the finite torus $Q / (kh+1)Q$ carries an action of $W$.

Haiman \cite{Haiman} used classical properties of reflection groups to prove 
that the character $\chi : W \rightarrow \CC$ of the permutation module 
$Q / (kh+1)Q$ is given by
\begin{equation}
\label{crystcharacter}
\chi(w) = (kh+1)^{\dim(V^w)},
\end{equation}
where $V$ is the reflection representation of $W$ and $V^w = \{ v \in V \,:\, w.v = v \}$ is the 
fixed space of $w$.  

In type A$_{n-1}$, we have that 
$V = \{ (x_1, \dots, x_n) \in \CC^n \,:\, x_1 + \cdots + x_n = 0 \}$, so that 
$V^w$ is the number of cycles of $w$ minus one for $w \in W = \symm_n$.  Therefore,
the formula (\ref{crystcharacter}) restricts to the formula 
(\ref{acharacter}) in type A, so that $Q / (kh+1) Q$ is a generalization of 
$\Park_n(k)$.   
Also, the $W$-module $Q / (kh+1) Q$ is as combinatorial as possible, i.e., 
an explicit action of $W$
on a finite set.
However, this construction depends
on the root lattice $Q$, and therefore does not extend to noncrystallographic 
real reflection groups.  We present a generalization to 
arbitrary real reflection groups involving homogeneous
systems of parameters.

\subsection{Homogeneous systems of parameters;  arbitrary finite type}  
\label{Homogeneous systems of parameters;  arbitrary finite type}  
Let $W$ be a real
reflection group with Coxeter number $h$ and let $V$ be the reflection representation of
$W$ considered over $\CC$.  Let $\CC[V]$ be the coordinate ring of polynomial functions on
$W$, equipped with its usual polynomial grading.  The contragredient action of $W$
on $\CC[V]$ given by $(w.f)(v) := f(w^{-1}.v)$ for $w \in W, f \in \CC[V],$ and $v \in V$ gives
$\CC[V]$ the structure of a graded $W$-module.  We let $g$ be a distinguished generator
of the cyclic group $\ZZ_{kh}$ and $\omega$ be a primitive $kh^{th}$ root of unity.

Let $\theta_1, \dots, \theta_n \in \CC[V]$ be a homogeneous system of parameters (hsop) of degree
$kh+1$ carrying 
$V^*$ and let $(\Theta) \subset \CC[V]$ be the ideal generated by the $\theta_i$.  
That is, the $\theta_i$ are homogeneous of degree $kh+1$
the zero set in $V$ cut out by $\theta_1 = \cdots = \theta_n = 0$ consists only of
the origin $\{ 0 \}$, the $\CC$-linear span of $\theta_1, \dots, \theta_n$ is stable under the 
action of $W$, and $\mathrm{span}_{\CC}\{ \theta_1, \dots, \theta_n \} \cong_W V^*$.
\footnote{The existence of such hsops is a deep consequence of the theory of rational Cherednik 
algebras \cite{BEG, Etingof, Gordon}.}

The ideal 
$(\Theta)$ is $W$-stable, 
so  the action of $W$ descends to the quotient
$\CC[V] / (\Theta)$.   Moreover, since the $\theta_i$ are homogeneous of degree $kh+1$, the 
quotient $\CC[V] / (\Theta)$ carries an action of  $\ZZ_{kh}$, where
 $g \in \ZZ_{kh}$ scales by $\omega^d$ in degree $d$.  The action of 
$W$ commutes with the action of $\ZZ_{kh}$, giving
$\CC[V] / (\Theta)$ the structure of a $W \times \ZZ_{kh}$-module. 

It can be shown \cite{Haiman} that the module
$\CC[V] / (\Theta)$ is finite dimensional over $\CC$ and that its character 
$\chi: W \times \ZZ_{kh} \rightarrow \CC$ is given by
\begin{equation}
\label{hsopchar}
\chi(w, g^d) = (kh+1)^{\mult_w(\omega^d)},
\end{equation}
where $\mult_w(\omega^d)$ is the multiplicity of the eigenvalue 
$\omega^d$ in the spectrum of $w$ acting on $V$.  Restricting this action to $W$ and comparing with
(\ref{crystcharacter}), we see that $\CC[V] / (\Theta) \cong_W Q / (kh+1) Q$ in crystallographic type.
Moreover, the quotient $\CC[V] / (\Theta)$ carries the additional structure of a 
$W \times \ZZ_{kh}$-module which is not apparent in 
$Q / (kh+1) Q$ (or in $\Park_n(k)$).  However, it is an open problem to find an explicit basis
of $\CC[V] / (\Theta)$ which realizes $\CC[V] / (\Theta)$ as a $W$-permutation module isomorphic
to $Q / (kh+1)Q$.
It is an aim of this paper to define
a parking space for any reflection group
which is explicitly a $W \times \ZZ_{kh}$-permutation module with character
(\ref{hsopchar}).

\subsection{ The noncrossing parking space in type A }  
In \cite{ARR} a 
$W \times \ZZ_h$-module $\Park^{NC}_W$ is constructed for any real reflection group $W$.
We review this construction when $W = \symm_n$ is of type A$_{n-1}$.  

\begin{figure}
\includegraphics[scale=0.6]{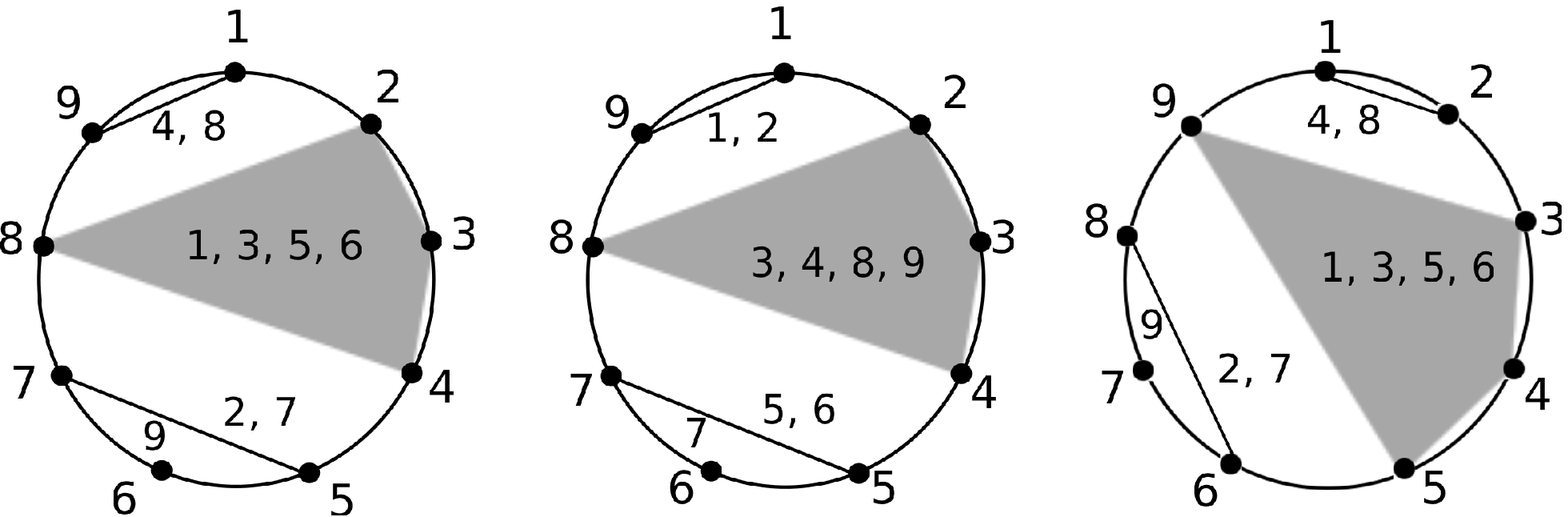}
\caption{Three noncrossing parking functions for $\symm_9$}
\label{fig:k1park}
\end{figure}

The elements of 
$\Park^{NC}_{\symm_n}$ are pairs $(\pi, f)$ where $\pi$ is a noncrossing partition of $[n]$ and 
$f: B \mapsto f(B)$ assigns each block $B \in \pi$ to a subset $f(B) \subseteq [n]$ such that
$|B| = |f(B)|$ and $[n] = \biguplus_{B \in \pi} f(B)$.  We can visualize the pair 
$(\pi, f)$ on the disc $\DD^2$ by drawing the noncrossing partition $\pi$ and labeling the block
$B$ of $\pi$ with the subset $f(B)$.  The group $\symm_n$ acts on the labels of these diagrams and
the cyclic group $\ZZ_h = \ZZ_n$ acts by clockwise rotation.

Figure~\ref{fig:k1park} shows three elements of $\Park^{NC}_{\symm_9}$ drawn on the disc
$\DD^2$.  The parking function on the left corresponds to the pair
$(\pi, f)$, where $\pi = \{ 1, 9 / 2, 3, 4, 8 / 5, 7 / 6 \}$ and the labeling $f$ is given by
\begin{align*}
f: \{1, 9 \} &\mapsto \{4, 8 \}, \\
\{ 2, 3, 4, 8 \} & \mapsto \{ 1, 3, 5, 6 \}, \\
\{ 5, 7 \} &\mapsto \{ 2, 7 \}, \\
\{ 6 \} &\mapsto \{ 9 \}.
\end{align*}
The parking function in the middle is the image of the parking function on the left under the action of
$(1, 3, 4)(2, 5, 8)(6, 9, 7) \in \symm_9$.  The parking function on the right is the image of the parking 
function on the left under the action of the distinguished generator of $\ZZ_9$.

It is proven \cite{ARR} that the character 
$\chi: \symm_n \times \ZZ_n \rightarrow \CC$ of the permutation module $\Park^{NC}_{\symm_n}$
is given by
(\ref{hsopchar}, $k = 1$).  A $\symm_n$-equivariant bijection 
$\Park^{NC}_{\symm_n} \xrightarrow{\sim} \Park_n$ is 
given by assigning $(\pi, f) \mapsto (a_1, \dots, a_n)$, where $a_i$ is the minimal element of the 
unique block $B \in \pi$ such that $i \in f(B)$.  From left to right, the three elements of 
$\Park^{NC}_{\symm_9}$ shown in
Figure~\ref{fig:k1park} map 
to the sequences $(2, 5, 2, 1, 2, 2, 5, 1, 6), 
(1, 1, 2, 2, 5, 5, 6, 2, 2), (3, 6, 3, 1, 3, 3, 6, 1, 7) \in \Park_9$.  
Observe that the action of $\ZZ_n$ is transparent on the set $\Park^{NC}_{\symm_n}$ but
not obvious on the set $\Park_n$.

It is clear that the orbits in $\Park^{NC}_{\symm_n} / \symm_n$ are indexed by noncrossing partitions 
$\pi$ of $[n]$, and that the orbit corresponding to a noncrossing partition $\pi$ is 
$\symm_n$-isomorphic to the coset representation ${\bf 1}_{\symm_{\lambda}}^{\symm_n}$, where
the block sizes in $\pi$ are $(\lambda_1 \geq \lambda_2 \geq \dots )$.  
This makes the $\symm_n$-module isomorphism
(\ref{iso}) transparent.

In \cite{ARR}, the module $\Park^{NC}_{\symm_n}$ is generalized to arbitrary reflection groups
$W$ using a $W$-analog of noncrossing partitions due to 
Bessis \cite{Bessis}, Brady and Watt \cite{BradyWatt}, and Reiner \cite{Reiner}.

\subsection{ The noncrossing $k$-parking space in type A }  In this paper we give a Fuss 
generalization $\Park^{NC}_W(k)$ of $\Park^{NC}_W$.  To gain intuition for this construction, 
we present it in type A$_{n-1}$.

 \begin{figure}
\includegraphics[scale=0.5]{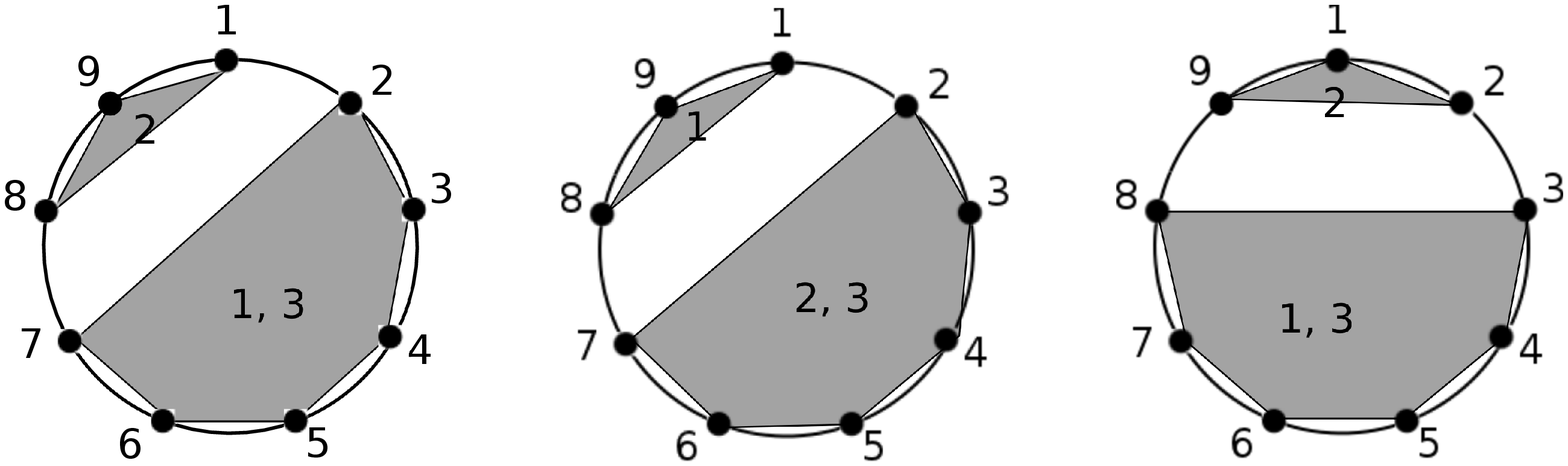}
\caption{Three elements of $\Park^{NC}_{\symm_n}(k)$ for 
$n = k = 3$}
\label{fig:Aaction}
\end{figure}

The elements of $\Park^{NC}_{\symm_n}(k)$ are pairs $(\pi, f)$ where 
$\pi$ is a $k$-divisible noncrossing partition of $[kn]$ and $f: B \mapsto f(B)$ assigns every
block $B \in \pi$ to a subset $f(B) \subseteq [n]$ such that 
\begin{enumerate}
\item for any block $B \in \pi$, we have $|B| = k |f(B)|$, and
\item we have $[n] = \biguplus_{B \in \pi} f(B)$.
\end{enumerate}
Drawing these pairs $(\pi, f)$ on 
$\DD^2$ as before, the group $\symm_n$ acts on the block labels and the cyclic
group $\ZZ_{kh} = \ZZ_{kn}$ acts by clockwise rotation.  

Figure~\ref{fig:Aaction} shows three examples of $\Park^{NC}_{\symm_n}(k)$
for $n = k = 3$.  The pair $(\pi, f)$ is depicted on the the left, where
$\pi = \{1, 8, 9 / 2, 3, 4, 5, 6, 7 \}$ and $f$ is the labeling
\begin{align*}
f: \{1, 8, 9 \} &\mapsto \{2\} \\
\{2,3,4,5,6,7\} &\mapsto \{1,3\}.
\end{align*}
The middle element of Figure~\ref{fig:Aaction} is the image of the left element 
under the action of $(1,2) \in \symm_3$.  The right element of Figure~\ref{fig:Aaction} is
the image of the left element under the action of $\ZZ_9$.

It will be proven in Proposition~\ref{weak-type-a} that 
the character $\chi: \symm_n \times \ZZ_{kn} \rightarrow \CC$ of this permutation representation is
given by (\ref{hsopchar}).  A $\symm_n$-equivariant bijection
$\Park^{NC}_{\symm_n}(k) \xrightarrow{\sim} \Park_n(k)$ can be obtained by sending 
$(\pi, f)$ to $(a_1, \dots, a_n)$, where $a_i$ is the minimal element of the unique block 
$B \in \pi$ such that $i \in f(B)$.  From left to right, the elements of $\Park_n(k)$ corresponding
to the elements of $\Park^{NC}_{\symm_n}(k)$ in Figure~\ref{fig:Aaction} are
$(2,1,2), (1, 2, 2),$ and $(3,1,3)$.

As in the $k = 1$ case, the $\symm_n$-orbits in $\Park^{NC}_{\symm_n} / \symm_n$ are indexed
by $k$-divisible noncrossing partitions $\pi$ of $[kn]$, and the orbit containing $\pi$ affords
the coset representation ${\bf 1}_{\symm_{\lambda}}^{\symm_n}$, where the block sizes 
in $\pi$ are $(k \lambda_1 \geq k \lambda_2 \geq \dots )$.
This makes the isomorphism (\ref{fussiso}) transparent.

The $\ZZ_{kn}$-structure carried by $\Park^{NC}_{\symm_n}(k)$ has
implications related to the cyclic sieving phenomenon (see \cite{RSWCSP}).
Let $\Park^{NC}_{\symm_n}(k)^{\symm_n}$ be the 
$\symm_n$-invariant subspace of $\Park^{NC}_{\symm_n}(k)$.  The space 
 $\Park^{NC}_{\symm_n}(k)^{\symm_n}$ has a basis consisting of orbit sums which are indexed
 by $k$-divisible noncrossing partitions $\pi$ of $[kn]$ and a residual $\ZZ_{kn}$-action which acts 
 on this basis by rotation.  We will use
 an algebraic interpretation of $\Park^{NC}_{\symm_n}(k)$
 to compute the character of 
 $\Park^{NC}_{\symm_n}(k)^{\symm_n}$
 and prove that the triple 
 $(X, \ZZ_{kn}, X(q))$ exhibits the {\em cyclic sieving phenomenon} 
 where $X$ is the set of $k$-divisible noncrossing partitions 
 of $[kn]$, $\ZZ_{kn}$ acts on $X$ by rotation, and 
 $X(q)$ is the {\sf $q$-Fuss-Catalan number} $\Cat^k(n;q) := \frac{1}{[kn+1]_q}{(k+1)n \brack n }_q$.
 Here we use the standard $q$-analog notation
 \begin{align*}
 [n]_q &:= 1 + q + q^2 + \cdots + q^{n-1} = \frac{1-q^n}{1-q}, \\
 [n]!_q &:= [n]_q [n-1]_q \cdots [1]_q, \\
 {n \brack k}_q &:= \frac{[n]!_q}{[k]!_q [n-k]!_q}.
 \end{align*}
 
 This cyclic sieving 
 result will be generalized to all real reflection groups $W$ modulo a conjecture about
 $\Park^{NC}_W(k)$ (Weak Conjecture).  This conjecture, in turn, will be proven
 for all types other than the exceptional types EFH.
 
 The type A version of this cyclic sieving phenomenon 
 was proven (unpublished) by D. White.  The generalization for 
 arbitrary real (in fact, well-generated complex) reflection groups 
 $W$ was conjectured by Armstrong \cite{Armstrong} (see also \cite{BessisR}).  Our work 
 proves Armstrong's conjecture in a uniform way modulo conjectures about $\Park^{NC}_W(k)$
 which are proven in types ABCDI.
 A case-by-case proof of Armstrong's conjecture as 
 it applies to any well-generated complex
 reflection group was given by Krattenthaler and M\"uller 
 \cite{KrattenthalerMuller, KrattenthalerMuller2}.

Our generalization of $\Park^{NC}_{\symm_n}(k)$ to arbitrary reflection groups $W$ will involve 
a Fuss analog of the $W$-noncrossing partitions developed in Armstrong's thesis
\cite{Armstrong}.  

\subsection{ The idea of the algebraic $k$-parking space }  We will attach
a second $W \times \ZZ_{kh}$-module called $\Park^{alg}_W(k)$ to any real reflection group $W$.  
The definition of $\Park^{alg}_W(k)$ is algebraic and involves deforming the ideal $(\Theta)$ in the 
quotient module $\CC[V] / (\Theta)$ 
of Section~\ref{Homogeneous systems of parameters;  arbitrary finite type}  
to get a new inhomogeneous ideal $(\Theta - \xx)$.  The quotient
$\CC[V] / (\Theta - \xx)$ will still carry an action of $W \times \ZZ_{kh}$ whose character is given by
(\ref{hsopchar}).  We conjecture that $\Park^{alg}_W(k)$ is naturally a $W \times \ZZ_{kh}$-permutation
module, and that we have an isomorphism of $W \times \ZZ_{kh}$-sets
$\Park^{NC}_W(k) \cong_{W \times \ZZ_{kh}} \Park^{alg}_W(k)$.  We will prove the strongest form
of our conjecture in rank 1 (and for the `dimension $\leq 1$ parts' of $\Park^{alg}_W(k)$ and
$\Park^{NC}_W(k)$ in arbitrary type), the intermediate form of our conjecture in types BCDI, and
the weak form of our conjecture in type A.

\subsection{Organization}
The remainder of the paper is organized as follows.  In {\bf Section 2} we will present background
material on classical noncrossing and $k$-divisible noncrossing partitions, reflection groups, 
$W$-noncrossing partitions, $W$-noncrossing parking functions, $k$-$W$-noncrossing partitions,
and homogeneous systems of parameters.

In {\bf Section 3} we present the two main constructions of this paper: the 
$k$-$W$-noncrossing parking space $\Park^{NC}_W(k)$ and the 
$k$-$W$-algebraic parking space $\Park^{alg}_W(k)$.  We present the three versions 
(Strong, Intermediate, and Weak)
of
our Main Conjecture which relates these spaces, as well as a cyclic sieving consequence 
of any form of the Main Conjecture.

In {\bf Section 4} we present evidence for the Strong Conjecture as it applies to any reflection
group. 

In {\bf Section 5} we prove the Intermediate Conjecture in dihedral type I.

In {\bf Section 6} we prove the Intermediate Conjecture in hyperoctohedral type BC.

In {\bf Section 7} we prove the Intermediate Conjecture in type D.  

In {\bf Section 8} we prove the Weak Conjecture in type A.  

In {\bf Section 9} we restrict our attention to crystallographic type and present a `nonnesting'
analog $\Park^{NN}_{\Phi}(k)$ of  $\Park^{NC}_W(k)$ for any crystallographic root system
$\Phi$.  This nonnesting analog only carries an action of the Weyl group $W = W(\Phi)$ rather 
than the product $W \times \ZZ_{kh}$, and does not exist outside of crystallographic type.

We close in {\bf Section 10} with open problems.

\section{Background}
\label{Background}

We begin with background related to classical noncrossing partitions,  intersection
lattices, absolute order and $W$-noncrossing partitions, Fuss analogs of $W$-noncrossing partitions, the 
definition of $W$-noncrossing parking functions in \cite{ARR}, 
and homogeneous systems of parameters.  

\subsection{Notation}
Given any poset $P$, a {\sf $k$-multichain} is a sequence of 
$k$ elements $x_1, \dots, x_k$ of $P$ satisfying $x_1 \leq \dots \leq x_k$.

We let $W$ denote an irreducible finite real reflection group and we let
$V$ denote the reflection representation of $W$.  
Except in Section~\ref{Nonnesting analogs}, we extend scalars to $\CC$ and consider 
$V$ to be a complex vector space.
We let $n = \dim(V)$ be the rank of $W$.

We let $S \subset W$ be a choice of simple reflections in $W$ and we let
$T = \bigcup_{w \in W} w S w^{-1}$ be the set of reflections in $W$.  We let $c \in W$ denote 
a choice of (standard) Coxeter element.  That is $c$ is an element of $W$
of the form $c = s_1 \dots s_n$, where
$S = \{ s_1, \dots, s_n \}$.  Any two Coxeter elements of $W$ are conjugate.  We denote
by $h$ their common order, also called the Coxeter number of $W$.

We denote by $k \in \ZZ_{> 0}$ a Fuss parameter.  We let $g$ be a distinguished
choice of generator for the cyclic group $\ZZ_{kh}$ and we let 
$\omega$ be a primitive $kh^{th}$ root of unity.

\subsection{Set partitions}
Set partitions of $[n]$ are partially ordered by {\sf refinement}: $\pi \leq \pi'$ if each block of 
$\pi'$ is a union of blocks of $\pi$.  The poset of set partitions of $[n]$ has the structure 
of a geometric lattice with minimal element $\hat{0} = \{1/2/ \dots/ n \}$ and maximal element
$\hat{1} = \{1,2,\dots,n \}$.

A partition $\pi$ of the set $[n]$ is called {\sf noncrossing} if whenever there are 
indices $1 \leq a < b < c < d \leq n$ such that $a \sim c$ and $b \sim d$ in $\pi$, we 
necessarily have $a \sim b \sim c \sim d$ in $\pi$.  Equivalently, if one labels the boundary
of the disc $\DD^2$ clockwise with the points $1, 2, \dots, n$, a partition $\pi$ of $[n]$
is noncrossing if and only if the convex hulls of its blocks do not intersect.
The left of Figure~\ref{fig:kreweras} shows the noncrossing partition
$\{ 1, 2, 5 / 3, 4 / 6 \}$ of $[6]$ drawn on $\DD^2$.  The right of Figure~\ref{fig:kreweras}
shows the noncrossing partition $\{1, 6 / 2 / 3, 5 / 4 \}$ drawn on $\DD^2$.
The cyclic group $\ZZ_n$ acts on noncrossing partitions of $[n]$ by clockwise rotation.

A set partition $\pi$ is called {\sf $k$-divisible} if each block of $\pi$ has size divisible by
$k$.  Refinement order on set partitions restricts to 
give a poset structures on the sets of noncrossing partitions and
$k$-divisible noncrossing partitions.

The poset of noncrossing partitions of $[n]$ has the structure of a complemented lattice.
An explicit complementation $K$ was introduced by Kreweras \cite{Kreweras}.
Label the boundary of the disc $\DD^2$ clockwise with $1', 1, 2', 2, \dots, n', n$.  Given
a noncrossing partition $\pi$ of $[n]$, draw the convex hulls of the blocks of $\pi$
on $\DD^2$ using the unprimed vertices.  The Kreweras complement $K(\pi)$ is the 
unique maximal partition of $[n]$ under refinement obtained using the primed 
vertices whose blocks do not intersect  the blocks of $\pi$ in $\DD^2$.  The partition
$K(\pi)$ is automatically noncrossing.

Figure~\ref{fig:kreweras} shows an example of Kreweras complementation when $n = 6$.
The Kreweras complement of $\{ 1,2,5 / 3,4 / 6\}$ is $\{1,6 / 2 / 3,5 / 4\}$.  

\begin{figure}
\includegraphics[scale=0.6]{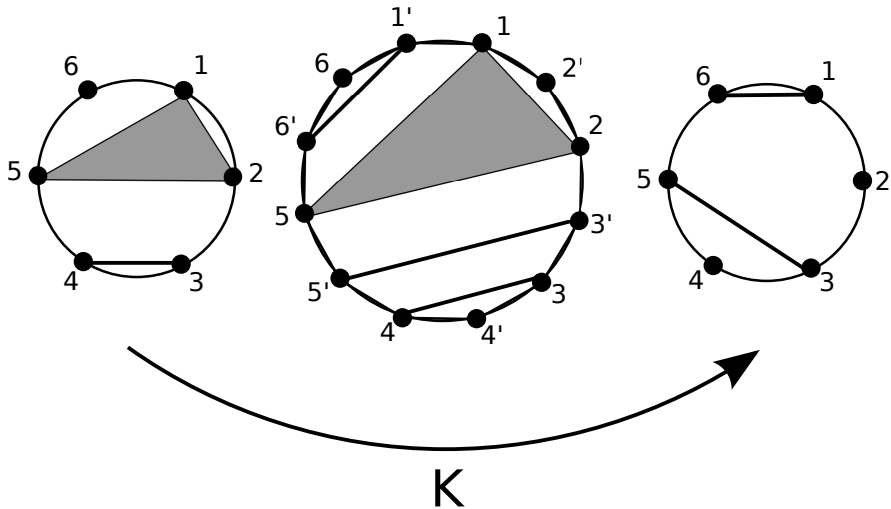}
\caption{Kreweras complementation}
\label{fig:kreweras}
\end{figure}

There is a natural injection from the set of noncrossing partitions of $[n]$ 
to $\symm_n$.  If $\pi$ is a noncrossing partition of $[n]$,
the corresponding permutation $\Omega(\pi) \in \symm_n$ of $[n]$ 
has cycles given by the blocks of $\pi$, read 
in clockwise order on $\DD^2$.  In particular, the maximal noncrossing partition
has permutation $\Omega( \{ 1, 2, \dots, n \} ) = (1,2,\dots,n)$ and the 
minimal noncrossing partition has permutation
$\Omega( \{ 1 / 2 / \dots / n \} ) = (1) (2) \dots (n) = ()$.  
The permutation of the partition $\pi$ on the left of Figure~\ref{fig:kreweras} is
$\Omega(\pi) = (1,2,5)(3,4)(6) \in \symm_6$.
Let $\Pi$ denote the inverse of the injection $\Omega$.  For example, we have
$\Pi ( (1,2,5)(3,4)(6) ) = \{1, 2, 5 / 3, 4 / 6 \}$.

There is a bijective correspondence between $k$-divisible noncrossing partitions of 
$[kn]$ and  $k$-multichains of noncrossing partitions of $[n]$ under refinement.  
To explain this bijection, we will need the notion of a shuffle of partitions.
Given two sets of integers $A$ and $B$ satisfying $|A| = |B|$ and a partition $\pi$ of
$A$, let $\pi(B)$ be the unique partition of $B$ which is order isomorphic to $\pi$.
For example, if $\pi = \{1,3 / 2,4 /5 \}$ (so that $A = [5]$) and $B = \{1, 5, 7, 8, 9\}$, then
$\pi(B) = \{1,7 / 5,8 / 9\}$

The {\sf shuffle} of a $k$-tuple $(\pi_1, \dots, \pi_k)$ of partitions of $[n]$ is the partition
 of $[kn]$ generated by $\pi_i( \{i, i+k, \dots, i+(n-1)k \} )$ for $1 \leq i \leq k$.  Denote this
 partition by $\shuffle(\pi_1, \dots, \pi_k)$. 
 In general, the shuffle of a pair of noncrossing partitions need not be noncrossing.  
For example, we have $\shuffle( \{1, 2 \}, \{1 , 2 \} ) = \{ 1, 3 / 2, 4\}$.  However, given
a $k$-multichain $(\pi_1 \leq \dots \leq \pi_k)$ of noncrossing partitions of $[n]$ under
refinement order,
we can obtain a noncrossing partition of $[kn]$ using shuffles. 

To make this precise, we abuse notation by letting $\Omega$ act on tuples 
$(\pi_1, \dots, \pi_k)$ of set partitions by 
$\Omega(\pi_1, \dots, \pi_k) := (\Omega(\pi_1), \dots, \Omega(\pi_k))$.  Similarly, we let
$\Pi$ act on tuples of permutations by
$\Pi(w_1, \dots, w_k) := (\Pi(w_1), \dots, \Pi(w_k))$, assuming $\Pi(w_i)$ is defined for all $i$.
Finally, we define a {\sf boundary map} $\delta$ on sequences 
$(w_1, \dots, w_k)$ of permutations by 
\begin{equation*}
\delta(w_1, \dots, w_k) = (w_1^{-1} w_2, w_2^{-1} w_3, \dots, w_{n-1}^{-1} w_n, w_n^{-1}c),
\end{equation*}
where $c = (1,2,\dots,n) \in \symm_n$.  The following lemma collects results from
\cite[Chapter 4]{Armstrong}.

\begin{lemma} \cite{Armstrong}
\label{classical-noncrossing}
Let $(\pi_1 \leq \dots \leq \pi_k)$ be a $k$-multichain of noncrossing 
partitions of $[n]$ in refinement order.
Then the partition 
\begin{equation*}
\shuffle \circ \Pi \circ \delta \circ \Omega (\pi_1 , \dots , \pi_k)
\end{equation*}
is noncrossing and the Kreweras complement
of this partition is a $k$-divisible noncrossing partition of $[kn]$.  This procedure defines 
a bijection $\nabla := K \circ \shuffle \circ \Pi \circ \delta \circ \Omega $ from the set of $k$-element multichains of noncrossing partitions 
of $[n]$ to the set of $k$-divisible noncrossing partitions of $[kn]$.  

The partition $\pi := \nabla(\pi_1 , \dots , \pi_k)$ of $[kn]$ has the following properties:
\begin{itemize}
\item for all indices $1 \leq i < j \leq n$, we have that $i \sim j$ in $\pi_1$ if and only if
$(k-1)i + 1 \sim (k-1)j + 1$ in $\pi$, and
\item given a block $B$ of $\pi_1$ of size $b$, the block of $\pi$ containing
$\{ (k-1)i + 1 \,:\, i \in B \}$ has size $kb$.
\end{itemize}
\end{lemma} 

The first item in Lemma~\ref{classical-noncrossing} states that $\pi_1$ is 
order isomorphic to the restriction of the image $\nabla(\pi_1 , \dots , \pi_k)$ to the set 
$\{1, k + 1, 2k + 1, \dots, (n-1)k + 1 \}$, so that $\pi_1$ can be easily recovered from
$\nabla(\pi_1 , \dots , \pi_k)$.  This fact will be useful for obtaining a visual interpretation
of type A $k$-parking functions.  Recovering $\pi_i$ from 
$\nabla(\pi_1 , \dots , \pi_k)$ is more complicated for $i > 1$.

Figure~\ref{fig:Afuss} shows an example of the map $\nabla$ when
$n = 4$ and $k = 3$ (if one temporarily ignores the numbers inside the blocks).  
We start with the $3$-multichain
$ \{1,4 / 2 / 3 \} \leq \{ 1, 3, 4 / 2 \} \leq \{1,3,4 / 2 \}$ of noncrossing partitions of 
$[4]$.  Applying the map $\Omega$ yields the triple
$((1,4), (1,3,4), (1,3,4))$ of permutations in $\symm_4$.  Applying the boundary 
map $\delta$ yields
$((1,3),(),(1,2))$.  Applying $\shuffle \circ \Pi$ yields the noncrossing partition
$\{1,7 / 2,5 / 3 / 4 / 6/ 8/ 9/ 10/ 11/12\}$ of $[12]$.  Applying the
Kreweras map $K$ yields the $3$-divisible noncrossing partition
$\{1,8,9,10,11,12 / 2,6,7 / 3,4,5 \}$ of $[12]$.  Observe that the restriction of 
this latter partition to $\{1,4,7,10\}$ is $\{1,10/4/7\}$, which is order isomorphic to
$\{1,4/2/3\}$.

\subsection{Reflection groups and intersection lattices}  
A {\sf standard parabolic subgroup} of $W$ is a subgroup $W_I$ generated by any 
subset $I \subseteq S$.  A {\sf parabolic subgroup} of $W$ is any $W$-conjugate of a 
standard parabolic subgroup.

Given $t \in T$, let $H_t \subset V$
be the hyperplane through which $t$ reflects.  The {\sf Coxeter arrangement}
$\Cox(W)$ is the hyperplane arrangement in $V$ with hyperplanes 
$\{ H_t \,:\, t \in T \}$.  The {\sf intersection lattice} $\LLL$ of $W$ is the collection
of all intersections of the hyperplanes in $\Cox(W)$, partially ordered
by {\bf reverse} inclusion ($X \leq_{\LLL} Y \Leftrightarrow X \supseteq Y$).  
Subspaces $X \in \LLL$ are called {\sf flats}. 
The lattice $\LLL$ carries a natural action of the group $W$ given by $w.X = wX$.
 It can be shown
that $V^w \in \LLL$ for any $w \in W$.  The map $W \rightarrow \LLL$ defined by $w \mapsto V^w$
is surjective, but is almost never injective.

Given any subset $X \subseteq V$, let $W_X := \{ w \in W \,:\, w.v = v \text{ for all } v \in X \}$ be
the corresponding isotropy subgroup of $W$.  Similarly, for any $U \subseteq W$, let 
$V^U = \bigcap_{u \in U} V^u$ be the fixed space of $U$.  Parabolic subgroups and flats are 
related as follows.

\begin{theorem} \cite{BarceloIhrig} (Galois correspondence)  
\label{galois}
The map $X \mapsto W_X$ defines an order-preserving bijection between the intersection
lattice $\LLL$ and the set of parabolic subgroups of $W$ ordered by inclusion.  The
inverse to this map is given by $U \mapsto V^U$.
\end{theorem}

\begin{example}
\label{ex:typeA1}
Suppose that $W = \symm_n$ has type A$_{n-1}$.  Then the reflection representation $V$ is
the subspace of $\CC^n$ which on which the 
equality $x_1 + \cdots + x_n = 0$ is satisfied, where the $x_i$ are
the standard coordinate functions.  The group $W = \symm_n$ acts on $V$ by coordinate
permutation.
The set of reflections is $T = \{ (i,j) \,:\, 1 \leq i < j \leq n \} \subset W$, where $(i,j)$ reflects 
across the hyperplane $x_i - x_j = 0$.  The set $S = \{ s_1, \dots, s_{n-1} \}$ is a choice 
of simple reflections, where $s_i = (i, i+1)$.  The rank is $n-1$, a choice
of Coxeter element is $c = (1,2, \dots, n)$, and the Coxeter number
$h = n$.

The Coxeter arrangement $\Cox(W)$ consists of the hyperplanes 
$\{ x_i - x_j = 0 \,:\, 1 \leq i < j \leq n \}$ in $V$.  The intersection lattice $\LLL$ can be identified
with the lattice of set partitions of $[n]$ by sending a flat $X \in \LLL$ to the partition of $[n]$
defined by $i \sim j$ if and only if the coordinate equality $x_i = x_j$ holds on $X$.  
For example, if $n = 6$ the partition $\{ 1,3 / 2, 4, 5 / 6 \}$ corresponds to the flat defined by the
equalities $x_1 = x_3$ and $x_2 = x_4 = x_5$.
The isotropy subgroup $W_{X(\pi)}$ of the flat
$X(\pi) \in \LLL$ associated to a partition $\pi$ of $[n]$ is the subgroup of $\symm_n$ which 
permutes  elements within the blocks of $\pi$.  For example, 
$W_{X(\{1,3/2,4,5/6\})} = \symm_{\{1,3\}} \times \symm_{\{2,4,5\}} \times \symm_{\{6\}}$.
\end{example}

\subsection{Reflection length and $W$-noncrossing partitions}  For any $w \in W$, a 
{\sf $T$-reduced expression} for $w$ is a word $t_1 \dots t_r$ in $T$ of minimal length
such that $w = t_1 \dots t_r$.  The {\sf reflection length} of $w$, denoted
$\ell_T(w)$, is the length $r$ of any $T$-reduced expression for $w$.  (This 
is {\em not} the usual Coxeter length function, in which one would consider words in
$S$ instead of $T$.)  We have that $\ell_T(w) = n - \dim(V^w)$ for any $w \in W$.

Reflection length can be used to define a partial order on $W$.  {\sf Absolute order}
$\leq_T$ on $W$ is defined by 
\begin{equation}
u \leq_T v \Leftrightarrow \ell_T(v) = \ell_T(u) + \ell_T(u^{-1}v).
\end{equation}
Equivalently, we have that $u \leq_T v$ if and only if $u$ occurs as a prefix in some $T$-reduced
expression for $v$.  Also, 
if there exists $w \in W$ such that $u, v \leq_T w$,
we have that $u \leq_T v$ if and only if $V^u \supseteq V^v$.
The poset $(W, \leq_T)$ is graded with rank function given 
by reflection length.  The identity element $1 \in W$ is the unique minimal element of absolute order.
The Coxeter elements of $W$ form a subset of the maximal elements of $W$.
Absolute order was used by Brady and Watt \cite{BradyWatt} and Bessis \cite{Bessis}
to define a generalization of the  noncrossing partitions to any reflection group $W$.

\begin{defn} \cite{BradyWatt}, \cite{Bessis}
\label{definition-noncrossing-partitions}
Let $c \in W$ be a Coxeter element and define the poset $NC(W)$ of
{\sf $W$-noncrossing partitions} to be the interval $[1, c]_T$ in absolute order.  Brady and 
Watt proved that we have an injection
\begin{align*}
NC(W) &\hookrightarrow \LLL \\
w & \mapsto V^w.
\end{align*}
A flat $X \in \LLL$ will be called {\sf noncrossing} if it is in the image of this injection.
\end{defn}

While the definition of $NC(W)$ given above may appear to depend on the choice
of Coxeter element $c$, this dependence is superficial.  If $c, c' \in W$ are Coxeter elements,
there exists $w \in W$ such that $c = w c' w^{-1}$.  It can be shown that the map 
$v \mapsto w v w^{-1}$ defines a poset isomorphism $[1, c']_T \rightarrow [1,c]_T$.  

Recall that $W$ acts on $\LLL$ be multiplication.
The following 
lemma states that any orbit in $\LLL / W$ contains at least one 
noncrossing flat.  

\begin{lemma} \cite[Lemma 5.2]{ARR}
\label{conjugate-to-noncrossing}
Every flat is $W$-conjugate to a noncrossing flat.
\end{lemma}

The map $w \mapsto c w c^{-1}$ of conjugation by $c$ gives an automorphism of 
the set $NC(W) = [1, c]_T$.  On the level of noncrossing flats this action translates to
$X \mapsto c X$.  We will need the following technical result about one-dimensional 
flats.  

\begin{lemma} \cite[Lemma 5.3]{ARR}
\label{W-and-C-orbits}
Let $X \in \LLL$ be a one-dimensional noncrossing flat.  The set of noncrossing 
flats in the $W$-orbit of $X$ is precisely the $C$-orbit of $X$, where 
$C$ is the cyclic subgroup of $W$ generated by $c$.
\end{lemma}

The proofs of Lemmas~\ref{conjugate-to-noncrossing} and \ref{W-and-C-orbits} in
\cite{ARR} are both uniform.

\begin{example}
\label{ex:typeA2}
Suppose $W = \symm_n$ has type A$_{n-1}$.  The reflection length
$\ell_T(w)$ of a permutation $w \in W$ equals $n - r(w)$, where $r(w)$ is the number
of cycles in the cycle decomposition of $w$.  The set of maximal elements of the poset
$(W, \leq_T)$ consists precisely of the set of Coxeter elements (i.e., long cycles) in $W$, although
in general type the set of Coxeter elements in $W$ typically forms a proper subset of 
the maximal elements in reflection order.  Identifying the intersection lattice
$\LLL$ with the set partitions of $[n]$ ordered by refinement, the map $W \rightarrow \LLL$
sending $w$ to $V^w$ is a homomorphism of posets $(W, \leq_T) \rightarrow (\LLL, \leq_{\LLL})$.

We choose the Coxeter element $c = (1,2,\dots,n) \in W$.  It can be shown that a permutation
$w \in \symm_n$ is noncrossing with respect to $c$ if and only if each cycle in $w$ 
can be written as an increasing subsequence of $1 , 2 , \dots , n$.  The  group
elements in $NC(W)$ map to flats in $\LLL$ which correspond to the classical noncrossing
partitions of $[n]$.
\end{example}

\subsection{$W$-noncrossing parking functions}
We review the combinatorial generalization in \cite{ARR} of parking functions to arbitrary
type.  

\begin{defn}
\label{k1ncpark}
A {\sf $W$-noncrossing parking function} is an equivalence class in 
\begin{equation}
\{ (w, X) \in W \times \LLL \,:\, \text{$X$ is noncrossing} \} / \sim,
\end{equation}
where $(w, X) \sim (w', X')$ if $X = X'$ and $w W_X = w' W_X$.  We denote by
$\Park^{NC}_W$ the set of $W$-noncrossing parking functions.
Equivalence classes in $\Park^{NC}_W$ are denoted with square brackets, e.g., 
$[w, X]$.
\end{defn} 

The set $\Park^{NC}_W$ an action of $W \times \ZZ_h$ defined by
\begin{equation}
(v,g^d).[w,X] := [vwc^{-d}, c^d X],
\end{equation}
where we use the fact that $c X$ is a noncrossing flat whenever $X$ is.  To check that
this action is well defined, it is useful to note that 
$W_{c^d X} = c^d W_X c^{-d}$.

\begin{example}
\label{ex:typeA3}
Suppose $W = \symm_n$.  Given $w \in W$ and a noncrossing partition $\pi$ with 
corresponding flat $X \in \LLL$, we represent the $W$-noncrossing parking function
$[w, X]$ by drawing the noncrossing partition $\pi$ on the disc $\DD^2$ and
labeling each block $B$ of $\pi$ with the subset $w(B)$ of $[n]$.  

This identifies the set of $W$-noncrossing parking functions in type A$_{n-1}$ with
pairs $(\pi, f)$ where $\pi$ is a noncrossing partition of $[n]$ and $f$ assigns each block
$B$ of $\pi$ to a subset $f(B)$ of $[n]$ such that $|f(B)| = |B|$ and 
$[n] = \biguplus_{B \in \pi} f(B)$.  The group $W = \symm_n$ acts on block labels and
the group $\ZZ_h = \ZZ_n$ acts by $n$-fold rotation.
In Section~\ref{Introduction} we explained how to get a $\symm_n$-equivariant bijection
between the set of pairs $(\pi, f)$ and the classical parking functions 
$\Park_n$.
\end{example}

\subsection{$k$-$W$-noncrossing flats and partitions}  
In his thesis, 
Armstrong \cite{Armstrong} gave a type-uniform generalization of $k$-divisible 
noncrossing partitions
using $k$-element multichains of noncrossing group elements.

\begin{defn} \cite{Armstrong}
\label{definition-fuss-noncrossing-partitions}
The set $NC^k(W)$ 
of {\sf $k$-$W$-noncrossing partitions} consists of all $k$-multichains
$(w_1 \leq_T \dots \leq_T w_k)$ in the poset $NC(W)$.

There is an injection
\begin{align*}
NC^k(W) &\hookrightarrow \LLL^k \\
(w_1 \leq_T \dots \leq_T w_k) &\mapsto (V^{w_1} \leq_{\LLL} \dots \leq_{\LLL} V^{w_k})
\end{align*}
coming from the injection in Definition~\ref{definition-noncrossing-partitions}.   A  $k$-multichain
$(X_1 \leq_{\LLL} \dots \leq_{\LLL} X_k)$ of flats in $\LLL$ will be called a {\sf noncrossing $k$-flat}
if it is in the image of this injection.
\end{defn}

To avoid notational clutter, we will typically omit the subscripts $T$ or $\LLL$ on $\leq$
when writing an element of $NC^k(W)$ or a noncrossing $k$-flat.

The set $NC^k(W)$ can also be interpreted in terms of  certain factorizations of 
the underlying Coxeter element $c$.
We call a sequence $(w_0, w_1, \dots, w_k) \in W^{k+1}$ an {\sf $\ell_T$-additive
factorization of $c$} if $w_0 w_1 \dots w_k = c$ and $\sum_{i = 0}^k \ell_T(w_i) = \ell_T(c)$.
We let $NC_k(W)$ denote the set of $\ell_T$-additive factorizations $(w_0, w_1, \dots, w_k)$
which have length $k+1$.
The maps $\partial$ and $\int$ given below give mutually inverse bijections between
$NC^k(W)$ and $NC_k(W)$.
\begin{align*}
&\partial: &NC^k(W) &\rightarrow NC_k(W) \\
&\partial: &(w_1 \leq \dots \leq w_k) &\mapsto 
(w_1, w_1^{-1} w_2, \dots, w_{k-1}^{-1} w_k, w_k^{-1} c)\\
&\text{$\int$}: &NC_k(W) &\rightarrow NC^k(W) \\
&\text{$\int$}: &(w_0, w_1, \dots, w_k) &\mapsto (w_0 \leq w_0 w_1 \leq \dots \leq w_0 w_1 \dots w_{k-1}).
\end{align*}

The cyclic group $\ZZ_{kh}$ acts on 
$NC_k(W)$ via
\begin{equation*}
g.(w_0, w_1 \dots, w_k) = (v, c w_k c^{-1}, w_1, w_2, \dots, w_{k-1}), 
\end{equation*}
where
$v = (c w_k c^{-1}) w_0 (c w_k c^{-1})^{-1}$ \cite{Armstrong}.  The map $\int$ transfers 
this action to an action of $\ZZ_{kh}$ on $NC^k(W)$.  In turn, the map in 
Definition~\ref{definition-fuss-noncrossing-partitions} transfers this action to an action on
noncrossing $k$-flats.  The following lemma will help us get a well-defined action of
$\ZZ_{kh}$ on noncrossing parking functions.

\begin{lemma}
\label{lemma-first-component}
Suppose that $(X_1 \leq \dots \leq X_k)$ is a noncrossing $k$-flat and that 
$(Y_1 \leq \dots \leq Y_k) = g.(X_1 \leq \dots \leq X_k)$ is its image under the 
distinguished generator $g$ of $\ZZ_{kh}$.  Then $Y_1 = c w_k^{-1} X_1$, where
$w_k \in NC(W)$ is such that $X_k = V^{w_k}$.  
\end{lemma}

\begin{proof}
For $1 \leq i \leq k$, let $w_i \in NC(W)$ be the unique
noncrossing group element such that $X_i = V^{w_i}$.  Then the noncrossing $k$-flat
$(X_1 \leq \dots \leq X_k)$ corresponds to the element
$(w_1 \leq w_2 \leq  \dots \leq w_k) \in NC^k(W)$ under the map in 
Definition~\ref{definition-fuss-noncrossing-partitions}.  The image of this element
of $NC^k(W)$ under the map $\partial$ is
$(w_1, w_1^{-1} w_2, \dots, w_{k-1}^{-1} w_k, w_k^{-1} c)$.  The image of this latter
element of $NC_k(W)$ under $g$ is 
$( (c w_k^{-1} c  c^{-1} ) w_1 (c w_k^{-1} c c^{-1})^{-1}, c w_k^{-1} c c^{-1}, w_1^{-1} w_2, 
w_2^{-1} w_3, \dots, w_{k-1}^{-1} w_k)$.  The image $(v_1 \leq \dots \leq v_k)$ of 
this under $\int$ has first component 
$v_1 = (c w_k^{-1}) w_1 (c w_k^{-1})^{-1}$.  The fixed space 
$Y_1 = V^{v_1}$ is therefore equal to $c w_k^{-1} V^{w_1} = c w_k^{-1} X_1$.
\end{proof}

While the action of $g$ on the set of noncrossing $k$-flats has a somewhat involved
definition, the $k^{th}$ power of $g$ acts on noncrossing $k$-flats by translation by $c$:
\begin{equation}
\label{kthpower}
g^k.(X_1 \leq \dots \leq X_k) = (c X_1 \leq \dots \leq c X_k).
\end{equation}
To verify Equation~\ref{kthpower}, one can check that $g^k$ acts by 
componentwise conjugation by $c$ on $NC_k(W)$.

\begin{example}
\label{ex:typeA4}
Suppose $W = \symm_n$ has type A$_{n-1}$.  Then we can identify noncrossing $k$-flats
$(X_1 \leq \dots \leq X_k)$ with $k$-multichains of noncrossing partitions of $[n]$
under refinement order.  The map $\nabla$ of Lemma~\ref{classical-noncrossing}
 bijects these with $k$-divisible noncrossing partitions
of $[kn]$.  Embedding these partitions in the disc $\DD^2$, the group $\ZZ_{kh} = \ZZ_{kn}$ acts
by $kn$-fold rotation.  
\end{example}

\section{Definitions and Main Conjecture}
\label{Definitions and Main Conjecture}

\subsection{The $k$-$W$-noncrossing parking space} 
We present our Fuss analog of Definition~\ref{k1ncpark}.

\begin{defn}
\label{noncrossing-definition}
Let $W$ be a reflection group and let $k \geq 1$.  A {\sf $k$-$W$-noncrossing parking 
function} is an equivalence class in
\begin{equation}
\{ (w, X_1 \leq \dots \leq X_k) \,:\, \text{$w \in W$, $X_1 \leq \dots \leq X_k$ a
noncrossing $k$-flat} \} / \sim,
\end{equation}
where $(w, X_1 \leq \dots \leq X_k) \sim (w', X_1' \leq \dots \leq X_k')$
 if and only if $X_i = X_i'$ for all $i$ and $w W_{X_1} = w' W_{X_1}$.
 
The set of $k$-$W$-noncrossing parking functions is denoted $\Park^{NC}_W(k)$.
\end{defn} 

If $X_1 \leq \dots \leq X_k$ is any multichain of flats in $\LLL$, we have a
corresponding multichain of isotropy subgroups 
$W_{X_1} \leq \dots \leq W_{X_k}$ of $W$.  It follows that the equivalence relation
$\sim$ in Definition~\ref{noncrossing-definition} could also have been expressed as 
$(w, X_1 \leq \dots \leq X_k) \sim (w', X_1' \leq \dots \leq X_k')$ 
if and only if $X_i = X_i'$ and $w W_{X_i} = w' W_{X_i}$ for all $i$.  We will use square brackets
to denote equivalence classes, so that the $k$-$W$-noncrossing parking function
containing $(w, X_1 \leq \dots \leq X_k)$ is written
$[w, X_1 \leq \dots \leq X_k]$.  We may also use exponential notation to denote multichains,
e.g. $X^3Y^2 = X \leq X \leq X \leq Y \leq Y$.

We endow the set $\Park^{NC}_W(k)$ with an action of $W \times \ZZ_{kh}$.  The action 
of $W$ is straightforward:  given $v \in W$ and $[w, X_1 \leq \dots \leq X_k]$, set
\begin{equation}
v.[w, X_1 \leq \dots \leq X_k] := [vw, X_1 \leq \dots \leq X_k].
\end{equation}
For a fixed noncrossing $k$-flat $X_1 \leq \dots \leq X_k$, the action of 
$W$ on the set of $k$-$W$-noncrossing parking functions of the form 
$[w, X_1 \leq \dots \leq X_k]$ is isomorphic to the coset representation
$W / W_{X_1} = {\bf 1}_{W_{X_1}}^W$ of $W$.  It follows that
\begin{equation}
\Park^{NC}_W(k) \cong_W \bigoplus_{X} m_X {\bf 1}_{W_X}^W, 
\end{equation}
where the direct sum is over all noncrossing flats $X \in \LLL$ and the multiplicity $m_X$ equals
the number of noncrossing $k$-flats $X_1 \leq \dots \leq X_k$ with $X_1 = X$.

The action of $\ZZ_{kh}$ on $\Park^{NC}_W(k)$ is more involved.  
Section 2.4 gives an action of $g$ on the set of 
noncrossing $k$-flats, denoted $X_1 \leq \dots \leq X_k \mapsto g.(X_1 \leq \dots \leq X_k)$.
Moreover, if $X_1 \leq \dots \leq X_k$ is a noncrossing $k$-flat, Section 2.3 implies that 
there exists a unique noncrossing group element $u_k \in W$ such that $X_k = V^{u_k}$.
We define the action of $g$ on the element $[w, X_1 \leq \dots \leq X_k]$ of $\Park^{NC}_W(k)$
by
\begin{equation}
g.[w, X_1 \leq \dots \leq X_k] := [w u_k c^{-1}, g.(X_1 \leq \dots \leq X_k)].
\end{equation}

\begin{proposition}
\label{well-defined}
The above action of $\ZZ_{kh}$ on $\Park^{NC}_W(k)$ is well defined and commutes
with the action of $W$, giving $\Park^{NC}_W(k)$ the structure of a 
$W \times \ZZ_{kh}$-module.
\end{proposition}
\begin{proof}
We first check that $g.[w, X_1 \leq \dots \leq X_k]$ does not depend on
the representative of $[w, X_1 \leq \dots \leq X_k]$.
Suppose that $[w, X_1 \leq \dots \leq X_k] = [w', X_1 \leq \dots \leq X_k]$ for 
$w, w' \in W$.  Then we have that $w W_{X_1} = w' W_{X_1}$.  If we write
$X_k = V^{u_k}$ for $u_k \in W$ noncrossing, Lemma~\ref{lemma-first-component} implies
that $Y_1 = c u_k^{-1} X_1$, where 
$g.(X_1 \leq \dots \leq X_k) = (Y_1 \leq \dots \leq Y_k)$.  It follows that 
$W_{Y_1} = c u_k^{-1} W_{X_1} u_k c^{-1}$, so that 
$w u_k c^{-1} W_{Y_1} = w' u_k c^{-1} W_{Y_1}$.  This proves that the element
$[wu_k c^{-1}, X_1 \leq \dots \leq X_k] \in \Park^{NC}_W(k)$ does not depend on the representative 
of the equivalence class
$[w, X_1 \leq \dots \leq X_k]$.  So, the element $g.[w, X_1 \leq \dots \leq X_k] \in \Park^{NC}_W(k)$ 
is well defined.

We show next that $g^{kh}$ fixes every element of $\Park^{NC}_W(k)$.  
Fix an element $[w, X_1 \leq \dots \leq X_k] \in \Park^{NC}_W(k)$.
We will prove
the stronger statement that 
\begin{equation}
g^k . [w, X_1 \leq \dots \leq X_k] = [w c^{-1}, c X_1 \leq \dots \leq c X_k].
\end{equation}
Since $c$ has order $h$, this implies that $g^{kh}$ fixes every element of $\Park^{NC}_W(k)$, 
so that $g$ induces an action of $\ZZ_{kh}$ on $\Park^{NC}_W(k)$.  To reduce clutter, 
we use left exponential notation for  conjugation within $W$, viz. $\leftexp{a}{b} = a b a^{-1}$.

For $1 \leq i \leq k$, let $u_i \in W$ be the unique noncrossing group element so that
$X_i = V^{u_i}$.  Also, let $(v_0, v_1, \dots, v_k)$ be the image of 
$(u_1, \dots, u_k)$ under $\partial$, so that $v_0 v_1 \cdots v_k$ is an $\ell_T$-reduced
factorization of $c$.  To make our argument more transparent, we will replace the
noncrossing $k$-flats in the second position of parking functions in $\Park^{NC}_W(k)$ with
elements of $NC_k(W)$.  We have the following chain of 
equalities:

\begin{align*}
g^k . [w, (v_0, v_1, \dots, v_k)] &= 
g^{k-1} . [w v_0 v_1 \cdots v_{k-1} c^{-1}, (\leftexp{c v_k c^{-1}}{v_0}, \leftexp{c}{v_k}, v_1, v_2, \dots,
v_{k-1})]  \\
&= g^{k-2} . [w v_0 v_1 \cdots v_{k-1} c^{-1} (\leftexp{c v_k c^{-1}}{v_0}) (\leftexp{c}{v_k})
v_1 v_2 \cdots v_{k-2} c^{-1}, g^2.(v_0, \dots, v_k)] \\
&= g^{k-2} . [w v_0 v_1 \cdots v_{k-2} c^{-1}, g^2.(v_0, \dots, v_k)] \\
& \vdots \\
&= [w v_0 c^{-1}, g^k.(v_0, \dots, v_k)] \\
&= [w v_0 c^{-1}, (\leftexp{c}{v_0}, \dots, \leftexp{c}{v_k})].
\end{align*}

The first and second equalities follow from the definition of the action of $g$.  The third equality
is a computation within the group $W$ and follows from the fact that $v_0 v_1 \cdots v_k = c$.
In the equalities contained in the ellipses, each successive action of $g$ replaces 
the group element of the form $w v_0 v_1 \cdots v_i c^{-1}$ with $w v_0 v_1 \cdots v_{i-1} c^{-1}$.
The final equality is Equation~\ref{kthpower} on the level of $NC_k(W)$.

It remains to show that 
$[w v_0 c^{-1}, c X_1 \leq \dots \leq c X_k] = [w  c^{-1}, c X_1 \leq \dots \leq c X_k]$.
This is equivalent to the coset equality
$w v_0 c^{-1} W_{c X_1} = w  c^{-1} W_{c X_1}$.  Using the fact that 
$W_{c X_1} = c W_{X_1} c^{-1}$, we are reduced to proving the equality
$v_0 W_{X_1} =  W_{X_1}$.  But this last equality is trivial because 
$X_1 = V^{u_1}$ and $u_i = v_0 v_1 \cdots v_{i-1}$ for all $i$, so $u_1 = v_0$. 
This completes the proof of the formula 
$g^k . [w, X_1 \leq \dots \leq X_k] = [w c^{-1}, c X_1 \leq \dots \leq c X_k]$, and the 
fact that $g$ induces an action of $\ZZ_{kh}$ on $\Park^{NC}_W(k)$.

The only statement left to prove is that the actions of $W$ and $\ZZ_{kh}$ on $\Park^{NC}_W(k)$
commute.  This is obvious from the definitions.
\end{proof}

 \begin{figure}
\includegraphics[scale=0.6]{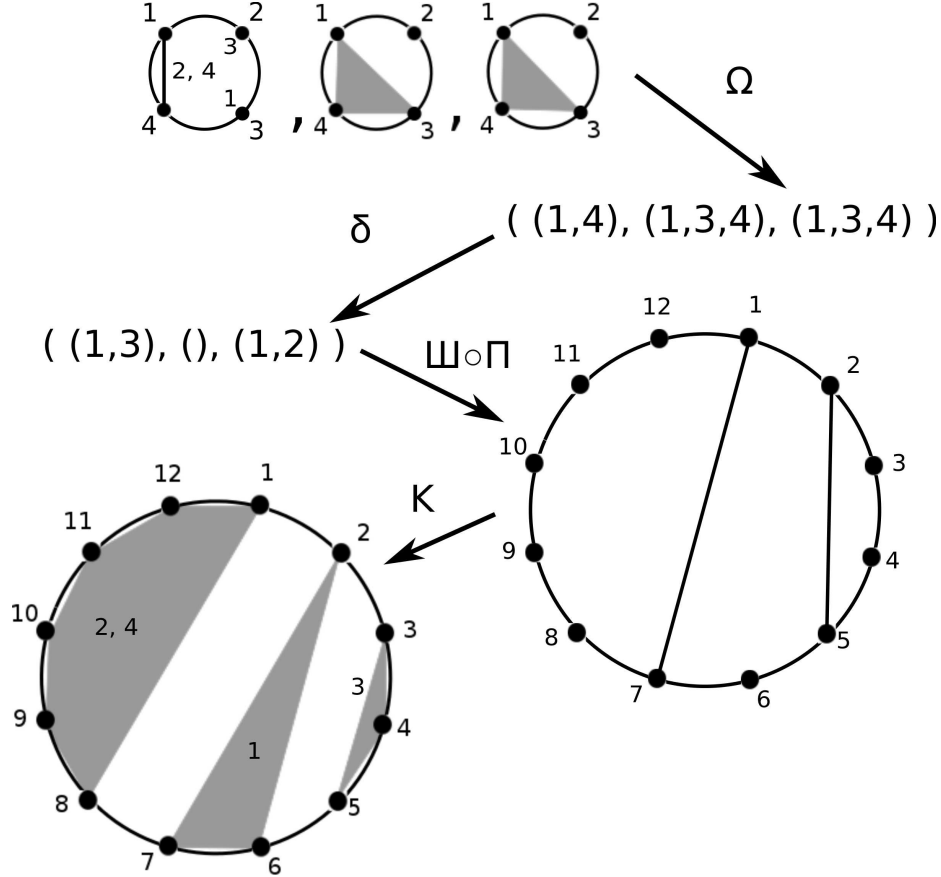}
\caption{A type $k$-$W$-noncrossing parking function for $k = 3$ and $W = \symm_4$}
 \label{fig:Afuss}
\end{figure}

\begin{example}
\label{ex:AkNC}
Suppose $W = \symm_n$ with Coxeter element $c = (1,2 ,\dots, n)$.  
The map $\nabla$ of Lemma~\ref{classical-noncrossing} can be embellished to give a visual
interpretation
of $\Park^{NC}_W(k)$ on $\DD^2$.  Given a $k$-$W$-noncrossing
parking function $[w, X_1 \leq \dots \leq X_k]$, we let $\pi_i$ be the noncrossing 
partition of $[n]$ associated with $X_i$.  We represent this parking function as a sequence
$\pi_1, \dots, \pi_k$ of noncrossing partitions of $[n]$, where the block $B$ of $\pi_1$ is 
labeled with the subset $w(B)$ of $[n]$.
Applying the map $\nabla$ to the sequence $\pi_1 , \dots , \pi_k$ of noncrossing
partitions yields a $k$-divisible noncrossing partition of $[kn]$.  By the last sentence of
Lemma~\ref{classical-noncrossing}, we have that the resriction of 
$\nabla(\pi_1 , \dots , \pi_k)$ to $\{1, k+1, \dots, (n-1)k+1\}$ equals
$\pi_1(\{1, k+1, \dots, (n-1)k+1 \} )$.  We label the block 
of $\nabla(\pi_1 \leq \dots \leq \pi_k)$
equal to the image of $B$
under $i \mapsto (i-1)k + 1$ with $w(B)$.

Figure~\ref{fig:Afuss} demonstrates this procedure when $n = 4$ and $k = 3$.  We start 
with the $k$-$W$-noncrossing parking function $[(1,2,3), \{1,4/2/3 \} \leq \{1,3,4 / 2\} \leq \{1,3,4 / 2\}]$.
Applying the map $\nabla = K \circ \shuffle \circ \Pi \circ \delta \circ \Omega$ yields
$\{ 1,8,9,10,11,12 / 2,6,7 / 3,4,5 \}$.  We remember the labels, labelling 
$\{1,8,9,10,11,12\}$ with $\{2,4\}$, $\{2,6,7\}$ with $\{1\}$, and $\{3,4,5\}$ with $\{3\}$.
Lemma~\ref{classical-noncrossing} guarantees that this remembering is a well defined
procedure.

The visualizations obtained by this procedure are the pairs $(\pi, f)$ where $\pi$ is a 
$k$-divisible noncrossing partition of $[kn]$ and $f$ maps each block $B$ of $\pi$
to a subset $f(B)$ of $[n]$ in such a way that $|f(B)| = k |B|$ and $[n] = \biguplus_{B \in \pi} f(B)$.
The group $W = \symm_n$ acts on these pairs by permuting the labels and the group 
$\ZZ_{kh} = \ZZ_{kn}$
acts by $kn$-fold rotation.

\end{example}

\subsection{The algebraic $k$-$W$-parking space}  The construction of the algebraic
$k$-$W$-parking space is very similar to the $k = 1$ case in \cite{ARR}.

Let $\theta_1, \dots, \theta_n$ be a hsop of degree $kh+1$ inside $\CC[V]$
carrying $V^*$ and let
$(\Theta)$ be the ideal in $\CC[V]$ generated by $\theta_1, \dots, \theta_n$.  
The algebraic parking space is obtained by deforming the ideal $(\Theta)$.  Since
the hsop $\theta_1, \dots, \theta_n$ carries $V^*$, there exists a basis $x_1, \dots, x_n$
of $V^*$ such that the linear map induced by $x_i \mapsto \theta_i$ is $W$-equivariant.

\begin{defn}
\label{algebraic}
Let $\theta_1, \dots, \theta_n$ and $x_1, \dots, x_n$ be as above.  The 
{\sf algebraic $k$-$W$-parking space} is the quotient
\begin{equation}
\Park^{alg}_W(k) := \CC[V] / (\Theta - \xx),
\end{equation}
where $(\Theta - \xx)$ is the ideal in $\CC[V]$ generated by
$\theta_1 - x_1, \dots, \theta_n - x_n$.
\end{defn}

If $W$ 
acts on $\CC[V]$ by linear substitutions and $\ZZ_{kh}$ scales by $\omega^d$ in degree 
$d$, the ideal
$(\Theta - \xx)$ is still $W \times \ZZ_{kh}$-stable, so that 
$\Park^{alg}_W(k)$ is a $W \times \ZZ_{kh}$-module.
The deformation $(\Theta) \leadsto (\Theta - \xx)$ of defining ideals is not so severe as to
affect the $W \times \ZZ_{kh}$-module structure.

\begin{proposition}
\label{deformation}
Preserve the notation of Definition~\ref{algebraic}.
There is a $W \times \ZZ_{kh}$-module isomorphism 
\begin{equation*}
\CC[V] / (\Theta) \cong_{W \times \ZZ_{kh}} \CC[V] / (\Theta - \xx) = \Park^{alg}_W(k).
\end{equation*}
\end{proposition}
\begin{proof}
The proof of this result can be obtained from the proof of \cite[Proposition 2.11]{ARR} word-for-word
by replacing $h+1$ by $kh+1$ everywhere.
\end{proof}

Proposition~\ref{deformation} shows that the deformation
$(\Theta) \leadsto (\Theta - \xx)$ is not `too big'.  We want this deformation
to be `big enough' to make $\Park^{alg}_W(k)$ a `combinatorial' 
$W \times \ZZ_{kh}$-module, i.e., a $W \times \ZZ_{kh}$-permutation representation.

To make this goal precise, for any point $v \in V$ with coordinates $(v_1, \dots, v_n) \in \CC^n$,
let $\mathfrak{m}_v \subseteq \CC[V]$ denote the maximal ideal generated by
$x_1 - v_1, \dots, x_n - v_n$.  Given any ideal $I \subseteq \CC[V]$ such that 
$\CC[V] / I$ is finite dimensional and any point $v$ in the variety $\mathcal{Z}(I)$ cut out by $I$, 
 the {\sf multiplicity} of $v$ is the dimension of the space $\CC[V] / I_v$, where
$I_v$ is the $\mathfrak{m}_v$-primary component of $I$.  In particular, if 
$\theta_1, \dots, \theta_n$ is as above, the variety $\mathcal{Z}(\Theta)$ consists of 
one point at the origin of $(kh+1)^n$.

When $\theta_1, \dots, \theta_n$ and $x_1, \dots, x_n$ are as above, let 
$V^{\Theta}$ denote the subvariety $\mathcal{Z}(\Theta - \xx)$ of $V$ cut out 
by the ideal $(\Theta - \xx)$.  The superscript notation reflects the fact that 
$V^{\Theta}$ is the fixed point set of the 
degree $kh+1$ polynomial
map $\Theta: V \rightarrow V$ which
sends a point with coordinates $(x_1, \dots, x_n)$ to the point with 
coordinates $(\theta_1, \dots, \theta_n)$.  Since the ideal $(\Theta - \xx)$ is 
$W \times \ZZ_{kh}$-stable, the set $V^{\Theta}$ is 
stable under the action of $W \times \ZZ_{kh}$ on $V$.

The Strong Conjecture of this paper
asserts that $V^{\Theta}$ consists of $(kh+1)^n$ distinct points of multiplicity one.  That is,
the deformation $(\Theta) \leadsto (\Theta - \xx)$ blows apart the high multiplicity zero at the
origin and yields a collection of multiplicity one zeros.  In other words, 
the $W \times \ZZ_{kh}$-module $\Park^{alg}_W(k) \cong V^{\Theta}$ is naturally a 
$W \times \ZZ_{kh}$-permutation
module.  
\footnote{Strictly speaking, this would prove that $\Park^{alg}_W(k)$ is isomorphic 
to the {\em dual} of $V^{\Theta}$, but permutation representations are self-dual.}
By Proposition~\ref{deformation}, the character of $\Park^{alg}_W(k)$ 
is given by Equation~\ref{hsopchar}.  The Strong Conjecture would therefore imply 
that the character of $V^{\Theta}$ is also given by Equation~\ref{hsopchar}.

To gain intuition for the construction of $\Park^{alg}_W(k)$, we examine the case where
$W$ has rank $1$.

\begin{example}
\label{rank1}
Suppose that $W = \{1, s \}$ has rank $1$, so that $h = 2$.  Then $V = \CC$ so that we can identify
$\CC[V] \cong \CC[x]$ and $\{ x \}$ is a basis for $V^*$.  
The nonidentity element of $W$ acts by $s.x = -x$.
Any hsop of degree $2k + 1$ has 
the form $\theta = \alpha x^{2k+1}$ for some $\alpha \in \CC^{\times}$, and
this hsop carries $V^*$.  Moreover, the 
$\CC$-linear map induced by $x \mapsto \alpha x^{2k+1}$ is $W$-equivariant.

The ideal $(\Theta) = (\alpha x^{2k+1})$ cuts out the origin $\{ 0 \}$ with multiplicity 
$2k+1$.  On the other hand, we have that 
$V^{\Theta} = \mathcal{Z}(\alpha x^{2k+1} - x) = \{ 0, \beta, \omega \beta, \dots, \omega^{2k-1} \beta \}$, where 
$\beta^{-2k} = \alpha$ (and $\omega$ is a primitive $2k^{th}$ root of unity).  Moreover, each point in
$V^{\Theta}$ is cut out with multiplicity one, so that $V^{\Theta} \cong \Park^{alg}_W(k)$.
The nonzero element $s \in W$ acts on $V^{\Theta}$ as multiplication by $-1$ and the 
distinguished
generator $g$ of $\ZZ_{2k}$ scales by $\omega$. 

Let $c = s \in W$ be the unique choice of Coxeter element.  Then $NC(W) = W$ and
$\LLL = \{ 0, V \}$.   
Also,
 $\Park^{NC}_W(k)$ consists of the $2k + 1$ equivalence classes
\begin{equation*}
\Park^{NC}_W(k) = 
\{ [1, 0^k] \} \uplus \{ [1, V^i 0^{k-i}], [s, V^i 0^{k-i}] \,:\, 0 < i \leq k \}.  
\end{equation*}
 The action of $g \in \ZZ_{2k}$ on $\Park^{NC}_W(k)$ is given by
 \begin{align*}
g.[1, 0^k] &= [1, 0^k],  \\
g.[w, V^i 0^{k-i}] &= [w, V^{i+1} 0^{k-i-1}]  \text{ for $0 < i < k,$}\\
g.[w, V^k] &= [sw, V^1 0^{k-1}].
\end{align*} 
 Therefore, the assignments 
$[1, 0^k] \mapsto 0$ and $[1, V^k] \mapsto \beta$ induce a 
$W \times \ZZ_{2k}$-equivariant biejction 
$\Park^{NC}_W(k) \xrightarrow{\sim} \Park^{alg}_W(k)$.

 \begin{figure}
\includegraphics[scale=0.4]{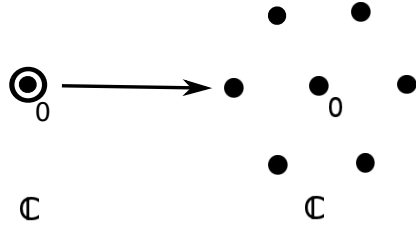}
\caption{The deformation in rank 1 when $k = 3$.  The origin has multiplicity $7$ on the 
left.}
\label{fig:deform}
\end{figure}

Figure~\ref{fig:deform} shows the deformation of varieties induced by 
$(\Theta) \leadsto (\Theta - \xx)$ in rank $1$ when $k = 3$.  The origin has multiplicity $7$
on the left, and the $7$ multiplicity one points on the right lie at the vertices and center
of a regular hexagon in the complex plane centered at the origin.  The nonidentity element
$s \in W$ acts by the scalar $-1$ and
$\ZZ_6$ acts by rotation.

\end{example}

\subsection{The Main Conjecture}  The Main Conjecture of this paper asserts a
deep relationship between $\Park^{NC}_W(k)$
and $\Park^{alg}_W(k)$.  We present its three avatars in decreasing order of strength.

\begin{strongc}
{\bf For any} hsop $\theta_1, \dots, \theta_n$ of degree $kh+1$ carrying $V^*$ and
any basis $x_1, \dots, x_n$ of $V^*$ such that the $\CC$-linear map induced by
$x_i \mapsto \theta_i$
is $W$-equivariant, the variety $V^{\Theta}$ cut out by $(\Theta - \xx)$ consists of
$(kh+1)^n$ distinct points of multiplicity one, so that 
$V^{\Theta} \cong_{W \times \ZZ_{kh}}  \Park^{alg}_W(k)$.
Furthermore, there is a $W \times \ZZ_{kh}$-equivariant bijection 
$V^{\Theta} \xrightarrow{\sim} \Park^{NC}_W(k)$.
\end{strongc}

Example~\ref{rank1} shows that the Strong Conjecture is true in rank 1.
In Section~\ref{The Main Conjecture: Strong Version} we give further evidence for this
conjecture as it applies to an arbitrary reflection group $W$.  Even if the Strong 
Conjecture is false, its conclusion may still hold for special choices of hsop.

\begin{intermediatec}
{\bf There exists a} hsop $\theta_1, \dots, \theta_n$ of degree $kh+1$ carrying $V^*$ and
a  basis $x_1, \dots, x_n$ of $V^*$ such that the $\CC$-linear map induced by
$x_i \mapsto \theta_i$
is $W$-equivariant and the variety $V^{\Theta}$ cut out by $(\Theta - \xx)$ consists of
$(kh+1)^n$ distinct points of multiplicity one, so that 
$V^{\Theta} \cong_{W \times \ZZ_{kh}}  \Park^{alg}_W(k)$.
Furthermore, there is a $W \times \ZZ_{kh}$-equivariant bijection 
$V^{\Theta} \xrightarrow{\sim} \Park^{NC}_W(k)$.
\end{intermediatec}

An argument of Etingof reproduced in \cite[Theorem 12.1]{ARR} implies that there exists 
an hsop $\theta_1, \dots, \theta_n$ of degree $kh+1$ such that the variety $V^{\Theta}$ consists of
$(kh+1)^n$ distinct points.  Unfortunately, this hsop arises in a subtle fashion
from the theory of rational Cherednik 
algebras and we are unable to calculate 
the $W \times \ZZ_{kh}$-structure of the resulting variety $V^{\Theta}$.

In Sections \ref{Type I}, \ref{Type BC}, and \ref{Type D}, we prove the 
Intermediate Conjecture when $W$ has type BCDI.  These proofs rely on convenient
choices of hsops $\theta_1, \dots, \theta_n$ and explicit visualizations of 
$\Park^{NC}_W(k)$ in these types.  In type BC, our visualization of $\Park^{NC}_W(k)$
will come from labeled centrally symmetric $k$-divisible noncrossing partitions on the 
disc $\DD^2$.  When $W$ is of type $D$, the visualization will come from an
interpretation of $k$-$W$-noncrossing partitions on the {\bf annulus} $\AAA^2$ due
to Krattenthaler and M\"uller \cite{KrattenthalerMuller} and Kim \cite{Kim}.

Either the Strong or Intermediate Conjectures imply that the 
character of the $W \times \ZZ_{kh}$-module $\Park^{NC}_W(k)$ is given by Equation~\ref{hsopchar}.  Even if the geometric statements about
$V^{\Theta}$ asserted in these conjectures are false, the module $\Park^{NC}_W(k)$ might
still have this character.

\begin{weakc}
The character $\chi : W \times \ZZ_{kh} \rightarrow \CC$ of the module 
$\Park^{NC}_W(k)$ is given by Equation~\ref{hsopchar}.
Equivalently, we have an isomorphism $\Park^{NC}_W(k) \cong_{W \times \ZZ_{kh}}
\Park^{alg}_W(k)$ of $W \times \ZZ_{kh}$-modules.
\end{weakc}

In Section~\ref{Type A} we prove the Weak  
Conjecture in type A using the visual representation of $\Park^{NC}_{\symm_n}(k)$
in Example~\ref{ex:AkNC}.  The difficulty of constructing hsops in type A has so far obstructed
proof of the Intermediate and Strong Conjectures for the symmetric group.

Even the Weak Conjecture remains open in the exceptional types EFH.  
While the $k = 1$ case in \cite{ARR} was a finite computer check in these types 
(which could not be performed in types E$_7$ and E$_8$ due to computational 
complexity), this is (a priori) an infinite statement in our context since the 
Fuss parameter $k$ could be any positive integer.

\begin{table}
\label{truth}
\caption{The strongest version of the Main Conjecture proven in a given type.}
\begin{tabular} {| p{3.5cm} |  p{4cm} | p{4cm} |}
\hline
Conjecture Version  & Proven for $k = 1$ & Proven for arbitrary $k$  \\ \hline
Strong & A$_1$, I$_2(m)$ & A$_1$  \\ \hline
Intermediate & B$_n$/C$_n$, D$_n$ & B$_n$/C$_n$, D$_n$, I$_2(m)$ \\ \hline
Weak & A$_{n-1}$, E$_6$, F$_4$, H$_3$, H$_4$ & A$_{n-1}$ \\ \hline
\end{tabular}
\end{table}

Table~\ref{truth} gives the status of the Main Conjecture for various reflection
groups.

\subsection{Consequence of the Weak Conjecture:  A Cyclic Sieving Phenomenon}
Even the Weak Conjecture determines the cycle structure of 
the action of $g$ on $NC^k(W)$ (or $NC_k(W)$), where $\ZZ_{kh} = \langle g \rangle$, and
proves a cyclic sieving conjecture of Armstrong \cite[Conjecture 5.4.7]{Armstrong}.  Our proof
modulo the Weak Conjecture
is type uniform and representation theoretic (rather than directly enumerative).

Let $\Park^{NC}_W(k)^W$ denote the $W$-invariant subspace of 
$\Park^{NC}_W(k)$.  This invariant subspace has a natural basis indexed by 
orbits in $\Park^{NC}_W(k)/W$, and these orbits biject naturally with $k$-$W$-noncrossing
partitions.  The residual action of $\ZZ_{kh}$ on this basis is given by the action
of $\ZZ_{kh}$ on $NC^k(W)$.  Therefore, the trace of $g^d \in \ZZ_{kh}$ acting on 
$\Park^{NC}_W(k)^W$ equals the number of elements
$(w_1 \leq \dots \leq w_k) \in NC^k(W)$ which 
satisfy $g^d.(w_1 \leq \dots \leq w_k) = (w_1 \leq \dots \leq w_k)$.

On the other hand, if the Weak Conjecture holds for $W$, the 
trace of $g^d$ on $\Park^{NC}_W(k)^W$ equals the evaluation of the 
Hilbert series for $\CC[V] / (\Theta) ^ W$ at $q = \omega^d$.  
This Hilbert series
is the {\sf $q$-Fuss-Catalan number} for $W$ (see \cite{BessisR, BEG, Gordon}):
\begin{equation}
\Cat^k(W; q) = \prod_{i = 1}^n \frac{1 - q^{kh + d_i}}{1 - q^{d_i}},
\end{equation}
where $d_1, \dots, d_n$ are the degrees of $W$.

In summary, whenever the 
weak conjecture holds for $W$ we recover  a cyclic sieving result proven
originally by Krattenthaler and M\"uller \cite{KrattenthalerMuller, KrattenthalerMuller2}.

\begin{theorem} 
\label{CSPtheorem}
Suppose the Weak Conjecture holds for $W$.  Then the triple
\begin{equation*}
(NC^k(W), \Cat^k(W; q), \ZZ_{kh})
\end{equation*}
exhibits the {\sf cyclic sieving phenomenon}: the number of elements 
in $NC^k(W)$ fixed by $g^d$ equals the polynomial $\Cat^k(W; q)$
evaluated at $q = \omega^d$.
\end{theorem}

The $k = 1$ case of
Theorem~\ref{CSPtheorem} was proven in a case-by-case manner by 
Bessis and Reiner \cite{BessisR}.  Krattenthaler and M\"uller gave a case-by-case
proof for arbitrary $k$.  
The above argument proves Theorem~\ref{CSPtheorem} uniformly modulo the Weak
Conjecture.


\section{The Strong Conjecture}
\label{The Main Conjecture: Strong Version}

In this section we give evidence for the Strong Conjecture by
computing certain subsets of the variety $V^{\Theta}$ and showing that they consist
of multiplicity one points and have $W \times \ZZ_{kh}$-equivariant injections onto 
explicit subsets of $\Park^{NC}_W(k)$.

To state our evidence for the Strong Conjecture,
we introduce a stratification of $V^{\Theta}$ which is stable under 
the $W \times \ZZ_{kh}$-action.  Define the {\sf dimension} of a point $v \in V^{\Theta}$
to be the minimum dimension $\dim(X)$ of a flat $X \in \LLL$ such that $v \in X$.
Let $V^{\Theta}(d)$ denote the set of $d$-dimensional points in $V^{\Theta}$.  Then we 
have a disjoint union decomposition
\begin{equation*}
V^{\Theta} = V^{\Theta}(0) \uplus V^{\Theta}(1) \uplus \dots \uplus V^{\Theta}(n),
\end{equation*}
where $n = \dim(V)$ is the rank of $W$.  Furthermore, each of the sets in this decomposition
is stable under the action of $W \times \ZZ_{kh}$.

\subsection{Analysis of $V^{\Theta}(0)$}  We have the following result about the set 
$V^{\Theta}(0)$.

\begin{proposition}
\label{dimension-zero-proposition}
The set $V^{\Theta}(0)$ consists of the single point $\{ 0 \}$, and this point is cut out
by $(\Theta - \xx)$
with multiplicity one.  The map
\begin{equation*}
0 \mapsto [1, 0 \leq \dots \leq 0]
\end{equation*}
defines a $W \times \ZZ_{kh}$-equivariant injection 
$V^{\Theta}(0) \hookrightarrow \Park^{NC}_W(k)$.
\end{proposition}

\begin{proof}
It is clear that the origin is in the zero locus of $(\Theta - \xx)$, so that
$V^{\Theta}(0) = \{ 0 \}$.  The equivariance of the assignment in the proposition is 
also clear.  For the multiplicity one assertion, consider the Jacobian 
$( \frac{\partial f_i}{\partial x_j} )_{1 \leq i < j \leq n}$ of the map
$(f_1, \dots, f_n) = (\theta_1 - x_1, \dots, \theta_n - x_n)$ evaluated at 
$x_1 = \dots = x_n = 0$.  Since each $\theta_i$ is homogeneous of degree
$kh + 1 > 1$, this evaluation equals the negative of the $n \times n$ identity matrix,
which is nonsingular.  
\end{proof}

\subsection{Analysis of $V^{\Theta}(1)$}  The analog of 
Proposition~\ref{dimension-zero-proposition}
for one dimensional points is more difficult to prove.  We recall classical material
related to bipartite Coxeter elements and the Coxeter plane which can  be
found in \cite[Chapter 3]{Humphreys}.

Let $S = S_+ \uplus S_-$
be a partition of the simple reflections in $W$ into two sets of pairwise commuting elements.
Such a partition can be achieved because the Coxeter graph associated to $W$ is a tree.
A {\sf bipartite} Coxeter element $c$ of $W$ is a Coxeter element written
$c = c_+ c_-$, where $c_{\epsilon}$ is a product of the generators in $S_{\epsilon}$ (in 
any order) for $\epsilon \in \{ + , - \}$. 

Under the above assumptions, the {\sf Coxeter plane} $P$ is a $2$-dimensional subspace
$P \subseteq V$ on which the dihedral group $\langle c_+, c_- \rangle$ generated by
$c_+$ and $c_-$ acts.  It is known that the bipartite Coxeter element $c = c_+ c_-$ acts
on $P$ by $h$-fold rotation.  We have the following fact, proven uniformly in
\cite{ARR}, which relates noncrossing lines (with respect to the 
bipartite Coxeter element $c$) and the Coxeter plane.

\begin{lemma} \cite[Lemma 5.4]{ARR}
\label{noncrossing-lines}
Noncrossing one-dimensional flats are never orthogonal to the Coxeter plane.
\end{lemma}

We will also need the following result.

\begin{lemma}
\label{even-h}
Suppose there exists a one-dimensional flat $X \in \LLL$ and $w \in W$
such that $w$ acts as $-1$ on $X$.  Then the Coxeter number $h$ is even.
\end{lemma}
\begin{proof}
We will reduce the statement of the lemma to two successively weaker claims.
For the remainder of the proof, consider the reflection representation $V$ over $\RR$.

Write the simple roots as $\Pi = \{ \alpha_1, \dots, \alpha_n \}$ and let 
$\delta_1, \dots, \delta_n \in V$
be the corresponding {\sf fundamental weights} defined by
 $\langle \alpha_i, \delta_j \rangle = \delta_{i,j}$.  Then the {\sf fundamental chamber}
 $F = \RR_{\geq 0} \delta_1 + \cdots + \RR_{\geq 0} \delta_n$ is a fundamental domain for
 the action of $W$ on $V$.
 
 {\bf Claim 1.}  Suppose there exists $1 \leq i \leq n$ and $w \in W$ so that 
 $w$ acts as $-1$ on $X = \RR \delta_i \in \LLL$.  Then $h$ is even.

 To see why Claim 1 proves the lemma, let $X \in \LLL$ be arbitrary and suppose $w \in W$
 acts as $-1$ on $X$.  By standard geometric results about reflection groups (see
 \cite{Humphreys}), there exists $v \in W$ and $1 \leq i \leq n$ such that
 $v X = \RR \delta_i$.  Therefore, the group element $v w v^{-1} \in W$ acts as $-1$
 on $\RR \delta_i$.
 
Let $w_0 \in W$ be the long element.
 
 {\bf Claim 2.}  Suppose there exists $1 \leq i \leq n$ such that 
 $w_0$ acts as $-1$ on $X = \RR \delta_i$.  Then $h$ is even.
 
 To see why Claim 2 proves Claim 1, suppose that $1 \leq i \leq n$ and
 $w$ acts as $-1$ on $X = \RR \delta_i$ for some $w \in W$.  In particular, we have that
 $w \delta_i = - \delta_i$.  Since $w_0$ carries $-F$ to $F$, we have that 
 $- w_0 \delta_i = w_0 w \delta_i \in F$.  By \cite[Theorem 1.12]{Humphreys}, this forces
 $w_0 w \delta_i = \delta_i$, so that $- w_0 \delta_i = \delta_i$ and $w_0$ acts as $-1$
 on $\delta_i$.  
 
 We claim that $- w_0 \alpha_i = \alpha_i$.  To see this, notice that
 for $1 \leq j \leq n$, we have
 $\langle -w_0 \alpha_i, \delta_j \rangle = \langle \alpha_i, - w_0 \delta_j \rangle$
 because $-w_0$ is an isometry.  Since $-w_0 \delta_i = \delta_i$, we conclude that
 $\langle -w_0 \alpha_i, \delta_i \rangle = 1$.  For $j \neq i$, we have that 
 $-w_0 \delta_j$ is proportional to 
 $\delta_{\ell}$ for some $\ell$ because the rays generated by the fundamental
 weights are the extremal rays for $F$ and $w_0$ carries $F$ to $-F$.  We cannot have
 $i = \ell$, so we conclude that $\langle -w_0 \alpha_i, \delta_j \rangle = 0$ for $i \neq j$.
 This proves that $-w_0 \alpha_i = \alpha_i$.
 
Without loss of generality, assume that the Coxeter element $c = c_+ c_-$ is bipartite.
For the sake of contradiction, assume that $h$ is odd.  In this case we can write
$w_0 = c^{\frac{h-1}{2}} c_+$. 
We have that $c_+ \alpha_i = - \alpha_i$ if $\alpha_i$ corresponds to a simple reflection $s_i$ in
$c_+$ and $c_- \alpha_i = - \alpha_i$ if $\alpha_i$ corresponds to a simple reflection $s_i$ in $c_-$.
It follows that 
\begin{equation*}
 \alpha_i = - w_0 \alpha_i = \begin{cases}
c^{\frac{h-1}{2}} \alpha_i & \text{if $s_i$ occurs in $c_+$,} \\
c^{\frac{h+1}{2}} \alpha_i & \text{if $s_i$ occurs in $c_-$.}
\end{cases}
\end{equation*}
Therefore, if $s_i$ occurs in $c_-$, we have that $\alpha_i = (c^{\frac{h+1}{2}})^2 \alpha_i
= c \alpha_i$, which contradicts the fact that $V^c = 0$.  Similarly, if $s_i$ occurs in 
$c_+$, we have that $\alpha_i = c^{-1} \alpha_i$, which also contradicts the fact
that $V^{c} = 0$.  These contradictions show that $h$ must be even, as desired.
\end{proof}

A case-by-case check of Lemma~\ref{even-h} is easy to perform.  The 
Coxeter number $h$ is only odd in type A$_{n-1}$ for $n$ odd and type I$_2(m)$ for
$m$ odd.  In type A$_{n-1}$, any one-dimensional flat is $W$-conjugate to a flat of the form
$\{ (a, \dots, a, b, \dots, b) \in V \}$, where there are $k$ copies of $a$ and $n-k$ copies of
$b$ and $ka + (n-k)b = 0$.  If $n$ is odd, no element of $W = \symm_n$ acts as $-1$ on
such a flat.  Lemma~\ref{even-h} can be verified in type I$_2(m)$ for $m$ odd by noting that
no symmetry of the regular $m$-gon acts as $-1$ on any reflecting hyperplane.

\begin{proposition}
\label{dimension-one-proposition}
The set $V^{\Theta}(1)$ consists entirely of points cut out by $(\Theta - \xx)$ with 
multiplicity one.  Moreover, there is a $W \times \ZZ_{kh}$-equivariant injection
$V^{\Theta}(1) \hookrightarrow \Park^{NC}_W(k)$ whose image is precisely the
set of $k$-$W$-noncrossing parking functions 
$[w, X_1 \leq \dots \leq X_k]$ with $\dim(X_1) = 1$.
\end{proposition}

We remark that when $n = 1$, the sets $V^{\Theta}(0)$ and $V^{\Theta}(1)$ exhaust
$V^{\Theta}$, so that Propositions~\ref{dimension-zero-proposition} 
and \ref{dimension-one-proposition}
reprove the 
Strong Conjecture in rank 1.  Roughly speaking, Propositions~\ref{dimension-zero-proposition}
and \ref{dimension-one-proposition} state that the Strong Conjecture holds `in dimension $\leq 1$'
for arbitrary $W$.

\begin{proof}
Without loss of generality, assume the underlying Coxeter element $c \in W$ is 
bipartite and let $P \subseteq V$ be the associated Coxeter plane.

For any flat $X \in \LLL$, there exists $w \in W$ such that $X = V^w$.  Since the polynomial
map $\Theta: V \rightarrow V$ is $W$-equivariant, this implies that $\Theta(X) \subseteq X$
and the restriction $\Theta |_X$ of $\Theta$ to $X$ is a well defined polynomial map
of degree $kh+1$.  

Assume that the flat $X \in \LLL$ is one-dimensional.
Identifying the coordinate ring $\CC[X]$ with $\CC[x]$, up to
scaling $\Theta|_X$ can be identified with $\alpha x^{kh+1} - x$ for some 
$\alpha \in \CC^{\times}$.  The solutions to this equation all have multiplicity one, and 
are given by
\begin{equation*}
\{ 0, \beta, \omega \beta, \dots, \omega^{kh - 1} \beta \},
\end{equation*}
where $\beta^{-kh} = \alpha$ and $\omega$ is a primitive $kh^{th}$ root of unity.  
Since $V^{\Theta}(0) = \{ 0 \}$, we have 
$V^{\Theta}(1) \cap X = \{ \beta, \omega \beta, \dots, \omega^{kh-1} \beta \}$.
The cyclic group $\ZZ_{kh}$ acts on this set by multiplication by $\omega$ and
the group $W$ acts on all of $V^{\Theta}(1)$ by the restriction of its action on $V$.

Let $X \in \LLL$ be one-dimensional and noncrossing
and write $V^{\Theta}(1) \cap X$ as
$\{ \beta, \omega \beta, \dots, \omega^{kh-1} \beta \}$ as in the last paragraph.  To build 
our $W \times \ZZ_{kh}$-equivariant map $V^{\Theta}(1) \rightarrow \Park^{NC}_W(k)$ we will begin
by proving the following lemma equating the stabilizers of $\beta \in V^{\Theta}(1)$ and
$[1, X0^{k-1}] \in \Park^{NC}_W(k)$.

\begin{lemma}
\label{stabilizer-lemma}
The $W \times \ZZ_{kh}$-stabilizers of $\beta \in V^{\Theta}(1)$ and 
$[1, X0^{k-1}] \in \Park^{NC}_W(k)$ are both equal to
\begin{equation*}
\left \{  (w, g^d) \in W \times \ZZ_{kh} \,:\, \text{$w$ fixes $X$ pointwise and $d \equiv 0$ modulo
$kh$} \right \},
\end{equation*}
if $h$ is odd and
\begin{equation*}
\left \{ (w, g^d) \in W \times \ZZ_{kh} \,:\, \begin{matrix} \text{$w$ fixes $X$ pointwise and
$d \equiv 0$ modulo $kh$, or} \\ \text{$w$ acts as $-1$ on $X$ and $d \equiv \frac{kh}{2}$ 
modulo $kh$} \end{matrix} \right \},
\end{equation*}
if $h$ is even.
\end{lemma}

\begin{proof} (of Lemma~\ref{stabilizer-lemma})
Let $G$ be the subgroup of $W \times \ZZ_{kh}$ in the statement of the lemma, let $H$
be the $W \times \ZZ_{kh}$-stabilizer of $\beta$, and let $K$ be the $W \times \ZZ_{kh}$-stabilizer
of $[1, X 0^{k-1}]$.

We begin by showing that $G = H$.  Let $(w, g^d) \in G$.  We have that 
$(1, g^d).\beta = \omega^d \beta$.  Also, the group element $w \in W$ acts as the scalar transformation
$\omega^{-d}$ on $X$.  We conclude that $(w, g^d)$ stabilizes $\beta$ and 
$G \subseteq H$. 

Let $(w, g^d) \in H$.  We have that $w.\beta = \omega^{-d} \beta$.  Since $\beta$ is a nonzero 
element of the one-dimensional flat $X$, this implies that $w$ stabilizes $X$.  Since $w$ acts
as an orthogonal transformation on the real locus $X_{\RR}$, either $w$ acts as $1$ on $X$
(forcing $d \equiv 0$ modulo $kh$) or $w$ acts as $-1$ on $X$ (forcing 
$d \equiv \frac{kh}{2}$ modulo $kh$ and $h$ to be even by Lemma~\ref{even-h}).
We conclude that $H \subseteq G$.

We show next that $G = K$.  To do this, we observe that if $Y \in \LLL$ is any one-dimensional
noncrossing flat and $0 < i < k$, we have
\begin{equation}
\label{g-action-one}
g: [1, Y^i 0^{k-i}] \mapsto [1, Y^{i+1} 0^{k-i-1}].
\end{equation}
Since $g^k.[w, X_1 \leq \dots \leq X_k] = [w c^{-1}, c X_1 \leq \dots \leq c X_k]$ always, 
we have that
\begin{equation}
\label{g-action-two}
g: [1, Y^k] \mapsto [c^{-1}, (cY) 0^{k-1}].
\end{equation}

Let $(w, g^d) \in G$.  If $d \equiv 0$ modulo $kh$, Equation~\ref{g-action-two} implies that
$g^d.[1, X 0^{k-1}] = [c^{-h}, (c^h X) 0^{k-1}] = [1, X 0^{k-1}]$.  Since $w \in W_X$, we have that
$(w, g^d).[1, X 0^{k-1}] = [w, X 0^{k-1}] = [1, X 0^{k-1}]$, so that
$(w, g^d) \in G$.  If $h$ is even, $d \equiv \frac{kh}{2}$ modulo $kh$, and 
$w$ acts as $-1$ on $X$, Equation~\ref{g-action-two} implies that 
$(w, g^d).[1, X 0^{k-1}] = [w c^{-\frac{h}{2}}, (c^{\frac{h}{2}}X) 0^{k-1}]$.  It is well known that
when $h$ is even, the group element $c^{\frac{h}{2}}$ acts as the scalar transformation
$-1$ on $V$.  Therefore, we have that $c^{\frac{h}{2}} X = X$ and 
$w c^{-\frac{h}{2}} \in W_X$ because $w$ acts as $-1$ on $X$.
We conclude that $G \subseteq K$.

Let $(w, g^d) \in K$.  Then we have that $g^d.(X0^{k-1}) = (X 0^{k-1})$.  By
Equations~\ref{g-action-one} and \ref{g-action-two} this forces 
$d \equiv 0$ modulo $k$.  Writing $d = mk$, we must also have that 
$c^m X = X$.  
Let $\pi: V \twoheadrightarrow P$ be the orthogonal projection map, where $P$ is the
Coxeter plane.
Since $c^m$ is an orthogonal transformation, we have that 
$c^m. \pi(X) = \pi(X)$.  By Lemma~\ref{noncrossing-lines}, the space
$\pi(X)$ is nonzero, and therefore a one-dimensional subspace of $P$.  Since
$c$ acts on the real locus of $P$ by $h$-fold rotation, this forces $m \equiv 0$ modulo
$h$ and $c^m$ to act as $1$ on $X$ or $m \equiv \frac{h}{2}$ modulo $h$ and
$c^m$ to act as $-1$ on $X$.  (By Lemma~\ref{even-h}, this latter case only arises
when $h$ is even.)  Moreover, we must have that $w c^{-m}$ acts as $1$ on $X$.  
This implies that $w$ fixes $X$ pointwise or both $w$ and $c^{-m}$ act as $-1$ on $X$, and 
the latter case only arises when $h$ is even.  We conclude that
$K \subseteq G$.
\end{proof}

Let $\mathcal{O}_1, \dots, \mathcal{O}_t$ be a complete list of the 
$W \times \ZZ_{kh}$-orbits in $V^{\Theta}(1)$.  For each $1 \leq i \leq t$, let 
$\beta_i \in \mathcal{O}_i$ and let $X_i \in \LLL$ be the unique one-dimensional flat
such that $\beta_i \in X_i$.  By Lemma~\ref{conjugate-to-noncrossing}, we can assume without
loss of generality that $X_i$ is noncrossing for each $i$. Moreover, we have that
\begin{equation*}
\mathcal{O}_i = V^{\Theta}(1) \cap \bigcup_{w \in W} w.X_i,
\end{equation*}
and $\mathcal{O}_i \cap w.X_i$ contains $kh$ points in $V^{\Theta}(1)$ for all $w \in W$.  
However, we also know that {\em any} one-dimensional flat $Y \in \LLL$ contains
$kh$ points of $V^{\Theta}(1)$.
Since $\mathcal{O}_1, \dots, \mathcal{O}_t$ is a complete list of the 
$W \times \ZZ_{kh}$-orbits in $V^{\Theta}(1)$, we conclude that the
list $X_1, \dots, X_t$ is a transversal for the action of $W$ on one-dimensional flats
in $\LLL$.

By Lemma~\ref{stabilizer-lemma}, the assignments 
$\beta_i \mapsto [1, X_i 0^{k-1}]$ (for $1 \leq i \leq t$) induce a unique $W \times \ZZ_{kh}$-equivariant
map $V^{\Theta}(1) \rightarrow \Park^{NC}_W(k)$.  Moreover, this map restricts to an injection
on each orbit $\mathcal{O}_i$.  By Equations~\ref{g-action-one} and \ref{g-action-two}, the image
of $\mathcal{O}_i$ under this map is
\begin{equation*}
\left \{ [w, Y_i^j 0^{k-j}] \,:\, \begin{matrix} \text{$w \in W$, $0 < j \leq k$} \\ 
\text{$Y_i$ is in the $C$-orbit of $X_i$} \end{matrix} \right \}.
\end{equation*}
By Lemma~\ref{W-and-C-orbits}, we can replace the above set with
\begin{equation*}
\left \{ [w, Y_i^j 0^{k-j}] \,:\, \begin{matrix} \text{$w \in W$, $0 < j \leq k$} \\ 
\text{$Y_i$ is in the $W$-orbit of $X_i$ and noncrossing} \end{matrix} \right \}.
\end{equation*}
Since $X_1, \dots, X_t$ forms a transversal for the action of $W$ on one-dimensional flats, 
we conclude that the $W \times \ZZ_{kh}$-equivariant
map $V^{\Theta}(1) \rightarrow \Park^{NC}_W(k)$ constructed above
has image as claimed
in the statement of the proposition. The fact that $X_1, \dots, X_t$ is a transversal also 
implies that the images of distinct orbits $\mathcal{O}_i$ under this map are disjoint.
This implies that the map $V^{\Theta}(1) \rightarrow \Park^{NC}_W(k)$ is injective, completing
the proof.
\end{proof}

\section{Type I}
\label{Type I}

In this section we let $W$ be of dihedral type I$_2(m)$ for $m \geq 3$.  We prove the  
Intermediate Conjecture for $W$.  

\subsection{Analysis of $V^{\Theta}$}  The reflection representation $V$ of $W$
can be identified with $\CC^2$, where $W$ acts by dihedral symmetries on the real
locus $V_{\RR} = \RR^2$.  
A choice of simple reflections is given by
$S = \{s, t \}$ where $s$ and
 $t$ are adjacent dihedral reflections.
The distinguished Coxeter element $c := st$ acts 
on $V_{\RR}$
as rotation by $\frac{2 \pi}{m}$ radians.
The Coxeter number $h$ is equal to $m$.

By diagonalizing the action of $c$ on $V$, we can choose coordinate functions 
$x, y$ on $V$ so that 
\begin{align*}
c.x &= \omega^k x, \\
c.y &= \omega^{-k} y, \\
s.x &= y, \\ 
s.y &= x.
\end{align*}

We identify $\CC[V] \cong \CC[x, y]$.
It follows that 
$\theta_1 = x^{kh+1}, \theta_2 = y^{kh+1}$ is an hsop of degree $kh+1$ carrying
$V^*$ and the map $x \mapsto \theta_1, y \mapsto \theta_2$ is $W$-equivariant.

The variety $V^{\Theta}$ is the set of points cut out by 
$(\Theta - \xx) = (x^{km+1} - x, y^{km+1} - y)$. 
This locus is given by
\begin{equation*}
V^{\Theta} = \{ (0, 0) \} \uplus \{ (0, \omega^i), (\omega^i, 0) \,:\, 0 \leq i < km \} \uplus
\{ (\omega^i, \omega^j) \,:\, 0 \leq i, j < km \}.
\end{equation*}
We have that $|V^{\Theta}| = 1 + 2km + (km)^2 = (kh + 1)^2$, and each point in
$V^{\Theta}$ is cut out by $(\Theta - \xx)$ with multiplicity one.  
If we write $W = \langle s, c \rangle$ and let $g$ be the distinguished generator of 
$\ZZ_{kh}$, the action of $W \times \ZZ_{kh}$ on $V^{\Theta}$ is given by
\begin{align}
\label{s-locus-action}
s.(v_1, v_2) &= (v_2, v_1), \\
\label{c-locus-action}
c.(v_1, v_2) &= (\omega^k v_1, \omega^{-k} v_2), \\
\label{g-locus-action}
g.(v_1, v_2) &= (\omega v_1, \omega v_2). 
\end{align}

\subsection{Analysis of $\Park^{NC}_W(k)$}  The intersection lattice 
$\LLL$ of $W$ consists of the $m$ reflecting hyperplanes in $W$, the 
reflection representation $V$, and $\{ 0 \}$.  All of these flats are noncrossing.

Let $H, H' \subset V$ be the reflecting hyperplanes for $s, t \in W$, respectively.  
When $m$ is odd, the 
hyperplanes in $\Cox(W)$ form a single $W$-orbit given by
$\{ H, c.H, \dots, c^{m-1}.H \}$.  When $m$ is even, the hyperplanes in $\Cox(W)$ break
up into two $W$-orbits given by
$\{H, c.H, \dots, c^{\frac{m}{2} - 1}.H \}$ and
$\{H' , c.H', \dots, c^{\frac{m}{2}-1}.H' \}$.  

A typical element of $\Park^{NC}_W(k)$ is an equivalence class
$[w, X_1 \leq \dots \leq X_k]$, where $X_1 \leq \dots \leq X_k$ is of the form
$V^i X^j 0^{k-i-j}$ and
$X$ is a hyperplane in $\Cox(W)$.  We record the following facts about the action of 
$g$ on $\Park^{NC}_W(k)$, where $X \in \LLL$ is one-dimensional.
\begin{align}
\label{dihedral-g-action-one}
g.[1, V^i 0^{k-i}] &= [1, V^{i+1} 0^{k-i-1}]  \text{\hspace{.1in}(for $0 < i < k$)} \\
\label{dihedral-g-action-two}
g.[1, V^k] &= [c^{-1}, V 0^{k-1}] \\
\label{dihedral-g-action-three}
g.[1, X^i 0^{k-i}] &= [1, X^{i+1} 0^{k-i-1}] \text{\hspace{.1in}(for $0 < i < k$)} \\
\label{dihedral-g-action-four}
g.[1, X^k] &= [c^{-1}, (cX) 0^{k-1}] \\
\label{dihedral-g-action-five}
g.[1, V^i X^j 0^{k-i-j}] &= [1, V^{i+1} X^j 0^{k-i-j-1}] 
\text{\hspace{.1in}(for $i, j > 0$, $k > i+j$)} \\
\label{dihedral-g-action-six}
g.[1, V^i X^{k-i}] &= [c^{-1}, V Y^i 0^{k-i-1}]
\text{\hspace{.1in}(for $0 < i < k$, where $X = V^{t}$ and $Y = V^{ct}$)} 
\end{align}

\subsection{Proof of the Intermediate Conjecture
in type I}  

\begin{proposition}
\label{intermediate-type-I}
The Intermediate Conjecture holds in type I.
\end{proposition}

\begin{proof}
The point $(0, 0) \in V^{\Theta}$ is stabilized by the entire group 
$W \times \ZZ_{km}$, so that the map
$(0, 0) \mapsto [1, 0^k]$ is $W \times \ZZ_{km}$-equivariant.

The subset $ \{ (0, \omega^i), (\omega^i, 0) \,:\, 0 \leq i < km \} \subset V^{\Theta}$ is a single
$W \times \ZZ_{km}$-orbit of size $2km$.  
By Equations~\ref{s-locus-action}-\ref{g-locus-action},
the stabilizer of any point in this orbit
is the $m$-element subgroup of 
$W \times \ZZ_{km}$ generated by $(c, g^{-k})$.  On the other hand, the set of 
elements of $\Park^{NC}_W(k)$ of the form
$[w, V^i 0^{k-i}]$ ($i > 0$) has size $2km$, and 
by Equations~\ref{dihedral-g-action-one} and \ref{dihedral-g-action-two} the
 $W \times \ZZ_{kh}$-stabilizer of any of these parking functions is generated by
$(c, g^{-k})$.  Therefore, the assignment
$(1,0) \mapsto [1, V^k]$ generates a $W \times \ZZ_{km}$-equivariant bijection
\begin{equation*}
\{ (0, \omega^i), (\omega^i, 0) \,:\, 0 \leq i < km \} \xrightarrow{\sim}
\{ [1, V^i 0^{k-1}] \,:\, \text{$i > 0$} \}.
\end{equation*}

Our analysis of the $W \times \ZZ_{km}$-module structure of the remaining 
points $ \{ (\omega^i, \omega^j) \,:\, 0 \leq i, j < km \} \subset V^{\Theta}$ breaks up into
two cases depending on whether $m$ is even or odd.

\noindent
{\bf Case 1: $m$ is odd.}
Let $H \subset V$ be the reflecting hyperplane for $s \in W$.  Since $m$ is odd, 
the reflections $T \subset W$ form a single $m$-element 
$W$-orbit and their reflecting hyperplanes are
$H, cH, \dots, c^{m-1}H$.

The subset $\{ (\omega^i, \omega^j) \,:\, k | (j-i) \} \subset V^{\Theta}$ is stable under 
the action of $W \times \ZZ_{km}$.  Since $m$ is odd, 
Equations~\ref{s-locus-action}-\ref{g-locus-action} imply that
this subset is a single 
$W \times \ZZ_{km}$-orbit.  The stabilizer of the point $(1,1) \in V^{\Theta}$ is the
$2$-element subgroup generated by $(s,1) \in W \times \ZZ_{km}$.  On the other hand, 
Equations~\ref{dihedral-g-action-three} and \ref{dihedral-g-action-four} imply that the 
stabilizer of $[1, H^k] \in \Park^{NC}_W(k)$ is also the 
$2$-element subgroup generated by $(s,1) \in W \times \ZZ_{km}$.  Therefore, the map
$(1,1) \mapsto [1, H^k]$ induces a $W \times \ZZ_{km}$-set isomorphism
from $\{ (\omega^i, \omega^j) \,:\, k | (j-i) \} \subset V^{\Theta}$ onto its image in
$\Park^{NC}_W(k)$.  It can be shown that this image consists exactly of the 
$k$-$W$-parking functions of the form
$[w, X^i 0^{k-i}]$ ($i > 0, \dim(X) = 1$).  We therefore have a 
$W \times \ZZ_{km}$-equivariant bijection
\begin{equation*}
\{ (\omega^i, \omega^j) \,:\, k | (j-i) \} \xrightarrow{\sim}
\{ [w, X^i 0^{k-i}] \,:\, i > 0, \dim(X) = 1 \}.
\end{equation*}

To relate the $W \times \ZZ_{km}$-structure 
of the remaining locus elements
$\{ (\omega^i, \omega^j) \,:\, k \nmid (j-i) \} \subset V^{\Theta}$ to 
$\Park^{NC}_W(k)$, we consider two subcases depending on the parity of $k$.

\noindent
{\bf Subcase 1a: $k$ is odd.}
When $k$ is odd, the set $\{ (\omega^i, \omega^j) \,:\, k \nmid (j-i) \} \subset V^{\Theta}$ breaks
up into $\frac{k-1}{2}$ free $W \times \ZZ_{km}$-orbits.  
Using Equations~\ref{s-locus-action}-\ref{g-locus-action}, one can check that a
complete list of orbit representatives
is $(1, \omega), (1, \omega^2), \dots, (1, \omega^{\frac{k-1}{2}}) $.

On the other hand, we claim that the set of elements of $\Park^{NC}_W(k)$ of the form
$[w, V^i X^j 0^{k-i-j}]$ 
with $\dim(X) = 1$ and $i, j > 0$
breaks up into $\frac{k-1}{2}$ free $W \times \ZZ_{km}$-orbits.  
Since the one-dimensional flats 
 are precisely $H, cH, \dots, c^{m-1}H$, 
Equations~\ref{dihedral-g-action-five} and \ref{dihedral-g-action-six} imply that
the $k$-$W$-parking functions
\begin{equation}
\label{odd-orbit-representatives}
[1, V  H 0^{k-2}], 
[1, V H^2 0^{k-3} ], 
\dots,
[1, V H^{\frac{k-1}{2}} 0^{\frac{k-1}{2}}]
\end{equation}
generate
the $W \times \ZZ_{km}$-stable set 
$\{ [w, V^i X^j 0^{k-i-j}] \,:\, w \in W, \dim(X) = 1, i, j > 0 \}$.  
This set has $(2m)(m){k \choose 2} = (2km^2) \frac{k-1}{2}$ elements.
Since the group $W \times \ZZ_{km}$ has size $2km^2$, this forces
the orbits of each of the $k$-$W$-parking functions
listed in (\ref{odd-orbit-representatives}) to be free and distinct.

We conclude that the assignments 
\begin{equation*}
(1, \omega^i) \mapsto [1, V H^i 0^{k-i-1}], \hspace{.2in} 1 \leq i \leq \frac{k-1}{2}
\end{equation*}
generate a $W \times \ZZ_{km}$-equivariant bijection 
\begin{equation*}
\{ (\omega^i, \omega^j) \,:\, k \nmid (j-i) \} 
\xrightarrow{\sim}
\{ [w, V^i X^j 0^{k-i-j}] \,:\, w \in W, \dim(X) = 1, i, j > 0 \} 
\end{equation*}
Therefore, when $m$ and $k$ are odd, there exists a $W \times \ZZ_{km}$-equivariant
bijection $V^{\Theta} \xrightarrow{\sim} \Park^{NC}_W(k)$, completing the
proof of Proposition~\ref{intermediate-type-I} in this subcase.

\noindent
{\bf Subcase 1b: $k$ is even.}  When $k$ is even, 
Equations~\ref{s-locus-action}-\ref{g-locus-action} imply that
the set 
$\{ (\omega^i, \omega^j) \,:\, k \nmid (j-i) \} \subset V^{\Theta}$ breaks up into
$\frac{k}{2}$ orbits under the action of $W \times \ZZ_{km}$.  Of these orbits, $\frac{k}{2} - 1$
are free and are represented by the elements of the set
$\{ (1, \omega), (1, \omega^2), \dots, (1, \omega^{\frac{k}{2}-1}) \}$.  The remaining 
orbit is generated by $(1, \omega^{\frac{k}{2}})$ and has
stabilizer equal to the $2$-element subgroup of $W \times \ZZ_{km}$ generated by
$(s c^{\frac{m-1}{2}}, g^{\frac{km}{2}})$.  

On the other hand, we claim that the subset 
$\{ [w, V^iX^j0^{k-i-j}] \,:\, \text{$\dim(X) = 1$, $i, j > 0$} \} \subseteq
\Park^{NC}_W(k)$
breaks up into $\frac{k}{2}$ orbits, of which $\frac{k}{2} - 1$ are free
and one a representative with 
$W \times \ZZ_{km}$-stabilizer equal to the $2$-element subgroup generated 
by $(s c^{\frac{m-1}{2}}, g^{\frac{km}{2}})$.
Representatives for the 
$\frac{k}{2} - 1$ free orbits are
\begin{equation}
\label{subcase-one-b-reps}
[1, V H 0^{k-2}], 
[1, V H^2 0^{k-3}], 
\dots,
[1, V H^{\frac{k}{2} - 1} 0^{\frac{k}{2}}].
\end{equation}
It can be shown that the
orbit generated by $[1, V H^{\frac{k}{2}} 0^{\frac{k}{2} - 1}]$ 
has stabilizer equal to the $2$-element
subgroup of $W \times \ZZ_{km}$ generated by $(s c^{\frac{m-1}{2}}, g^{\frac{km}{2}})$.

As in Subcase 1a,
it follows that the assignments
\begin{align*}
(1, \omega^i) \mapsto [1, V H^i 0^{k-i-1}], \hspace{.2in} 1 \leq i \leq \frac{k}{2}
\end{align*}
extend to a $W \times \ZZ_{km}$-equivariant
bijection 
\begin{equation*}
\{ (\omega^i, \omega^j) \,:\, k \nmid (j-i) \}  \xrightarrow{\sim}
\{ [w, V^i X^j 0^{k-i-j}] \,:\, w \in W, \dim(X) = 1, i, j > 0 \} .
\end{equation*}
Thus, when $m$ is odd and $k$ is even, we get a 
$W \times \ZZ_{km}$-equivariant bijection
$V^{\Theta} \xrightarrow{\sim} \Park^{NC}_W(k)$, completing the proof of 
Proposition~\ref{intermediate-type-I} in this subcase.

\noindent
{\bf Case 2: $m$ is even.}  When $m$ is even, the hyperplanes in $\Cox(W)$ break up
into two $W$-orbits, each of size $\frac{m}{2}$.  If we denote the reflecting hyperplanes for
$s, t \in W$ by $H, H'$ (respectively), we can express these orbits as 
$\{H, cH, \dots, c^{\frac{m}{2} - 1} H \}$ and
$\{H', cH', \dots, c^{\frac{m}{2} - 1} H' \}$.

The subset $\{ (\omega^i, \omega^j) \,:\, 2k | (j-i) \} \subset V^{\Theta}$ is stable under the action of
$W \times \ZZ_{kh}$ and forms a single orbit.  Moreover, the $W \times \ZZ_{km}$-stabilizer
of $(1, 1)$ is the $4$-element subgroup of $W \times \ZZ_{km}$ generated by $(s,1)$ and
$(c^{\frac{m}{2}}, g^{\frac{km}{2}})$.  
 Equations~\ref{dihedral-g-action-three} and \ref{dihedral-g-action-four} imply that
the $W \times \ZZ_{km}$-stabilizer of $[1, H^k]$ is equal to the same $4$-element subgroup.
Therefore,
 the assignment
\begin{equation*}
(1,1) \mapsto [1, H^k]
\end{equation*}
induces a $W \times \ZZ_{km}$-equivariant bijection 
\begin{equation*}
\{ (\omega^i, \omega^j) \,:\, 2k | (j-i) \} \xrightarrow{\sim} \{[w, X^i 0^{k-i}] \,:\, i > 1, X \in W.H \}.
\end{equation*}
Using similar reasoning, we get a 
$W \times \ZZ_{km}$-equivariant bijection
\begin{equation*}
\{ (\omega^i, \omega^j) \,:\, 2k | (j-i+k) \} \xrightarrow{\sim}
\{[w, X^i 0^{k-i}] \,:\, i > 1, X \in W.H' \}.
\end{equation*}

Finally, the subset 
$\{ (\omega^i, \omega^j) \,:\, k \nmid (j-i) \} \subset V^{\Theta}$ is stable 
under the action of $W \times \ZZ_{km}$.  The stabilizer of any element
of this set is the $2$-element subgroup of $W \times \ZZ_{km}$ generated by
$(c^{\frac{m}{2}}, g^{\frac{km}{2}})$.  Thus, this subset consists of 
$k-1$ orbits, each of size $km^2$.  A complete list of representatives for these
orbits is $(1, \omega), (1, \omega^2), \dots, (1, \omega^{k-1})$.
We leave it to the reader to check that the assignments 
\begin{equation*}
(1, \omega^i) \mapsto [1, V H^i 0^{k-i}], \hspace{.2in} 1 \leq i \leq k - 1
\end{equation*}
generate a 
$W \times \ZZ_{km}$-equivariant bijection
\begin{equation*}
\{ (\omega^i, \omega^j) \,:\, k \nmid (j-i) \} 
\xrightarrow{\sim}
\{ [w, V^i X^j 0^{k-i-j}] \,:\, w \in W, \dim(X) = 1, i, j > 0 \}.
\end{equation*}
In conclusion, we have 
a $W \times \ZZ_{km}$-equivariant bijection $V^{\Theta} \rightarrow \Park^{NC}_W(k)$
when $m$ is even, which completes the proof of Proposition~\ref{intermediate-type-I}
in Case 2.
\end{proof}

\section{Type BC}
\label{Type BC}

In this section we let $W$ have type B$_n$ or C$_n$ for $n \geq 1$.  We prove the 
Intermediate Conjecture for $W$.

\subsection{Visualizing type BC}  
Let $\pm [n] := \{-n, -(n-1), \dots, -1, 1, \dots, n-1, n \}$.
We identify $W$ with the {\sf hyperoctohedral group}, i.e. the group of
permutations $w$ of $\pm [n]$ which satisfy
$w(-i) = -w(i)$ for all $i$.  The cycle decomposition of  permutation $w \in W$ 
consists of 
{\sf balanced cycles} whose underlying subsets of $\pm [n]$ satisfy $S = -S$ and
{\sf paired cycles} whose underlying subsets of $\pm [n]$ do not satisfy $S = -S$.  
If $C$ is a paired cycle of $w$, the cycle $-C$ formed by negating every element
of $C$ is also a paired cycle of $w$.
To economize
notation, we denote balanced cycles by square brackets and pairs of paired cycles by double 
parentheses.  Also, we will sometimes write $\bar{i}$ for $-i$.  For example, we have that
$((1,\bar{4},3)) [5, \bar{6}] = (1, -4, 3)(-1, 4, -3) (5, -6, -5, 6)$.

The set $S = \{s_1, s_2, \dots, s_n \}$ forms a set of simple reflections for $W$ where
$s_i = ((i, i+1))$ for $1 \leq i \leq n-1$ and $s_n = [n]$.  The set $T$ of all reflections consists
of $T = \{ ((i,j)) \,:\, i \neq \pm j \} \uplus \{ [i] \,:\, 1 \leq i \leq n \}$.
We fix the Coxeter element $c := [1, 2, \dots, n] \in W$.  We have that
$h = 2n$.

The reflection representation $V$ 
of $W$
is $\CC^n$.  For $1 \leq i \leq n-1$ the simple reflection $s_i$ swaps the $i^{th}$ and
$(i+1)^{st}$ standard basis vectors.  The simple reflection $s_n$ negates the $n^{th}$ standard basis
vector and fixes all other standard basis vectors.

The Coxeter arrangement of $W$ is
$\Cox(W) = \{ x_i \pm x_j = 0 \,:\, 1 \leq i < j \leq n \} \uplus \{ x_i = 0 \,:\, 1 \leq i \leq n \}$.
Given any flat $X \in \LLL$, we obtain a partition of $\pm [n]$ defined by $i \sim j$ 
if and only if the coordinate equality $x_i = x_j$ holds on $X$, where we adopt
the convention that 
$x_{-i} = - x_i$.  This gives a bijection between the intersection lattice $\LLL$ and the collection
of set partitions $\pi$ of $\pm [n]$ such that
\begin{itemize}
\item $\pi$ is {\sf centrally symmetric}: $-B$ is a block of $\pi$ whenever $B$ is, and
\item $\pi$ has at most one block $Z$ which satisfies $Z = -Z$.  Such a block is called
 the {\sf zero block} of $\pi$, if it exists.
\end{itemize}
The term `zero block' comes from the fact that the coordinate equality $x_i = 0$
holds on a flat $X$ if and only if $i$ belongs to the zero block of the set partition
 corresponding to $X$.

A flat in $\LLL$ is noncrossing (with respect to $c$) if any only
if its representation on $\DD^2$ is noncrossing when one labels the boundary of 
$\DD^2$ clockwise with $1, 2, \dots, n, \bar{1}, \bar{2}, \dots, \bar{n}$.  
Observe that any set partition of $\pm [n]$ which is centrally symmetric and noncrossing 
on $\DD^2$ can have at most one zero block.

After identifying $\pm [n]$ with $[2n]$ via 
$i \mapsto \begin{cases} i & \text{for $i > 0$} \\ n-i & \text{for $i < 0$} \end{cases} $,
we can apply the 
map $\nabla$ of Lemma~\ref{classical-noncrossing}
to a  $k$-multichain of centrally symmetric noncrossing partitions of
$\pm [n]$
to get a $k$-divisible noncrossing partition of $\pm [kn]$.  
Armstrong proved that
the partition so obtained is also centrally symmetric, and that $\nabla$ restricts to 
a bijection (also denoted $\nabla$) from length $k$ multichains of centrally 
symmetric noncrossing partitions of $\pm [n]$ to $k$-divisible 
centrally symmetric noncrossing partitions of $\pm [kn]$ \cite[Theorem 4.5.6]{Armstrong}.  
We explain how
$\nabla$ can be enriched to give a map (again denoted $\nabla$) defined
on $\Park^{NC}_W(k)$.

Identifying sequences of noncrossing flats with sequences of noncrossing partitions,
we can view a $k$-$W$-parking function $[w, X_1 \leq \dots \leq X_k]$ as 
a  $k$-multichain of centrally symmetric noncrossing partitions of $\pm [n]$ 
corresponding to 
the $X_i$ with the block $B$ of the partition 
$X_1$ labeled by the subset $w(B)$ of $\pm[n]$.
As in type A, one can apply the map $\nabla$ 
to the sequence $X_1 \leq \dots \leq X_k$
to get a $k$-divisible centrally symmetric
noncrossing partition 
$\pi$
of $[kn]$.  The underlying $k$-divisible partition of $\pi$ has a zero block if and only
if $X_1$ does \cite{Armstrong}.  Moreover, as in type A, the labeling of the blocks of
the partition $X_1$ induces a labeling $f$ of the blocks of $\pi$.  
More precisely, the restriction of $\pi$ to $\pm \{ ik + 1 \,:\, 1 \leq i \leq n \}$ is order
isomorphic to the partition  $X_1$ of $\pm [n]$.  Given any block
$B \in \pi$, we set $f(B) = w(B')$, where $B'$ is the block of 
 $X_1$ corresponding to $B$.
We define the map
$\nabla$ by
$\nabla: [w, X_1 \leq \dots \leq X_k] \mapsto (\pi, f)$.

The map $\nabla$ bijects $\Park^{NC}_W(k)$ with pairs $(\pi, f)$, where
$\pi$ is a centrally symmetric $k$-divisible noncrossing partition of $\pm [kn]$ and $f$
labels each block $B$ of $\pi$ with a subset $f(B)$ of $\pm [n]$ such that 
\begin{enumerate}
\item $|B| = k |f(B)|$ for every block $B \in \pi$,
\item $f(-B) = -f(B)$ for every block $B \in \pi$, and
\item $\pm [n] = \biguplus_{B \in \pi} f(B)$.
\end{enumerate}
Under the map $\nabla$, the action of
the cyclic group $\ZZ_{kh}$  on pairs $(\pi, f)$ is given by clockwise rotation on $\DD^2$ and  
 group $W$ acts  by label permutation.

 \begin{figure}
\includegraphics[scale=0.6]{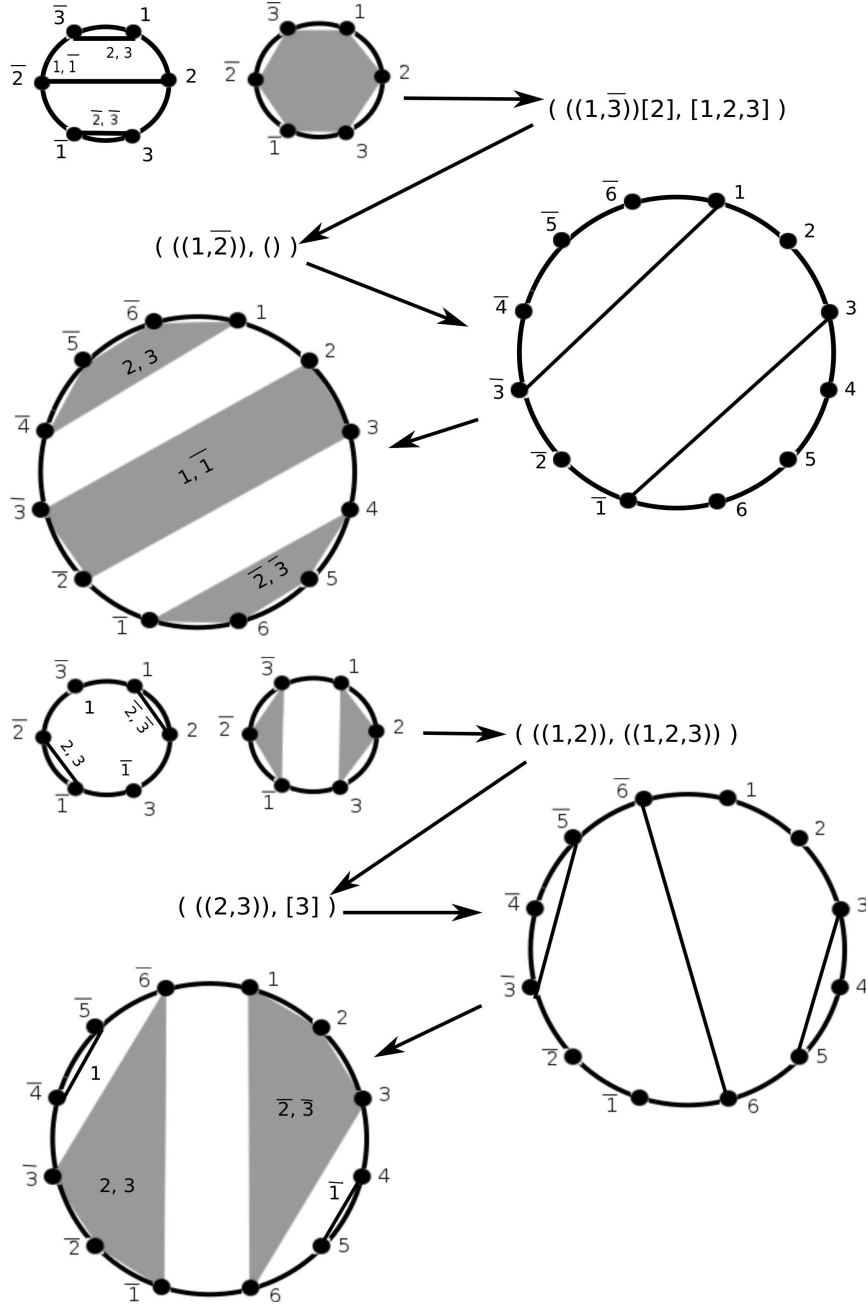}
\caption{Two $k$-$W$-noncrossing parking functions for $k = 2$ and $W$ of type B$_3$/C$_3$.}
\label{fig:Bfuss}
\end{figure}

Two examples of the map $\nabla$ on the level of parking functions 
 when $n = 3$ and $k = 2$ are shown in 
Figure~\ref{fig:Bfuss}.  
The upper example has a zero block while the lower example does not.

The upper example shows the $k$-$W$-noncrossing parking function
$[ ((1,3,\bar{2})), \{ 1, \bar{3} / 2, \bar{2} / \bar{1}, 3 \} \leq \{ 1, \bar{1}, 2, \bar{2}, 3, \bar{3} \}]$.
It is sent under $\nabla$ to the pair $(\pi, f)$, where 
$\pi = \{1, \bar{4}, \bar{5}, \bar{6} / 2, 3, \bar{2}, \bar{3} / 4, 5, 6, \bar{1} \}$ and
\begin{align*}
f: \{1, \bar{4}, \bar{5}, \bar{6} \} &\mapsto \{2, 3\}, \\
 \{2, 3, \bar{2}, \bar{3} \} &\mapsto \{1, \bar{1} \}, \\
\{4, 5, 6, \bar{1} \} &\mapsto \{ \bar{2}, \bar{3} \}.
\end{align*}

The lower example shows that $k$-$W$-noncrossing parking function
$[ ((1, \bar{3}))[2], \{ 1, 2 / 3 / \bar{1}, \bar{2} / \bar{3} \} \leq \{ 1, 2, 3 /  \bar{1}, \bar{2}, \bar{3} \}]$.
It is mapped under $\nabla$ to the pair $(\pi, f)$, where 
$\pi = \{ 1, 2, 3, 6 / 4, 5 /  \bar{1}, \bar{2}, \bar{3}, \bar{6} / \bar{4}, \bar{5} \}$ and
\begin{align*}
f: \{1, 2, 3, 6 \} &\mapsto \{ \bar{2}, \bar{3} \}, \\
 \{4, 5 \} &\mapsto \{ \bar{1} \}, \\
\{ \bar{1}, \bar{2}, \bar{3}, \bar{6} \} &\mapsto \{2, 3 \}, \\
 \{ \bar{4}, \bar{5} \} &\mapsto \{ 1 \}.
\end{align*}

\subsection{The Intermediate Conjecture in type BC}
We identify
$\CC[V]$ with $\CC[x_1, \dots, x_n]$ and choose 
$\theta_1 = x_1^{2nk+1}, \dots, \theta_n = x_n^{2nk+1}$ as our hsop of 
degree $2nk+1$ carrying $V^*$.  Then the map $x_i \mapsto \theta_i$ is $W$-equivariant, and
the zero locus cut out by $(\Theta - \xx)$ is
\begin{equation}
V^{\Theta} = \{ (v_1, \dots, v_n) \in \CC^n \,:\, \text{$v_i = 0$ or $v_i^{2nk} = 1$ for all $i$} \}.
\end{equation}
The $(2nk+1)^n = (kh+1)^n$ points in $V^{\Theta}$ are each cut out with multiplicity one.

We describe a map $\Phi: \Park^{NC}_W(k) \rightarrow V^{\Theta}$.  We represent an 
element of $\Park^{NC}_W(k)$ as a pair $(\pi, f)$ where $\pi$ is a $k$-divisible centrally
symmetric noncrossing partition of $\pm [kn]$ and $f$ is a labeling of the blocks of $\pi$ 
with subsets of $\pm [n]$ as 
above.  

Reiner \cite{Reiner} exhibited a bijection between centrally symmetric noncrossing partitions of
$\pm [n]$ and periodic parenthesizations of the doubly infinite sequence
\begin{equation*}
\dots 1, 2, \dots, n, \bar{1}, \bar{2}, \dots, \bar{n}, 1, 2, \dots, n, \bar{1}, \bar{2}, \dots, \bar{n}, \dots.
\end{equation*}
Given a periodic parenthesization, one constructs the nonzero blocks of the corresponding 
noncrossing partition by iteratively removing the innermost paired sets of parentheses.  
After all the parentheses are removed, the remaining symbols in the sequence form 
the zero block.
For example, when $n = 4$ the partition $\{ 1 / 2, \bar{4} / 3, \bar{3} / 4, \bar{2} / \bar{1} \}$ of
$\pm [4]$ corresponds to the parentization 
\begin{equation*}
\dots (1) 2) 3 (4 (\bar{1}) \bar{2}) \bar{3} (\bar{4} (1) 2) 3 (4 (\bar{1}) \bar{2}) \bar{3} (\bar{4} \dots.
\end{equation*}

  If $B$ is a nonzero block of a noncrossing centrally symmetric partition of 
$\pm[n]$, we say that $i \in \pm[n]$ {\sf opens} $B$ if $i$ occurs just to the right of the 
left parenthesis $``("$ corresponding to $B$.  In the above example, the 
block $\{1\}$ is opened by $1$, the block $\{2, \bar{4} \}$ is opened by $\bar{4}$, the block
$\{4, \bar{2} \}$ is opened by $4$, and the block $\bar{1}$ is opened by $1$.

We define $\phi(\pi,f) = (v_1, \dots, v_n)$ by 
\begin{equation*}
v_i = \begin{cases}
0 & \text{if $Z$ is the zero block of $\pi$ and $i \in f(Z)$} \\
\omega^j & \text{if $i \in f(B)$ for a nonzero block $B$ of $\pi$ opened by a positive index
$+j$} \\
-\omega^j & \text{if $i \in f(B)$ for a nonzero block $B$ of $\pi$ opened by a negative
index $-j$}
\end{cases}
\end{equation*}
For example, the parking function on the top of 
Figure~\ref{fig:Bfuss} is mapped under $\Phi$ to
$(0,\omega^{10}, \omega^{10})$ while the parking function on the bottom is mapped to
$(\omega^{10}, \omega^7, \omega^7)$.

\begin{proposition}
The Intermediate Conjecture holds when $W$ is of type $B_n$ or 
$C_n$.
\end{proposition}
\begin{proof}
It suffices to show that $\phi$ is a $W \times \ZZ_{kh}$-equivariant bijection from
 $\Park^{NC}_W(k)$ to $V^{\Theta}$.  To show that $\phi$ is a bijection, we construct its inverse.
 
 Define a map $\psi: V^{\Theta} \rightarrow \Park^{NC}_W(k)$ as follows.  We start with a 
 point $v = (v_1, \dots, v_n) \in V^{\Theta}$, so that for all $i$ either $v_i = 0$ or $v_i^{2nk} = 1$.
 We put left parentheses $``("$ just before each term $j$ and
 $\bar{j}$ in the doubly infinite sequence 
 \begin{equation*}
 \dots 1, 2, \dots, kn, \bar{1}, \bar{2}, \dots, \overline{kn}, 1, 2, \dots, kn, \dots 
 \end{equation*}
 such that $v_i = \pm \omega^j$ for some $1 \leq i \leq n$.  We insert right parentheses $``)"$ in 
 this sequence in the unique way such that for each left parenthesis which immediately
 precedes an index $j$, the number of symbols paired with this left parenthesis in
 the iterative pairing equals $km$, where $m$ is the number of $v_i$ such that
$v_i = \pm \omega^j$.
 This infinite cyclic parenthization determines a noncrossing 
 centrally symmetric partition $\pi$ of $\pm [kn]$.  By construction, the partition $\pi$ is 
 $k$-divisible.  We define a labelling $f$ of the blocks of $\pi$ by letting 
\begin{equation*}
f(B) = \{ i  \,:\, v_i = \omega^j \} \uplus \{ -i \,:\, v_i = - \omega^j \},
\end{equation*}
 where $B$ is a nonzero block of $\pi$ opened by $j$, and
 $f(Z) = \{ \pm i \,:\, v_i = 0 \}$, where $Z$ is the zero block of $\pi$ (if it exists).  By construction, 
 we have that $[n] = \biguplus_{B \in \pi} f(B)$, $|B| = k |f(B)|$, and $f(-B) = -f(B)$.  Therefore, 
 the pair $\psi(v) := (\pi, f)$ is the representation
 on $\DD^2$ of a $k$-$W$-noncrossing parking function.
 
 As an example of the map $\psi$, suppose that $n = 4$ and $k = 2$, so that 
 $\omega$ is a primitive $2*2*4^{th} = 16^{th}$ root of unity.
 Let $v = (v_1, v_2, v_3, v_4) = (\omega^4, 0, \omega^{12}, \omega^{5}) \in V^{\Theta}$.
 To construct $\psi(v) = (\pi, f)$, we first construct the parenthesization 
 of
 \begin{equation*}
 \dots 12345678 \bar{1} \bar{2} \bar{3} \bar{4} \bar{5} \bar{6} \bar{7} \bar{8} 
 12345678 \bar{1} \bar{2} \bar{3} \bar{4} \bar{5} \bar{6} \bar{7} \bar{8} \dots
 \end{equation*}
 corresponding
 to $\pi$.  Since $\omega^4$, $\omega^{12} = -\omega^4$, 
 and $\omega^5$ appear as coordinates of $v$, 
 we put a left parenthesis just before every $4, \bar{4}, 5$, and $\bar{5}$:
 \begin{equation*}
  \dots 123(4(5678 \bar{1} \bar{2} \bar{3} (\bar{4} (\bar{5} \bar{6} \bar{7} \bar{8} 
 123(4(5678 \bar{1} \bar{2} \bar{3} (\bar{4} (\bar{5} \bar{6} \bar{7} \bar{8} \dots.
 \end{equation*}
 Since the coordiates of $v$ have $2$ occurrences of $\pm \omega^4$ and 
 $1$ occurence of $\pm \omega^5$, we insert right parentheses in the unique way
 that each left parenthesis preceding a $4$ or $\bar{4}$ opens a block of 
 size $2k = 4$ and each left parenthesis preceding a $5$ or $\bar{5}$ opens a block
 of size $1k = 2$:
  \begin{equation*}
  \dots 1)23(4(56)78 \bar{1}) \bar{2} \bar{3} (\bar{4} (\bar{5} \bar{6}) \bar{7} \bar{8} 
 1)23(4(56)78 \bar{1}) \bar{2} \bar{3} (\bar{4} (\bar{5} \bar{6}) \bar{7} \bar{8} \dots.
 \end{equation*}
 We can use this parenthesization to read off the $2$-divisible centrally
 symmetric noncrossing partition $\pi$ of $\pm[8]$.  We have that
 $\pi = \{1, \bar{4}, \bar{7}, \bar{8} / 2, 3, \bar{2}, \bar{3} / 4, 7, 8, \bar{1} / 5, 6 / \bar{5}, \bar{6} \}$.
 Next, we determine the labelling $f$ of the blocks of $\pi$.  Since $v_1 = \omega^4$ and
 $v_3 = \omega^{4+8}$, we conclude the the block opened by $4$ is labelled with 
 $\{1, \bar{3} \}$.  That is, we have $f( \{4, 7, 8, \bar{1} \}) = \{1, \bar{3} \}$.  Similarly, we have 
 $f( \{ \bar{4}, \bar{7}, \bar{8}, 1 \}) = \{\bar{1}, 3 \}$, $f(\{5, 6\}) = \{\bar{4}\}$, and
 $f(\{ \bar{5}, \bar{6} \}) = \{4\}$.  Since $v_2 = 0$, we have that 
 $f(\{2,3,\bar{2},\bar{3}\}) = \{2, \bar{2} \}$.
 
We leave it to the reader to check that $\psi = \phi^{-1}$ and that the inverse bijections
$\phi$ and $\psi$ are $W \times \ZZ_{kh}$-equivariant.
\end{proof}

\section{Type D}
\label{Type D}

In this section we let $W$ be of type D$_n$ for $n \geq 3$.  We  recall the visualization
of the $k = 1$ case of $\Park^{NC}_W$ presented in \cite{ARR} and  explain how
one can use the Krattenthaler-M\"uller-Kim annular interpretation of Fuss noncrossing partitions
in type D to visualize elements of $\Park^{NC}_W(k)$.  Finally, we prove the Intermediate
Conjecture in type D.

\subsection{Visualizing type D noncrossing partitions and parking functions} We identify $W$ with the index two subgroup of the hyperoctohedral group
$W(B_n)$ consisting of permutations $w$ of the set $\pm [n]$ such that
$w(-i) = -w(i)$ for all $i$ and an even number of terms in the sequence
$w(1), w(2), \dots, w(n)$ are negative.  Since $W$ is a subgroup of $W(B_n)$, each
permutation in $W$ inherits a decomposition into balanced and paired cycles.

The set $S = \{ s_1, s_2, \dots, s_n \}$ forms a set of simple reflections for $W$, where
$s_i = ((i, i+1))$ for $1 \leq i \leq n-1$ and $s_n = ((n-1, -n))$.    The set $T$ of 
all reflections in $W$ is given by 
$T = \{ ((i, j)) \,:\, i, j \in \pm [n], i \neq \pm j \}$.  We make the choice of Coxeter element
$c = [1, 2, \dots, n-1] [n]$.  The Coxeter number is $h = 2n - 2$.

We use a topological model for type D noncrossing partitions on $\DD^2$ 
due to Athanasiadis and Reiner \cite{AthanasiadisR}.  Label the outer boundary of 
$\DD^2$ clockwise with $1, 2, \dots, n-1, \overline{1}, \overline{2}, \dots, \overline{n-1}$ and 
label the center of $\DD^2$ with both $n$ and $\overline{n}$.  Then, the relative interiors
of the convex hulls of the cycles of a noncrossing group element $w \in W$ do not 
intersect.  This sets up a bijection between noncrossing group elements $w \in W$ and 
centrally symmetric noncrossing set partitions drawn on $\DD^2$.  A {\sf zero block} of a centrally
symmetric noncrossing partition is a centrally symmetric block which contains
at least two boundary vertices.  
If it exists, a zero block is unique and necessarily contains both central 
labels $n$ and $\overline{n}$.  In order to recover the noncrossing group element
$w \in W$ associated to a centrally symmetric noncrossing partition $\pi$ drawn on 
$\DD^2$, we let $w$ be the permutation in $W$ with
\begin{itemize}
\item paired cycles given by the nonzero blocks of $\pi$, taken in the cyclic 
order
\begin{equation*}
\dots, 1, 2, 3, \dots, n, \overline{1}, \overline{2}, \dots, \overline{n}, 1, 2, \dots, n, \dots,
\end{equation*}
and
\item balanced cycles given by $[n]$ together with the elements in $Z - \{ n, \overline{n} \}$,
taken in the same cyclic order, 
if $\pi$ has a zero block $Z$.
\end{itemize}
Athanasiadis and Reiner prove that these maps are inverse bijections.

In \cite{ARR}, this structure is enriched to give a topological model for 
$W$-noncrossing parking functions.  If $[w, X] \in \Park^{NC}_W$, we draw 
the noncrossing partition $X$ on $\DD^2$, which
yields a figure $\pi$.  Given any block $B$ of $\pi$, we have that
$w(B) \subseteq \pm [n]$.  We define $f$ to be the labeling of the blocks of $\pi$ which
sends $B$ to $w(B)$.  The map $[w, X] \mapsto (\pi, f)$ sets up a bijection between
$\Park^{NC}_W$ and pairs $(\pi, f)$ such that $\pi$ is a type D noncrossing partition of
$\pm [n]$ and the following four conditions hold.
\begin{enumerate}
\item For any block $B \in \pi$ we have $|B| = |f(B)|$.
\item For any block $B \in \pi$ we have $f(-B) = -f(B)$.
\item We have $\pm [n] = \biguplus_{B \in \pi} f(B)$.
\item There exists a bijection $\phi: \pm[n] \rightarrow \pm[n]$ such that $\phi(B) = f(B)$ for 
all $B \in \pi$, $\phi(-i) = - \phi(i)$ for all $i \in \pm[n]$, and $\phi$ sends an even number of 
positive indices to negative indices.
\end{enumerate}
The group $W$ acts
on set of pairs $(\pi, f)$ by permuting the labels and the group $\ZZ_h$ acts by simultaneous
$(2n - 2)$-fold rotation on the outer boundary of $\DD^2$ and swapping the center labels
$n$ and $-n$.

\subsection{Visualizing type D $k$-noncrossing partitions and $k$-parking functions}
Visualizing the Fuss analogs of type D noncrossing partitions differs significantly from 
the topological models presented so far in that our partitions are drawn on the {\bf annulus}
$\AAA^2$ rather than on the disc $\DD^2$.  This construction was outlined by 
Krattenthaler and M\"uller \cite{KrattenthalerMuller} and a gap in the original 
Krattenthaler-M\"uller construction was filled by Kim \cite{Kim}.

Let $k \geq 1$ be a Fuss parameter.  Given any element $w \in W \subset \symm_{\pm [n]}$ and
any $1 \leq m \leq k$, we associate an element
$\tau_{k,m}(w) \in \symm_{\pm [kn] }$ by letting the cycles of $\tau_{k,m}(w)$ be obtained
from the cycles of $w$ by replacing $i$ with $m + (i-1)k$ 
and $\overline{i}$ with $\overline{m + (i-1)k}$ for all $i > 0$. 

The functions $\tau_{k,m}$ can be used to draw elements of $NC^k(W)$ on the annulus.
Given $(w_1 \leq \dots \leq w_k) \in NC^k(W)$, we define
$\tau(w_1 \leq \dots \leq w_k)$ to be the permutation
in $\symm_{\pm [kn]}$ given by
\begin{align*}
\tau(w_1 \leq \dots \leq w_k) = [1, &2, \dots, k(n-1)] [k(n-1) + 1, k(n-1) + 2, \dots, kn]  \\ &\circ
\tau_{k,1}(w_1^{-1} w_2)^{-1} \cdots \tau_{k,k-1}(w_{k-1}^{-1} w_k)^{-1}
\tau_{k,k}(w_k^{-1} c)^{-1}.
\end{align*}
The cycles of the permutation $\tau(w_1 \leq \dots \leq w_k)$ form can be viewed as a noncrossing
partition on the annulus $\AAA^2$ as follows.

Label the outer boundary of $\AAA^2$ clockwise with 
\begin{equation*}
1, 2, \dots, k(n-1), \overline{1}, \overline{2}, \dots, \overline{k(n-1)}.
\end{equation*}
Label the inner boundary of $\AAA^2$ {\em counterclockwise} with
\begin{equation*}
k(n-1) + 1, k(n-1) + 2, \dots, kn, \overline{k(n-1) + 1}, \overline{k(n-1) + 2}, \dots, \overline{kn}.
\end{equation*}
A partition $\pi$ of $\pm [kn]$ is said to be {\sf $k$-D-noncrossing} if the following five conditions
are satisfied.
\begin{enumerate}
\item If $B \in \pi$, then $-B \in \pi$ (i.e., $\pi$ is centrally symmetric).
\item The vertices on $\AAA^2$ corresponding to the blocks of $\pi$ can be connected
by simple closed curves contained in $\AAA^2$ such that the curves corresponding to distinct
blocks do not cross.
\item If $\pi$ contains a block $Z$ satisfying $Z = -Z$ (i.e., a centrally symmetric block),
then $Z$ must contain all of the vertices on the inner boundary of $\AAA^2$ 
together with at least $2k$
additional
vertices on the outer boundary of $\AAA^2$.
\item When the vertices of any block of $\pi$ are read in the cyclic order coming from the 
embedding in $\AAA^2$, the absolute values of consecutive vertices represent 
consecutive residue classes modulo $k$.
\item If every block of $\pi$ consists entirely of vertices on the outer boundary of $\AAA^2$
or entirely of vertices on the inner boundary of $\AAA^2$, then the inner blocks of $\pi$
are determined by the outer blocks of $\pi$ in the following way.  A block $B \in \pi$ on the outer
boundary is called {\sf visible} if it can be connected to the inner boundary by a curve which
does not cross any other blocks on the outer boundary.  Let $B$ be a visible block and let
$b \in B$ be the last vertex read in clockwise order.  Then the two inner blocks of $\pi$ are 
the unique pair of two consecutive $k$-element subsets of the inner boundary such that
one of the subsets ends in $a$ (in counterclockwise order) where $a > 0$ and
$a \equiv b$ (mod $k$).
\end{enumerate}
Condition 4 implies and is stronger than the condition that $\pi$ is $k$-divisible.  
Condition 5 was originally missed by Krattenthaler and M\"uller;  its necessity was observed
by Kim.

Given any partition $\pi$ of $\pm [kn]$ and any block $B$ of $\pi$, call $B$ {\sf inner} if 
$B$ only contains vertices on the inner boundary of $\AAA^2$, {\sf outer} if $B$ only contains
vertices on the outer boundary of $\AAA^2$, {\sf zero} if $B$ satisfies $B = -B$, and
{\sf annular} if none of the preceding conditions hold.  
If $\pi$ is $k$-D-noncrossing, we observe the following trichotomy.
\begin{itemize}
\item The partition $\pi$ contains a zero block and any other blocks of $\pi$ are outer.
\item The partition $\pi$ contains inner blocks and all other blocks of $\pi$ are outer.
\item The partition $\pi$ contains annular blocks and any other blocks of $\pi$ are outer.
\end{itemize}

\begin{theorem} 
\label{tau-theorem}
(\cite{KrattenthalerMuller} \cite{Kim})
Suppose $W$ has type D$_n$ for $n > 2$.
There is a bijection from the set $NC^k(W)$ of $k$-$W$-noncrossing partitions 
to the set of $k$-$D$-noncrossing partitions of $\pm [kn]$ given by sending a tuple
$(w_1 \leq \dots \leq w_k)$ to the partition $\pi$ of $\pm [kn]$ whose blocks are the cycles
of $\tau(w_1 \leq \dots \leq w_k)$.  Moreover, if $\pi_1$ is the partition of $\pm [kn]$ corresponding
to $w_1$, then
\begin{itemize}
\item the partition $\pi$ has a zero block if and only if $\pi_1$ does, in which case the zero block of 
$\pi$ has size $k$ times the size of the zero block of $\pi_1$, 
\item the partition $\pi$ has inner blocks if and only if the singletons $\{ n \}$ and $\{ -n \}$ are
blocks of $\pi_1$, and
\item if $\pi_1$ has $c_i$ nonzero blocks of size $i$, then $\pi$ has $c_i$ nonzero blocks
of size $ki$.
\end{itemize}
\end{theorem}

It can be shown from the definition of $\tau$ that the distinguished generator $g$ of $\ZZ_{kh}$
acts by simultaneous clockwise rotation on the outer annulus and counterclockwise rotation
on the inner annulus.

If $\pi$ is a $k$-D-noncrossing partition of 
$\pm [kn]$ and $B$ is an outer or annular block of $\pi$, let the {\sf opener} of 
$B$ be the first
element $b \in B$ in the clockwise order on the outer vertices of $\AAA^2$.  
If we restrict $\pi$ to the vertices on the outer boundary, we obtain a 
$k$-divisible centrally symmetric partition of $\pm [k(n-1)]$ and this definition agrees
with the one in Section~\ref{Type BC}.

To prove the Intermediate Conjecture in type D it will be useful to determine how uniquely a 
$k$-D-noncrossing partition $\pi$ is determined by its set of openers, the size of the blocks
which they open, and the size of its zero block.  Define an {\sf antipode} on the set of 
$k$-D-noncrossing partitions by letting $A(\pi)$ be the partition obtained from $\pi$ by 
swapping the signs of the vertices on the inner boundary of $\AAA^2$ (or, equivalently,
rotating the inner boundary by 180 degrees).  Figure~\ref{fig:antipode} shows an
example of the antipodal involution when $n = 4$ and $k = 3$.  We have that
$A(\pi) \neq \pi$ if and only if $\pi$ contains annular blocks.  It is clear that the openers of $\pi$,
the size of the blocks which they open, and the size of the zero block of $\pi$ are preserved
by $A$, so that this data can at most determine $\pi$ up to antipode.  The next lemma states that
this data exactly determines $\pi$ up to antipode.

\begin{figure}
\includegraphics[scale=0.4]{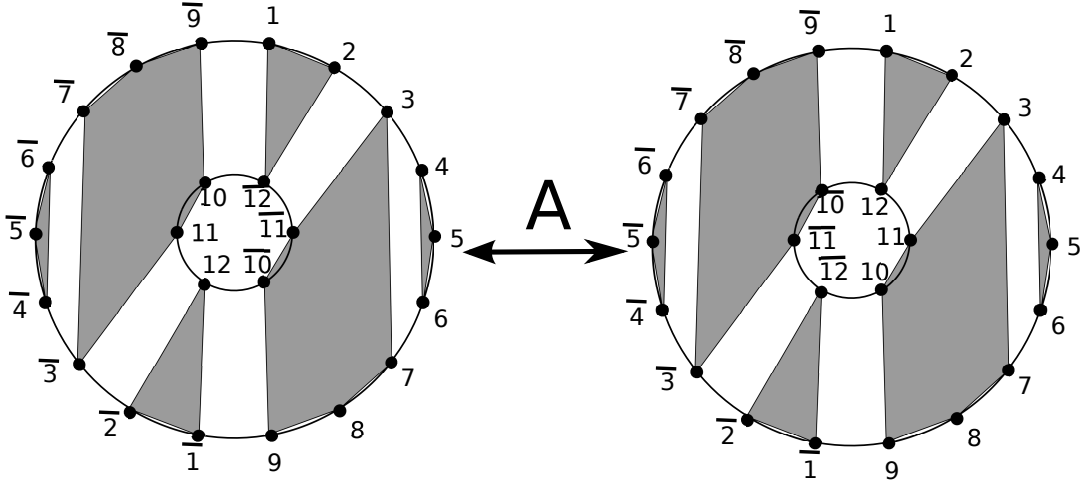}
\caption{The antipode when $n = 4$ and $k = 3$}
\label{fig:antipode}
\end{figure}

\begin{lemma}
\label{recovery}
Let $\pi$ be a $k$-D-noncrossing partition.  The partition $\pi$ is uniquely determined 
up to antipode by
\begin{itemize}
\item the size of the zero block of $\pi$ (and, in particular, its presence or absence), 
\item the presence or absence of inner blocks of $\pi$,
\item the subset of  $\pm [k(n-1)]$ of openers $i$ of the outer and annular
blocks of $\pi$, and
\item the number $b_i$ (necessarily divisible by $k$) of elements in the block of $\pi$ 
which is opened by $i$.
\end{itemize}

Conversely, any sequence $(i_1, \dots, i_s )$ in $[k(n-1)]$ together with any sequence of 
integers $(b_1, \dots, b_s)$ divisible by $k$ and
a number $z$ divisible by $2k$ which
satisfy $2b_1 + \cdots + 2b_s + z = 2kn$ uniquely determine up to antipode a 
$k$-D-noncrossing partition $\pi$ with 
openers $\pm i_1, \dots, \pm i_s$ opening blocks of size 
 $b_1, \dots, b_s$ which has 
 \begin{itemize}
 \item no zero block if $z = 0$,
 \item two inner blocks if $z = 2k$, or
 \item a zero block of size $z$ if $z > 2k$.
 \end{itemize}
\end{lemma}

\begin{proof}
Let $\pi$ be a $k$-D-noncrossing partition of $\pm [kn]$.  We show that $\pi$ is determined
up to antipode by the given information.

Suppose first that $\pi$ has a zero block $Z$.  Then, $k | |Z|$ and 
$Z$ contains all of the $2k$ vertices on the inner boundary of $\AAA^2$.  
Let $\overline{\pi}$ be the $k$-$W(B_{n-1})$-noncrossing partition formed by restricting
$\pi$ to the outer boundary.  The openers of $\overline{\pi}$ are  determined from the openers
of $\pi$, as well as the size of the blocks they open.  Also, the size of the zero block of
$\overline{\pi}$ is $|Z| - 2k$.  By Reiner's bijection \cite{Reiner}, 
type B noncrossing partitions can be recovered from
their openers, the size of the blocks which they open, and the size of their zero blocks.  It follows
that $\overline{\pi}$ can be recovered from the data above and $\pi$ can be recovered
from $\overline{\pi}$.

Next suppose that $\pi$ has two inner blocks $I, -I \in \pi$.  Again consider the restriction
$\overline{\pi} = \pi - \{ I, -I \}$ of $\pi$ to the outer boundary.  By the same reasoning as before,
the restriction $\overline{\pi}$ can be recovered from the above data.  Condition 5 in the definition
of $k$-D-noncrossing partitions implies that the inner blocks of $\pi$ are determined by the 
outer blocks of $\pi$, so that $\pi$ is determined by $\overline{\pi}$.

Now assume that every block of $\pi$ is outer or annular.  We focus on the sequence 
$( i_1, \dots, i_s )$ of openers together with the sizes of the blocks
$(b_1, \dots, b_s)$ that they open.
The block $B_r \in \pi$ 
opened by $i_r$ has at least $b_r - k$ and at most $b_r$ vertices on the outer 
boundary of $\AAA^2$ and the absolute values of the
entries in $B_r$ must represent consecutive residue
classes modulo $k$ when read clockwise around the block.  
Moreover, if $i$ and $j$ are two consecutive outer boundary vertices in $B_r$, then 
the vertices between $i$ and $j$ must be a union of outer blocks of $\pi$.  Finally, if $i$ and 
$j$ are two consecutive vertices in $B_r$ with $i$ on the outer boundary and $j$ on the inner
boundary, we have that $j$ is uniquely determined from $i$ by the condition on residue 
classes plus a choice of elements in an antipodal equivalence class.
Working in increasing order of block sizes, 
it can be shown that the blocks $B_r$ of $\pi$ can be reconstructed.  
This completes
the proof that $\pi$ can be reconstructed from $(i_1, \dots, i_s)$ and $(b_1, \dots, b_s)$ 
up to antipodal equivalence. 

The proof that 
any triple $(i_1, \dots, i_s)$, $(b_1, \dots, b_s)$, and $z$ as above give rise to a unique 
$k$-$D$-noncrossing partition with the stated properties is similar and left to the reader.
\end{proof}

We can visualize elements of $\Park^{NC}_W(k)$ on $\AAA^2$,
but this visualization is less canonical than those in types ABC.  
\begin{lemma}
\label{d-visualization}
Let $W$ have type D$_n$ for $n \geq 3$ and let $k \geq 1$.  Define a 
$W \times \ZZ_{kh}$-set $M_n(k)$ as follows.  As a set $M_n(k)$ consists
of the pairs $(\pi, f)$ where $\pi$ is a $k$-D-noncrossing partition of $\pm [kn]$ and 
$f: B \mapsto f(B)$ is a labeling of the blocks of $\pi$ with subsets of $\pm [n]$ such that
the following conditions are satisfied.
\begin{enumerate}
\item For any $B \in \pi$ we have $|B| = k |f(B)|$.
\item We have that $\pm [n] = \biguplus_{B \in \pi} f(B)$.
\item For any $B \in \pi$ we have $f(-B) = -f(B)$.
\item Let $S := \{ i, \overline{i} \,:\, 1 \leq i \leq kn, i \equiv 1$ (mod $k$)$\}$.  There
exists a bijection
$\phi: S \rightarrow \pm [n]$ which satisfies $\phi(-i) = - \phi(i)$ for all indices $i$ and 
restricts to a bijection $B \cap S \rightarrow f(B)$ for all blocks $B \in \pi$, we have that
$\phi$ sends an even number of positive indices to negative indices.
\end{enumerate}
The $W \times \ZZ_{kh}$-set structure on $M_n(k)$ is given by letting $W$ act on the labels
and letting $\ZZ_{kh}$ act by clockwise rotation on the outer boundary of $\AAA^2$
and counterclockwise rotation on the inner boundary of $\AAA^2$.

There is a $W \times \ZZ_{kh}$-equivariant bijection $M_n(k) \cong \Park^{NC}_W(k)$.
\end{lemma}

\begin{proof}
By 
Theorem~\ref{tau-theorem}, the map $\tau$ induces a bijection between 
multichains $(\pi_1 \leq \dots \leq \pi_k)$ of type D noncrossing partitions of $\pm [n]$
and $k$-D-noncrossing partitions of $\pm [kn]$. 
We can visualize a $k$-$W$-parking function $[w, \pi_1 \leq \dots \leq \pi_k]$ as a 
multichain of type $D$ noncrossing partitions $(\pi_1 \leq \dots \leq \pi_k)$ with the blocks 
of $\pi_1$ labeled via $B \mapsto w(B)$. 

As in the case of types ABC, the idea is to extend the bijection 
$(\pi_1 \leq \dots \leq \pi_k) \mapsto \pi$ to a $W \times \ZZ_{kh}$-equivariant
isomorphism $\Park^{NC}_W(k) \rightarrow M_n(k)$ by using the labels of the block of 
$\pi_1$ to label the blocks of $\pi$.  However, unlike in types ABC, it is {\em not} the case that 
the set partition of $S$ (from Condition 4 of the statement) induced by $\pi$ is order isomorphic
to $\pi_1$.  As a result, there is not a `canonical' choice for labeling the blocks of $\pi$.

Let $\mathcal{O}_1, \dots, \mathcal{O}_t$ be the complete set of $W \times \ZZ_{kh}$-orbits
in $\Park^{NC}_W(k)$ and for $1 \leq i \leq t$ let
$[1, \pi_1^{(i)} \leq \dots \leq \pi_k^{(i)}] \in \Park^{NC}_W(k)$ be a representative of
$\mathcal{O}_i$.  For $1 \leq i \leq t$ let $\pi^{(i)}$ be the $k$-D-noncrossing partition of 
$\pm [kn]$ corresponding to the multichain $\pi_1 \leq \dots \leq \pi_k$.  
By Theorem~\ref{tau-theorem}, the partition $\pi^{(i)}$ has a zero block (respectively inner blocks)
if and only if $\pi^{(i)}_1$ has a zero block (respectively singletons $\{n\}$ and $\{\overline{n}\}$).
Also, the number of blocks of $\pi^{(i)}$ of size $kb$ equals the number of blocks 
of $\pi^{(i)}_1$ of size $b$ for all $b$.  It follows that we can assign the parking function
$[1, \pi_1^{(i)} \leq \dots \leq \pi^{(i)}_k]$ to a pair $(\pi^{(i)}, f^{(i)}) \in M_n(k)$ where
the set of label sets $\{ f^{(i)}(B) \,:\, B \in \pi^{(i)} \}$ is precisely the set partition $\pi^{(i)}_1$
of $\pm [n]$ (there is in general more than one 
function $f^{(i)}$ such that we have
$(\pi^{(i)}, f^{(i)}) \in M_n(k)$ and
$\{ f^{(i)}(B) \,:\, B \in \pi^{(i)} \} = \pi_1^{(i)}$; arbitrarily choose one such
$f^{(i)}$ for each $1 \leq i \leq t$).  

A straightforward analysis of stabilizers
shows that the assignment $[1, \pi^{(i)}_1 \leq \dots \leq \pi^{(i)}_k] \mapsto (\pi^{(i)}, f^{(i)})$
gives a $W \times \ZZ_{kh}$-equivariant embedding $\mathcal{O}_i \hookrightarrow 
M_n(k)$ for all $i$.  These embeddings patch together to give a $W \times \ZZ_{kh}$-equivariant
map $\alpha: \Park^{NC}_W(k) \rightarrow M_n(k)$.  The map $\alpha$ is injective because 
the map $(\pi_1 \leq \dots \leq \pi_k) \mapsto \pi$ induced by Theorem~\ref{tau-theorem}
is injective.  To check that $\alpha$ is surjective, let $(\pi, f) \in M_n(k)$ be arbitrary.  By the 
definition of $\alpha$ and the fact that the map induced by Theorem~\ref{tau-theorem} is
bijective, there exists a function $f'$ such that $(\pi, f') \in M_n(k)$ is in the image of $\alpha$.
But by the definition of $M_n(k)$, there exists $w \in W$ such that 
$w.(\pi, f') = (\pi, f)$.  That is, the pair $(\pi, f)$ is in the image of $\alpha$.  It follows that $\alpha$ gives
an isomorphism as required in the statement.
\end{proof}

We will use the permutation module $M_n(k)$ rather than the permutation module
$\Park^{NC}_W(k)$ to prove the Intermediate Conjecture in type D.

\begin{example}
Let $n = 4$ and $k = 3$.  Three elements of $M_4(3)$ are depicted in Figure~\ref{fig:Dfuss}.

The upper left of Figure~\ref{fig:Dfuss} shows an element $(\pi, f) \in M_4(3)$ where $\pi$ has a 
zero block.  We have that 
$\pi = \{ \pm B, Z \}$, were $B = \{2, 3, 4, 5, 6, 7 \}$ and 
$Z = \{1, \overline{1}, 8, \overline{8}, 9, \overline{9} \}$.  The labeling $f$ of the blocks of $\pi$ is given 
by 
\begin{align*}
f(\pm B) &= \pm \{ \overline{1}, 3 \}, \\
f(Z) &= \{2, 4, \overline{2}, \overline{4} \}.
\end{align*}
It is clear that the pair $(\pi, f)$ satisfies Conditions 1-3 of Lemma~\ref{d-visualization}.  To see that 
$(\pi, f)$ satisfies Condition 5, observe that the set $S$ is given by 
$S = \pm \{1, 4, 7, 10 \}$ and that a choice of bijection $\phi$ as in Condition 4 satisfies 
$\phi(1) = 2, \phi(4) = \overline{1}, \phi(7) = \overline{2},$ and $\phi(10) = 4$, so that $\phi$ sends an
even number of positive indices to negative indices.  It follows that $(\pi, f)$ satisfies Condition 4
and $(\pi, f) \in M_4(3)$.  More generally, it can be shown that whenever $\pi$ has a zero block, any
pair $(\pi, f)$ satisfying Conditions 1-3 of Lemma~\ref{d-visualization} must also satisfy 
Condition 4 since the parity of the number of positive elements sent to negative elements 
by $\phi$ can be swapped by changing the image of $\phi$ on zero block elements.

The upper right of Figure~\ref{fig:Dfuss} shows an element $(\pi', f') \in M_4(3)$ where
$\pi'$ has annular blocks.  We have that $\pi' = \{ \pm B_1, \pm B_2, \pm B_3 \}$, where
$B_1 = \{1, 2, \overline{8}, \overline{9}, 10, \overline{12} \}, 
B_2 = \{3, 4, \overline{11} \},$ and
$B_3 = \{5, 6, 7\}$.  The labeling $f'$ of the blocks of $\pi'$ is given by
\begin{align*}
f'(\pm B_1) &= \pm \{2, 4\}, \\
f'(\pm B_2) &= \pm \{\overline{1}\}, \\
f'(\pm B_3) &= \pm \{\overline{3}\}.
\end{align*}
As before, it is clear that $(\pi', f')$ satisfies Conditions 1-3 of Lemma~\ref{d-visualization}.  To
see that $(\pi', f')$ satisfies Condition 4, observe that there exists a bijection $\phi$ as in 
Condition 4 which satisfies $\phi(1) = 2, \phi(4) = \overline{1}, \phi(7) = \overline{3},$
and $\phi(10) = 4$, so that $\phi$ sends an even number of positive indices to negative indices.  
In general, if $\pi'$ has annular blocks and $(\pi', f') \in M_n(k)$, then any bijection $\phi$
as in Condition 4 sends an even number of positive indices to negative indices.  However, there
exist examples of pairs $(\pi', f')$ where $\pi'$ has annular blocks which satisfy Conditions 1-3
but not Condition 4 (see Case 3 in the Proof of Proposition~\ref{intermediate-type-d}).

The bottom of Figure~\ref{fig:Dfuss} shows an element $(\pi'', f'') \in M_4(3)$ where $\pi''$ has 
inner blocks.  We have that $\pi'' = \{ \pm B_1, \pm B_2, \pm I \}$, where
$B_1 = \{1, 2, 3, 4, \overline{8}, \overline{9} \}, B_2 = \{5, 6, 7\},$ and 
$I = \{ \overline{10}, \overline{11}, 12 \}$.  Observe that the inner blocks $\pm I$ of $\pi''$ satisfy 
Condition 5 in the definition of $k$-D-noncrossing partitions.  The labeling $f''$ of the 
blocks of $\pi''$ is given by
\begin{align*}
f''(\pm B_1) &= \pm \{ \overline{1}, 2 \}, \\
f''(\pm B_2) &= \pm \{ \overline{3} \}, \\
f''(\pm I) &= \pm \{ 4 \}.
\end{align*}
It is again clear that the pair $(\pi'', f'')$ satisfies Conditions 1-3 of Lemma~\ref{d-visualization}.
There exists a bijection $\phi$ as in Condition 4 which satisfies 
$\phi(1) = \overline{1}, \phi(4) = 2, \phi(7) = \overline{3},$ and  $\phi(10) = 4$.   As with
the case of annular blocks, if $\pi''$ has inner blocks and $(\pi'', f'') \in M_n(k)$, then any 
bijection $\phi$ as in Condition 4 sends an even number of positive indices to negative indices.
Also as with the case of annular blocks, there exist pairs $(\pi'', f'')$ where
$\pi''$ has inner blocks which satisfy 
Conditions 1-3 but not Condition 4 (see Case 2 in the Proof of 
Proposition~\ref{intermediate-type-d}).
\end{example}

 \begin{figure}
\includegraphics[scale=0.7]{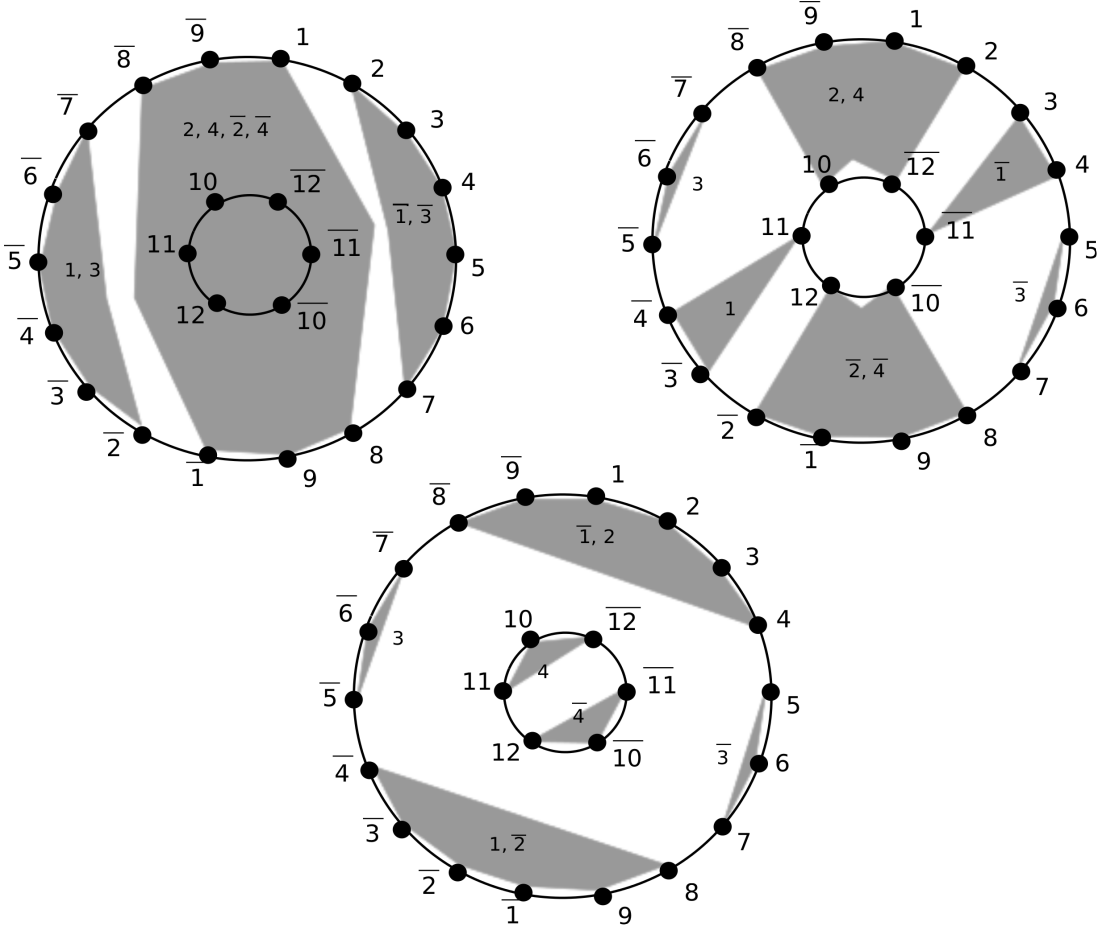}
\caption{Three $k$-$W$-noncrossing parking functions for $W$ of type D$_4$ for $n = 3$}
\label{fig:Dfuss}
\end{figure}

\subsection{The Intermediate Conjecture in type D}

\begin{proposition}
\label{intermediate-type-d}
The Intermediate Conjecture holds when $W$ is of type D$_n$.
\end{proposition}
\begin{proof}
We identify $\CC[V]$ with $\CC[x_1, \dots, x_n]$.  Then 
$\theta_1 = x_1^{kh+1}, \dots, \theta_n = x_n^{kh+1}$ is an hsop of degree $kh+1$
carrying $V^*$, and the map $x_i \mapsto \theta_i$ is $W$-equivariant.  The zero locus 
$V^{\Theta}$ cut out by $(\Theta - \xx)$ is 
\begin{equation}
V^{\Theta} = \{ (v_1, \dots, v_n) \in \CC^n \,:\, \text{$v_i = 0$ or $v_i^{2(n-1)k} = 1$ for all $i$} \}.
\end{equation}
We have that $|V^{\Theta}| = (2(n-1)k + 1)^n = (kh+1)^n$, and each point in 
$V^{\Theta}$ is cut out by $(\Theta - \xx)$ with multiplicity one.

Let $M_n(k)$ be the $W \times \ZZ_{kh}$-module from Lemma~\ref{d-visualization}.
By Lemma~\ref{d-visualization},
it suffices to construct a $W \times \ZZ_{kh}$-equivariant bijection
$\phi: M_n(k) \rightarrow \ZZ_{kh}$.
Let $(\pi, f) \in M_n(k)$.  We define
 the point $\phi(\pi, f) = (v_1, \dots, v_n)$ with coordinates given by
 \begin{equation*}
 v_i = \begin{cases}
 0 &  \text{if $Z \in \pi$ is a zero block and $i \in f(Z)$} \\
 0 &  \text{if $I \in \pi$ is an inner block and $i \in f(I)$} \\
 \omega^j &  \text{if $B \in \pi$ is an annular  or outer block opened by $j > 0$ and
 $i \in f(B)$} \\
- \omega^j &  \text{if $B \in \pi$ is an annular or outer block opened by
 $j < 0$ and $i \in f(B)$.}
 \end{cases}
 \end{equation*}
 For example, consider the three elements of $M_4(3)$ shown in Figure~\ref{fig:Dfuss}.
 Starting at the top left and reading clockwise, the images of these pairs $(\pi, f)$ under 
 $\phi$ are $(\omega^{10}, 0, \omega^2, 0), 
 (\omega^3, \omega^{17}, \omega^{14}, \omega^{17}),$ and
$(\omega^8, \omega^{17}, \omega^{14}, 0)$.
 
 To show that $\phi$ is a bijection, we construct the inverse map 
 $\psi: V^{\Theta} \rightarrow M_n(k)$.  Let $v = (v_1, \dots, v_n) \in V^{\Theta}$.  The 
 construction of $\psi(v) = (\pi, f)$ depends on the number of times zero appears as a coordinate
 of $v$.
 
 \noindent
 {\bf Case 1. Zero appears more than once as a coordinate of $v$.}  
 Let $z > 1$ be the number of times zero appears as a coordinate of $v$.  For 
 $i \in \pm[k(n-1)]$, let $b_i$ be the total number of times either $\omega^i$ or $-\omega^i$
 appears as a coordinate of $v$.  By Lemma~\ref{recovery}, there exists a unique $k$-D-noncrossing
 partition $\pi$ of $\pm [kn]$ which has a zero block of size $kz$ and openers 
 $\{ i \in \pm[k(n-1)] \,:\, b_i > 0 \}$, where the opener $i$ opens a block of size $k b_i$ (since
 $\pi$ has a zero block, we have that $A(\pi) = \pi$).  
 
 We define a labeling $f$ of the blocks of $\pi$ as follows.  For any block $B \in \pi$, define
 $f(B) \subset \pm [n]$ by 
 \begin{equation*}
 f(B) = \begin{cases}
 \{ i, -i \,:\, v_i = 0 \} & \text{if $B = Z$ is the zero block of $\pi$} \\
 \{ i \,:\, v_i = \omega^j \} \uplus \{ -i \,:\, v_i = -\omega^j \} & 
 \text{if $B$ is nonzero and opened by $j > 0$} \\
 \{i \,:\, v_i = -\omega^j \} \uplus \{ -i \,:\, v_i = \omega^j \} & 
 \text{if $B$ is nonzero and opened by $j < 0$.}
 \end{cases}
 \end{equation*}
 
 It follows that the pair $\psi(v) = (\pi, f)$ 
 satisfies Conditions 1-4 of Lemma~\ref{d-visualization} and is 
 in $M_n(k)$. 
 
 As an example, suppose that $n = 4$, $k = 3$, and $v = (0, \omega^2, \omega^{13}, 0)$.
 We have that $z = 2$, $b_2 = b_{\overline{2}} = b_4 = b_{\overline{4}} = 1$, and 
 $b_i = 0$ for all other subscripts $i$.  By Lemma~\ref{recovery}, there exists a unique
 $k$-D-noncrossing partition $\pi$ of $\pm [12]$ which has
 \begin{itemize}
 \item a zero block of size $12$,
 \item blocks of size $3$ opened by $2$ and $\overline{2}$, and
 \item blocks of size $3$ opened by $4$ and $\overline{4}$.
 \end{itemize}
 The partition $\pi$ has blocks $\pi = \{ \pm B_1, \pm B_2, Z \}$ where
 $B_1 = \{2, 3, 7\}, B_2 = \{4, 5, 6\},$ and $Z = \{1, \overline{1}, 8, \overline{8}, 9, \overline{9} \}$.
 We have that $\psi(v) = (\pi, f)$, where
 the labeling $f$ of the blocks of $\pi$ is given by
 \begin{align*}
 f(\pm B_1) &= \pm \{ 2 \}, \\
 f(\pm B_2) &= \pm \{ \overline{3} \}, \\
 f(Z) &= \{1, \overline{1}, 4, \overline{4} \}.
 \end{align*}
 
 \noindent
 {\bf Case 2. Zero appears exactly once as a coordinate of $v$.}
 As in Case 1, for  $i \in \pm[k(n-1)]$, let $b_i$ be the total number of times either 
 $\omega^i$ or $-\omega^i$
 appears as a coordinate of $v$.  By Lemma~\ref{recovery}, there exists a unique
 $k$-D-noncrossing partition $\pi$ consisting entirely of outer and inner blocks whose 
 outer blocks are opened by $\{ i \,:\, b_i > 0 \}$ such that the outer block opened by $i$
 has size $k b_i$ (since $\pi$ has two inner blocks, we have that $A(\pi) = \pi$).
 
  We define two labelings $f$ and $f'$ of the blocks of $\pi$ as follows.  For any outer block
  $B \in \pi$, define
  \begin{equation*}
  f(B) = f'(B) = \begin{cases}
   \{ i \,:\, v_i = \omega^j \} \uplus \{ -i \,:\, v_i = -\omega^j \} & 
 \text{if $B$ is  opened by $j > 0$} \\
 \{i \,:\, v_i = -\omega^j \} \uplus \{ -i \,:\, v_i = \omega^j \} & 
 \text{if $B$ is  opened by $j < 0$.}
  \end{cases}
  \end{equation*}
  Let $I$ and $-I$ be the inner blocks of $\pi$.  Define
  \begin{align*}
  f(\pm I) &= \pm \{ i \,:\, v_i = 0 \} \\
  f'(\pm I) &= \mp \{i \,:\, v_i = 0 \}.
  \end{align*}
  That is, the labelings $f$ and $f'$ are identical except for swapping the singleton labels of the 
  inner blocks.
 Both of the pairs $(\pi, f)$ and $(\pi, f')$ satisfy Conditions 1-3 of Lemma~\ref{d-visualization}.  
 Exactly one of these pairs also satisfies Condition 4, and is therefore in $M_n(k)$.  Without
 loss of generality, assume that $(\pi, f)$ satisfies Condition 4 and set
 $\psi(v) = (\pi, f)$.
 
 As an example, suppose that $n = 4$, $k = 3$, and 
 $v = (\omega^5, \omega^{11}, 0, \omega^2)$.  We have that 
 $b_5 = b_{\overline{5}} = 1$, $b_2 = b_{\overline{2}} = 2$, and $b_i = 0$ for all other 
 subscripts $i$.  By Lemma~\ref{recovery}, there exists a unique $k$-D-noncrossing partition
 $\pi$ of $\pm [12]$ which has
 \begin{itemize}
 \item only inner blocks and outer blocks,
 \item blocks of size $3$ opened by $5$ and $\overline{5}$, and
 \item blocks of size $6$ opened by $2$ and $\overline{2}$.
 \end{itemize}
 The outer blocks of $\pi$ are seen to be $\pm B_1$ and $\pm B_2$ where
 $B_1 = \{5, 6, 7\}$ and $B_2 = \{2, 3, 4, 8, 9, \overline{1} \}$.  Condition 5 in the definition of 
 $k$-D-noncrossing partitions forces the inner blocks of $\pi$ to be $\pm I$, where
 $I = \{10, \overline{11}, \overline{12} \}$.  The two labelings $f$ and $f'$ of the blocks of $\pi$
 are given by
 \begin{align*}
 f(\pm B_1) = f'(\pm B_1) &= \pm \{ 1 \}, \\
 f(\pm B_2) = f'(\pm B_2) &= \pm \{ \overline{2}, 4 \}, \\
 f(\pm I) &= \pm \{ 3 \}, \\
 f'(\pm I) &= \pm \{ \overline{3} \}.
 \end{align*}
 Both of the pairs $(\pi, f)$ and $(\pi, f')$ satisfy Conditions 1-3 of Lemma~\ref{d-visualization}.
 However, when the labeling $f$ is used there exists a bijection $\phi$ as in Condition 4
 which satisfies $\phi(1) = 2, \phi(4) = 4, \phi(7) = 1,$ and $\phi(10) = 3$ 
 (sending an even number of positive indices to negative indices)
 whereas when the labeling
 $f'$ is used there exists a bijection $\phi$ as in Condition 4 which satisfies
 $\phi(1) = 2, \phi(4) = 4, \phi(7) = 1,$ and $\phi(10) = \overline{3}$
 (sending an odd number of positive indices to negative indices).  It follows that 
 $(\pi, f) \in M_4(3)$ and $(\pi, f') \notin M_4(3)$, so that
 $\psi(v) = (\pi, f)$.
 
 \noindent
 {\bf Case 3. Zero does not appear as a coordinate of $v$.}  As in Cases 1 and 2, 
 let $b_i$ be the total number of times either 
 $\omega^i$ or $-\omega^i$
 appears as a coordinate of $v$.  By Lemma~\ref{recovery}, there exist exactly two
 $k$-D-noncrossing partitions $\pi$ and $\pi' = A(\pi)$ of $\pm [kn]$ which contain only 
 outer and annular blocks such that $i$ opens a block of size $k b_i$ for all 
 $i \in \pm [k(n-1)]$.
 
 We define a labeling $f$ of the blocks of $\pi$ and a labeling $f'$ of the blocks of $\pi'$ as follows.
\begin{align*}
 f(B) &=
 \begin{cases}
   \{ i \,:\, v_i = \omega^j \} \uplus \{ -i \,:\, v_i = -\omega^j \} & 
 \text{if $B \in \pi$ is  opened by $j > 0$} \\
 \{i \,:\, v_i = -\omega^j \} \uplus \{ -i \,:\, v_i = \omega^j \} & 
 \text{if $B \in \pi$ is  opened by $j < 0$.}
 \end{cases} \\
 f'(B) &= 
 \begin{cases}
   \{ i \,:\, v_i = \omega^j \} \uplus \{ -i \,:\, v_i = -\omega^j \} & 
 \text{if $B \in \pi'$ is  opened by $j > 0$} \\
 \{i \,:\, v_i = -\omega^j \} \uplus \{ -i \,:\, v_i = \omega^j \} & 
 \text{if $B \in \pi'$ is  opened by $j < 0$.}
 \end{cases}
 \end{align*}
 Both of the pairs $(\pi, f)$ and $(\pi', f')$ satisfy Conditions 1-3 of Lemma~\ref{d-visualization}.
 Exactly one of these pairs also satisfies Condition 4, and this therefore an element of
 $M_n(k)$.  Without loss of generality, assume that $(\pi, f)$ also satisfies Condition 4 and 
 set $\psi(v) = (\pi, f)$.
 
 As an example of this final case, let $n = 4$, $k = 3$, and 
 $v = (\omega^3, \omega^4, \omega^1, \omega^{12})$.  We have that
 $b_1 = b_{\overline{1}} = 1, b_4 = b_{\overline{4}} = 1$,  
 $b_3 = b_{\overline{3}} = 2$, and $b_i = 0$ for all other subscripts $i$.  By Lemma~\ref{recovery},
 there exist exactly two $k$-D-noncrossing partitions $\pi$ and $\pi' = A(\pi)$ with
 \begin{itemize}
 \item only outer and annular blocks,
 \item blocks of size $3$ opened by $1$ and $\overline{1}$,
 \item blocks of size $3$ opened by $4$ and $\overline{4}$, and
 \item blocks of size $6$ opened by $3$ and $\overline{3}$.
 \end{itemize}
 We choose $\pi = \{ \pm B_1, \pm B_2, \pm B_3 \}$ where
 $B_1 = \{1, 2, \overline{12} \}$, $B_2 = \{4, 5, 6\}$, and
 $B_3 = \{3, 7, 8, 9, \overline{10}, \overline{11} \}$.  This forces
 $\pi' = A(\pi) = \{ \pm B'_1, \pm B'_2, \pm B'_3 \}$ where
 $B'_1 = \{1, 2, 12\}, B'_2 = \{4, 5, 6\},$ and $B'_3 = \{3, 7, 8, 9, 10, 11\}$.  The labelings $f$ of the blocks
 of $\pi$ and $f'$ of the blocks of $\pi'$ are given by
 \begin{align*}
 f(\pm B_1) = f'(B'_1) &= \pm \{ 3 \}, \\
 f(\pm B_2) = f'(B'_2) &= \pm \{ 2 \}, \\
 f(\pm B_3) = f'(B'_3) &= \pm \{1, \overline{4} \}.
 \end{align*}
 Both of the pairs $(\pi, f)$ and $(\pi', f')$ satisfy Conditions 1-3 of Lemma~\ref{d-visualization}.  
 However, using the pair $(\pi, f)$ there exists a bijection $\phi$ as in Condition 4 
 which satisfies $\phi(1) = 3, \phi(4) = 2, \phi(7) = 1,$ and $\phi(10) = 4$ (sending an even
 number of positive indices to negative indices) whereas using the pair $(\pi', f')$ there 
 exists a bijection $\phi$ as in Condition 4 which satisfies 
 $\phi(1) = 3, \phi(4) = 2, \phi(7) = 1,$ and $\phi(10) = \overline{4}$ (sending an odd number
 of positive indices to negative indices).  We conclude that 
 $(\pi, f) \in M_4(3)$ and $(\pi', f') \notin M_4(3)$ and set $\psi(v) = (\pi, f)$.

 We leave it to the reader to check that $\phi$ and $\psi$ are mutually inverse bijections and
 that they are $W \times \ZZ_{kh}$-equivariant.
\end{proof}

\section{Type A}
\label{Type A}

We prove the Weak Conjecture in type A.  The proof is 
not so different from the $k = 1$ case presented in \cite{ARR} and is a direct counting 
argument.
As in the $k = 1$ case, a proof of the Intermediate Conjecture in type A has 
been hindered by the lack of a sufficiently simple hsop $\Theta$ which would allow an explicit
description of $V^{\Theta}$.

\subsection{The Main Conjecture in type A}  
Let $W = \symm_n$ have type A$_{n-1}$, so that $h = n$ and let $k \geq 1$.
Fix a pair $(w, g^d) \in \symm_n \times \ZZ_{kn}$ 
with $0 \leq d < kn$
and a primitive
$kn^{th}$ root of unity $\omega$.  Also fix the Coxeter element 
$c = (1,2, \dots, n) \in \symm_n$.

Recall that the Weak Conjecture asserts that
the number of equivalence classes in $\Park^{NC}_{\symm_n}(k)$ which 
are fixed by $(w, g^d)$ equals
$(kn+1)^{\mult_w(\omega^d)}$, where $\mult_w(\omega^d)$ is the multiplicity 
of the eigenvalue $\omega^d$ in the action of $w$ on the reflection representation.
The spectrum of $w$ can be determined from its cycle decomposition, so that 
$\mult_w(\omega^d)$ is the number of cycles of $w$, less one, when $d = 0$ and
$\mult_w(\omega^d)$ is the number of cycles of $w$ of length divisible by $m$ when
$d > 0$, where $m := \frac{kn}{gcd(kn,d)} \geq 2$ is the order of $g^d$.  We consider the cases
$d = 0$ and $d > 0$ separately.

When $d = 0$, the number
of fixed points of $(w, 1)$ equals $(kh+1)^{\dim(V^w)}$ because 
of the $\symm_n$-isomorphism between
$\Park^{NC}_{\symm_n}(k)$ and $\Park_n(k)$ given in Section~\ref{Introduction}.

We therefore assume that $d > 0$ and let
\begin{equation*}
r_m(w) := | \{ \text{cycles of $w$ of length divisible by $m$} \} |.
\end{equation*}
We aim to prove 
\begin{equation}
\label{main-a}
(kn+1)^{r_m(w)} = | \Park^{NC}_W(k) ^{(w, g^d)} |.
\end{equation}
The idea behind proving Equation~\ref{main-a} is to show that both sides count
the following set of objects.  

\begin{defn}
Let $w \in \symm_n$ and $g^d \neq 1$ as above.  Identify $g$ with the permutation
$(1,2, \dots, kn) (0) $ acting on $[kn] \cup \{ 0 \}$.  Call a function
$f: [n] \rightarrow [kn] \cup \{ 0 \}$ {\sf $(w, g^d)$-equivariant} if 
\begin{equation*}
f(w(j)) = g^d f(j)
\end{equation*}
for every $j \in [n]$.
\end{defn}

Proving that $(w, g^d)$-equivariant functions are counted by the right hand side of 
Equation~\ref{main-a} is simple.

\begin{lemma}
\label{prep-1}
The number of $(w, g^d)$-equivariant functions $f: [n] \rightarrow [kn] \cup \{ 0 \}$ equals
$(kn+1)^{r_m(w)}$.
\end{lemma}
\begin{proof}
We consider how a typical $(w, g^d)$-equivariant function
$f: [n] \rightarrow [kn] \cup \{ 0 \}$ is constructed.  For any $j \in [n]$, the value 
$f(j)$ determines the values of $f(i)$ for any $i$ in the same cycle of $j$ as $w$.  Moreover, if
$j$ belongs to a cycle of $w$ of length not divisible by $m$, we must have that 
$f(j) = 0$.  This leads to $kn+1$ choices for each cycle counted by $r_m(w)$.
\end{proof}

\begin{defn} (\cite[Definition 9.3]{ARR})
Given $w \in \symm_n$ and $m \geq 2$, say that a set partition 
$\pi = \{ A_1, A_2, \dots, \}$ of $[n]$ is $(w, m)$-admissible if
\begin{itemize}
\item $\pi$ is $w$-stable in the sense that $w(\pi) = \{ w(A_1), w(A_2), \dots \} = \pi$, and
\item at most one block $A_{i_0}$ of $\pi$ is itself $w$-stable in the sense that
$w(A_{i_0}) = A_{i_0}$, and
\item all other blocks besides the $w$-stable block $A_{i_0}$ (if present) are permuted
by $w$ in orbits of length $m$.  That is, the sets 
$ A_i, w(A_i), \dots, w^{m-1}(A_i) $ are pairwise distinct, but
$w^m(A_i) = A_i$.
\end{itemize}
\end{defn}

\begin{lemma}
\label{prep-2}
The number of $(w, g^d)$-equivariant functions 
$f: [n] \rightarrow [kn] \cup \{ 0 \}$ is
\begin{equation*}
\sum_{\pi} kn (kn - m) (kn - 2m) \cdots (kn - (r_{\pi} - 1)m ),
\end{equation*}
where the sum is over $(w, m)$-admissible partitions $\pi$ of $[n]$ and
 $r_{\pi}$ is the number of $w$-orbits of blocks of $\pi$ having length $m$.
\end{lemma}

\begin{proof}
This follows the same reasoning as the proof of \cite[Proposition 9.4]{ARR}.  We associate to
any $(w, g^d)$-equivariant function $f: [n] \rightarrow [kn] \cup \{0\}$ 
a partition $\pi = \{A_1, A_2, \dots \}$ on the
domain $[n]$ by letting the $A_i$ be the nonempty fibers $f^{-1}(j)$ of $f$.  Since
$f$ is $(w, g^d)$-equivariant, $\pi$ is $(w, m)$-admissible.  Moreover, the partition
$\pi$ contains a unique $w$-stable block $f^{-1}(0)$ if and only if 
$f$ takes on the value $0$.

On the other hand, if we fix a $(w, m)$-admissible partition $\pi$ of $[n]$, we can count
the number of $(w, g^d)$-equivariant functions $f$ which induce the partition 
$\pi$.  If $\pi$ has a $w$-stable block $A_{i_0}$, set $f(A_{i_0}) = 0$;  otherwise, 
$f$ does not take on the value $0$.  To determine the value of $f$ on the remaining
elements of $[n]$, choose representative blocks 
$A_1, A_2, \dots, A_{r_{\pi}}$ of $\pi$ from each of the $w$-orbits having length $m$.  After
choosing the value $f(A_1) = j_1$ from the $kn$ choices in $[kn]$, the requirement that
$f$ is $(w, g^d)$-equivariant forces the following values of $f$ (interpreted modulo $kn$):
\begin{align*}
f(w(A_1)) &= j_1 + d, \\
f(w^2(A_1)) &= j_1 + 2d, \\ 
&\vdots \\
f(w^{m-1}(A_1)) &= j_1 + (m-1)d.
\end{align*}
This leaves $kn - m$ choices for $f(A_2) = j_2$, forcing
\begin{align*}
f(w(A_2)) &= j_2 + d, \\
f(w^2(A_2)) &= j_2 + 2d, \\ 
&\vdots \\
f(w^{m-1}(A_2)) &= j_2 + (m-1)d.
\end{align*}
This then leaves $kn - 2m$ choices for $f(A_3) = j_3$, etc.
\end{proof}

Next, we relate the summands in Lemma~\ref{prep-2} to certain fixed points
inside $\Park^{NC}_{\symm_n}(k)^{(w, g^d)}$.  We identify
$k$-$\symm_n$-noncrossing parking functions with pairs $(\pi, f')$ where $\pi$ is a
$k$-divisible noncrossing partition of $[kn]$ and $f'$ assigns every block $B$ of 
$\pi$ to a subset $f'(B)$ of $[n]$ so that $|B| = k |f'(B)|$ and 
$\biguplus_{B \in \pi} f'(B) = [n]$.  Observe that $f'$ naturally determines a partition
$\tau(f') = \{ f'(B) \,:\, B \in \pi \}$
of $[n]$.

\begin{lemma}
\label{prep-3}
Suppose $(\pi, f') \in \Park^{NC}_{\symm_n}(k)$  is fixed by
$(w, g^d)$.  The partition $\tau(f')$ of $[n]$ determined by $f'$ is $(w, m)$-admissible.

Conversely, 
if $\tau$ is a fixed $(w, m)$-admissible partition of $[n]$,
the number of elements of $\Park^{NC}_{\symm_n}(k)$ which 
are fixed by $(w, g^d)$ and
induce the partition
$\tau$ is
\begin{equation*}
kn (kn - m) (kn - 2m) \cdots (kn - (r_{\tau} - 1)m),
\end{equation*}
where $r_{\tau}$ is as in Lemma~\ref{prep-2}.
\end{lemma}

\begin{proof}
This is similar to the proof of \cite[Proposition 9.5]{ARR}.
To see that $\tau(f')$ is $(w,m)$-admissible,
notice that $g^d$ acts on the $k$-divisible noncrossing partition
$\pi$ of $[kn]$ by the $d^{th}$ power of rotation and $w$ acts by permuting the labels in
$f'$.  
Since $(g^d, w)$ fixes $(\pi, f')$, the noncrossing partition $\pi$ must satisfy $g^d.\pi = \pi$.
This implies that $\pi$ can have at most one block left invariant under $g^d$ (and this block,
if present, has size divisible by $m$), and the other blocks of $\pi$ have $m$ distinct images
under the action of powers of $g^d$.
This means that $\tau(f')$ is $w$-stable, has at most one block which is itself 
$w$-stable, and the other blocks of $\tau(f')$ break up into $m$-element orbits under
the action of $w$.  We conclude that $\tau(f')$ is $(w,m)$-admissible.

The enumerative assertion is a consequence of a product formula of 
Athanasiadis \cite[Theorem 2.3]{Athanasiadis} counting $m$-fold symmetric noncrossing
partitions.  Define the {\sf type} of a $(w, m)$-admissible partition $\tau$ of $[n]$ to be the 
sequence $\mu = (\mu_1, \mu_2, \dots, \mu_{n})$, where 
$\tau$ has $\mu_j$ different $w$-orbits of length $m$ in which the blocks have size $j$.  
Also, define the type of a noncrossing $k$-divisible partition 
$\pi$
of $[kn]$ which has $m$-fold rotational symmetry to be the sequence
$\mu = (\mu_1, \mu_2, \dots, \mu_{n}) $, 
where $\pi$ has $\mu_j$  
length $m$ orbits of blocks under $m$-fold rotation which
have size $kj$.  
When $\tau(\pi, f')$ is associated to $(\pi, f')$ as above, observe that $\pi$ and $\tau$ have the
same type $\mu$.

Fix a $(w, m)$-admissible partition $\tau$ of $[n]$ of type $\mu$.  Athanasiadis proved that the number
of $d$-fold symmetric 
$k$-divisible
noncrossing partitions $\pi$ of $[kn]$ of type $\mu$ is
\begin{equation}
\label{athanasiadis-equation}
\frac{ (\frac{kn}{m}) (\frac{kn}{m} - 1) \cdots (\frac{kn}{m} - (r_{\tau} - 1) )}
{\mu_1! \mu_2! \cdots \mu_{n}!}.
\end{equation}
For each such $\pi$, and for each block size $j$, the labelling $f'$ must send the 
$\mu_j$ different $d$-fold rotation orbits in $\pi$ having blocks of size $kj$ to the $\mu_j$ 
different $w$-orbits of $\tau$ having blocks of size $j$.  There are $\mu_j!$ ways to do this 
matching.  Having picked such a matching, for each of the 
$r_{\tau} = \mu_1 + \mu_2 + \cdots + \mu_{n}$ matched orbit pairs, $f'$ has
$m$ choices for how to align them cyclically.

Hence there are $m^{r_{\tau}} \mu_1! \mu_2! \cdots \mu_{n}!$ ways to choose the image sets
$f'(B_i)$ after making the choice of $\pi$.  Together with Equation~\ref{athanasiadis-equation},
this gives the desired count:
\begin{equation*}
m^{r_{\tau}} \mu_1! \mu_2! \cdots \mu_n! \frac{ (\frac{kn}{m}) (\frac{kn}{m} - 1) \cdots (\frac{kn}{m} - (r_{\tau} - 1) )}
{\mu_1! \mu_2! \cdots \mu_{n}!} = kn (kn - m) (kn - 2m) \cdots (kn - (r_{\tau} - 1)m).
\end{equation*}
\end{proof}

\begin{proposition}
\label{weak-type-a}
The Weak Conjecture holds when $W$ is of type A.
\end{proposition}
\begin{proof}
Combine Lemmas \ref{prep-1}, \ref{prep-2}, and \ref{prep-3}.
\end{proof}

\section{Nonnesting analogs}
\label{Nonnesting analogs}

We introduce the third parking space of this paper.  The definition of this parking space is 
combinatorial and involves nonnesting partitions rather than noncrossing partitions.
The nonnesting analog only exists in crystallographic type
and involves the more refined data of the root system $\Phi$ rather than the reflection group 
$W(\Phi)$.  Moreover, the resulting module $\Park^{NN}_{\Phi}(k)$ carries only an action of 
$W$ rather than an action of $W \times \ZZ_{kh}$.

\subsection{$\Phi$-nonnesting partitions}  Let $\Phi$ be a crystallographic root system 
with Weyl group $W$, let $\Pi \subset \Phi$ be a choice of simple system, and let 
$\Phi^+ \subset \Phi$ be the corresponding set of positive roots.
We 
equip the set $\Phi^+$ with the partial order defined by $\alpha \leq \beta$ if and only if
$\beta - \alpha$ is a nonnegative linear combination of the simple roots.  Given $\alpha \in \Phi$,
let $H_{\alpha}$ be the hyperplane in $V$ perpendicular to $\alpha$.
The following definition is due to
Postnikov \cite{Postnikov}.

\begin{defn}
\label{nonnesting-partitions}
A {\sf $\Phi$-nonnesting partition} is an antichain $A$ in the positive root poset $\Phi^+$.  The
map from antichains in $\Phi^+$ to $\LLL$ defined by
\begin{equation}
A \mapsto \bigcap_{\alpha \in A} H_{\alpha}
\end{equation}
is injective.  A flat $X \in \LLL$ is called {\sf nonnesting} if it is in the image of this map.
\end{defn}

There is a natural bijection between the set of antichains in any poset $P$ and the set 
of order filters (up-closed subsets) of $P$ obtained by letting an antichain $A$ correspond
to the order filter $\{ x \in P \,:\, \text{$x \geq a$ for some $a \in A$} \}$.  This gives a bijection
between the $\Phi$-nonnesting partitions and the collection of order filters in $\Phi^+$.

\begin{example}
\label{nonnesting-type-A}

\begin{figure}
\includegraphics[scale=0.4]{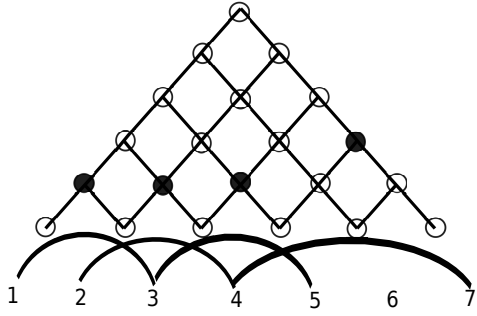}
\caption{A nonnesting partition of type A$_6$}
\label{fig:ANN}
\end{figure}

The term `nonnesting' comes from the case of type A.  When $W$ has type A$_{n-1}$, we identify
$\Phi^+$ with $\{ \alpha_{i,j}  \,:\, 1 \leq i < j \leq n \}$, 
where $\alpha_{i,j} = e_i - e_j$ and $\{ e_1, \dots, e_n \}$ is the standard
basis of $\RR^n$.  The corresponding simple roots are 
$\{ \alpha_{i,i+1} \,:\, 1 \leq i \leq n-1 \}$.  It follows that the partial
order on $\Phi^+$ is given by $\alpha_{i,j} \leq \alpha_{k,l}$ if and only if 
$k \leq i < j \leq l$.  
The upper part of
Figure~\ref{fig:ANN} shows the poset $\Phi^+$ for the case $n = 7$ together with
the antichain 
$\{ \alpha_{1,3}, \alpha_{2,4}, \alpha_{3,5}, \alpha_{4,7} \}$ drawn in black. 

Given an antichain $A$ in $\Phi^+$, we obtain a set partition $\pi$ of $[n]$ by letting $\pi$ be
generated by $i \sim j$ whenever $i < j$ and $\alpha_{i,j} \in A$.  This identifies the collection of 
antichains in $\Phi^+$ with the partitions $\pi$ of $[n]$ whose arc diagrams do not `nest':  
there do not exist indices $1 \leq a < b < c < d \leq n$ such that $a, d$ and $b, c$ are two pairs
of consecutive blockmates in $\pi$ (necessarily belonging to different blocks of $\pi$).  The
lower part of Figure~\ref{fig:ANN} shows the nonnesting partition
$\{ 1, 3, 5 / 2, 4, 7 / 6 \}$ which corresponds to the given antichain.
\end{example}

\subsection{Extended Shi arrangements and geometric multichains}
For the rest of this section, we consider the reflection representation
$V$ over $\RR$.

Given a Fuss parameter $k \geq 1$, the 
{\sf extended Shi arrangement} is the hyperplane arrangement in 
$V$ defined by
\begin{equation}
\Shi^k(\Phi) := \{ H_{\alpha, m} \,:\, \text{$\alpha \in \Phi^+$, $-k < m \leq k$} \}.
\end{equation}
A region of $\Shi^k(\Phi)$ is called {\sf dominant} if it lies in the dominant cone of the
Coxeter arrangement $\Cox(\Phi)$ defined by 
$\langle x, \alpha \rangle > 0$ for all $\alpha \in \Phi^+$.  

Given any dominant region $R$ of $\Shi^k(\Phi)$, let
$\mathcal{F}(R) = (F_1(R) \supseteq F_2(R) \supseteq \dots \supseteq F_k(R))$
be the descending multichain of filters in $\Phi^+$ given by
\begin{equation*}
F_m(R) := \{ \alpha \in \Phi^+ \,:\, \text{ $\langle x, \alpha \rangle > m$ holds on $R$} \}.
\end{equation*}
Clearly a dominant region of $\Shi^k(\Phi)$ is determined by its multichain of filters,
but not every multichain of filters in $\Phi^+$ corresponds to a region of $\Shi^k(\Phi)$.
Athanasiadis \cite{Athanasiadis} made the following key definition to study which 
multichains arise from dominant Shi regions.

\begin{figure}
\includegraphics[scale=0.4]{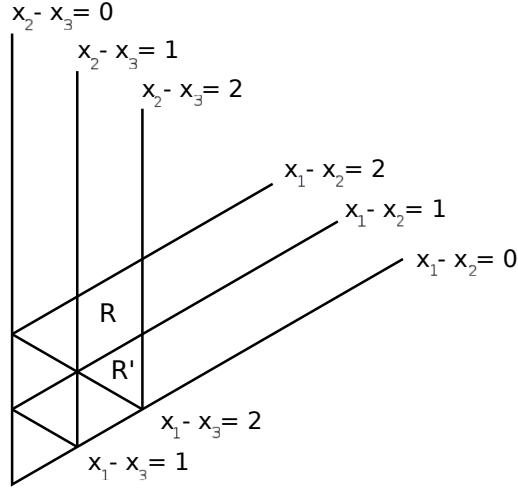}
\caption{The dominant cone of $\Shi^2(\Phi)$ in type A$_2$}
\label{fig:shidom}
\end{figure}

\begin{example}
Figure~\ref{fig:shidom} shown the dominant regions of $\Shi^k(\Phi)$ when $k = 2$ and
$\Phi$ has type A$_2$.  The multichain $\mathcal{F}(R)$ corresponding to the region labeled
$R$ is $(F_1 \supseteq F_2)$, where 
$F_1 = \Phi^+$ and $F_2 =  \{ \alpha_{1,3} \}$.  The multichain 
$\mathcal{F}(R')$ corresponding to the region labeled $R'$ is 
$(F_1' \supseteq F_2')$, where
$F_1' = \{ \alpha_{1,3}, \alpha_{2,3} \}$ and $F_2' = \{ \alpha_{1,3} \}$.
\end{example}

\begin{defn}
Let $\mathcal{F} := (F_1 \supseteq F_2 \supseteq \dots \supseteq F_k)$ 
be a multichain of filters and 
for $1 \leq i \leq k$ let $I_i$ be the order
ideal $I_i := \Phi^+ - F_i$.  The multichain
$\mathcal{F}$ is called
{\sf geometric} if for all indices $1 \leq i, j \leq k$ with $i + j \leq k$ we have
\begin{equation*}
(F_i + F_j) \cap \Phi^+ \subseteq F_{i+j}
\end{equation*}
and
\begin{equation*}
(I_i + I_j) \cap \Phi^+ \subseteq I_{i+j}.
\end{equation*}
\end{defn}

\begin{theorem} (Athanasiadis \cite{Athanasiadis})
\label{athshi}
The map $R \mapsto \mathcal{F}(R)$ given above defines a bijection 
between dominant regions $R$ of $\Shi^k(\Phi)$ and length $k$ geometric
multichains $\mathcal{F}$ of filters in $\Phi^+$.
\end{theorem}

\begin{example}
Let $k = 2$ and
suppose $\Phi$ has type A$_2$.  Consider the 
multichain $\mathcal{F} = (F_1 \supseteq F_2)$
of filters in $\Phi^+$ where $F_1 = F_2 = \{ \alpha_{1,3} \}$.  
The multichain $\mathcal{F}$ is not geometric because 
$\alpha_{1,2}, \alpha_{2,3} \in I_1$ but 
$\alpha_{1,2} + \alpha_{2,3} = \alpha_{1,3} \notin I_2$.  Therefore, there does not exist
a dominant region $R$ of $\Shi^2(\Phi)$ such that $\mathcal{F} = \mathcal{F}(R)$.  
Similarly, the multichain $\mathcal{F}' = (F_1 \supseteq F_2)$, where
$F_1 = \Phi^+$ and $F_2 = \emptyset$, 
is not geometric
because $\alpha_{1,2}, \alpha_{2,3} \in F_1$
but $\alpha_{1,2} + \alpha_{2,3} = \alpha_{1,3} \notin F_2$.  Every $2$-element multichain
$(F_1 \supseteq F_2)$ of filters in $\Phi^+$ other than $\mathcal{F}$ and
$\mathcal{F}'$ is geometric.
\end{example}

Given a dominant region $R$ of $\Shi^k(\Phi)$, a {\sf floor} of $R$ is a nonlinear
hyperplane in $\Shi^k(\Phi)$ which is the affine span of a facet of $R$ and separates 
$R$ from the origin.  We define a parabolic subgroup $W_R$ associated to $R$ to be
 the subgroup of $W$ generated by all reflections $t_{\alpha}$ for 
 $\alpha \in \Phi^+$ such that $H_{\alpha, k}$ is a floor of $R$.
 Athanasiadis \cite{Athanasiadis} gave an explicit set of generators of $W_R$ consisting of 
 reflections corresponding to roots in $\Phi^+$ which are `indecomposible with 
 respect of $\mathcal{F}$', where $\mathcal{F}$ is the filter corresponding to $R$.
 
 \begin{example}
 Let $k = 2$ and suppose $\Phi$ has type A$_2$.  Let $R$ and $R'$ be the dominant regions
 of $\Shi^2(\Phi)$ shown in Figure~\ref{fig:shidom}.  The floors of the region $R$ are 
 $x_1 - x_2 = 1$ and $x_2 - x_3 = 1$.  Since $k = 2$, the parabolic subgroup $W_R$
 corresponding to $R$ is the identity subgroup $1$ of $W = \symm_3$.  
 The only floor of the region $R'$ is $x_1 - x_3 = 2$, so that $W_{R'}$ is the parabolic
 subgroup of $\symm_3$ generated by $(1,3)$.
 \end{example}

\subsection{Nonnesting parking functions}  As with the case of noncrossing parking
functions, nonnesting parking functions will be defined by taking putting an equivalence
relation on a certain subset of $W \times \LLL^k$.  

We have a partial order $\preceq$ on filters in $\Phi^+$ given by {\em reverse} inclusion, i.e.,
$F_1 \preceq F_2 \Leftrightarrow F_1 \supseteq F_2$.  Identifying filters with their
minimal antichains yields a partial order $\preceq$ on antichain in $\Phi^+$ and hence 
a partial order $\preceq$ on nonnesting partitions and nonnesting flats.  This partial order on 
nonnesting flats is {\em different} from the partial order induced by $\LLL$. 
For example, in type A$_2$ we have that 
$X_1 \preceq X_2$ where $X_1$ is defined by $x_1 = x_3$ and $X_2$ is defined by 
$x_1 = x_2$, but these flats are are incomparable in $\LLL$.

We call a multichain of flats $\mathcal{X} = (X_1 \preceq \dots \preceq X_k)$ {\sf geometric}
if the corresponding multichain of order filters 
in $\Phi^+$ is geometric.  By Theorem~\ref{athshi}, geometric $k$-multichains of flats
naturally label regions of $\Shi^k(\Phi)$.  For any 
geometric $k$-multichain of flats $\mathcal{X} = (X_1 \preceq \dots \preceq X_k)$, define
$W_{\mathcal{X}} := W_R$, where $R$ is the region  of $\Shi^k(\Phi)$
labeled by $\mathcal{X}$.

\begin{defn}
Let $k \geq 1$ be a Fuss parameter.  A {\sf $k$-$\Phi$-nonnesting parking function} is an 
equivalence class in
\begin{equation*}
\{ (w, X_1 \preceq \dots \preceq X_k) \,:\, w \in W, \text{$X_i$ nonnesting, 
$X_1 \preceq \dots \preceq X_k$ geometric} \} / \sim
\end{equation*}
where $(w, X_1 \preceq \dots \preceq X_k) \sim (w', X_1' \preceq \dots \preceq X_k')$ if and
only if $X_i = X_i'$ for all $i$ and 
$w W_{\mathcal{X}} = w' W_{\mathcal{X}}$, where 
$\mathcal{X} = (X_1 \preceq \dots \preceq X_k)$.  The equivalence class of 
$(w, X_1 \preceq \dots \preceq X_k)$ is denoted 
$[w, X_1 \preceq \dots \preceq X_k]$.  The set of $k$-$\Phi$-nonnesting parking functions
is denoted $\Park^{NN}_{\Phi}(k)$.
\end{defn}

Although we always have an inclusion 
$W_{\mathcal{X}} \subseteq W_{X_k}$, we do not have equality in general.  
In Figure~\ref{fig:shidom}, the multichain $\mathcal{X}$ corresponding to the 
region labeled $R$ has $W_{\mathcal{X}}$ given by the identity subgroup of 
$W = \symm_3$ and $W_{X_k}$ given by the parabolic subgroup of $\symm_3$ generated
by $(1,3)$.

The set $\Park^{NN}_{\Phi}(k)$ is endowed with a $W$-module structure given by
\begin{equation}
w.[v, X_1 \preceq \dots \preceq X_k] := [wv, X_1 \preceq \dots \preceq X_k].
\end{equation}
The $W$-orbits of $\Park^{NN}_{\Phi}(k)$ biject with geometric length $k$ multichains
of nonnesting partitions, and hence with dominant regions of $\Shi^k(\Phi)$.  The 
$W$-stabilizer of the parking function
$[1, \mathcal{X}]$ is $W_{\mathcal{X}}$, so that
\begin{equation}
\label{nonnesting-reformulation}
\Park^{NN}_{\Phi}(k) \cong_W \bigoplus_{\mathcal{X}} {\bf 1}_{W_{\mathcal{X}}}^W \cong_W
\bigoplus_R {\bf 1}_{W_R}^W,
\end{equation}
where the direct sums are over length $k$ geometric multichains $\mathcal{X}$ of 
nonnesting partitions and 
dominant regions $R$ of $\Shi^k(\Phi)$.

\begin{proposition}
\label{nonnesting-character}
There is a $W$-equivariant bijection  
\begin{equation*}
\Park^{NN}_{\Phi}(k) \cong_W Q / (kh+1) Q, 
\end{equation*}
where
$Q = \ZZ[\Phi]$ is the root lattice corresponding to $\Phi$.  In particular, if 
$\chi: W \rightarrow \CC$ is the character of $\Park^{NN}_{\Phi}(k)$, we have
\begin{equation*}
\chi(w) = (kh+1)^{\dim(V^w)}.
\end{equation*}
\end{proposition}
\begin{proof}
This is essentially a reformulation of results of Haiman \cite{Haiman} and 
Athanasiadis \cite{Athanasiadis}.  
Let $A_0 \subset V$ be the closure of the  {\sf fundamental alcove} defined
by the equations $\langle x, \alpha \rangle \geq 0$ for every simple root $\alpha$
and $\langle x, \theta^{\vee} \rangle \leq 1$, where $\theta \in \Phi^+$ is the highest root 
(the unique maximal element in the poset $\Phi^+$) and $\theta^{\vee}$ is the corresponding
coroot.  We have that $A_0$ is the closure of
the unique dominant region of $\Shi^k(\Phi)$ whose closure contains the origin.
For any positive integer $p$, let 
$p A_0$ be the corresponding dilation of $A_0$.  

For any positive integer $p$,
Haiman \cite[Lemma 7.4.1]{Haiman} gave a natural bijective correspondence between 
$W$-orbits in $Q / p Q$ and points in $Q \cap p A_0$ such that the $W$-stabilizer of 
an orbit in $Q / p Q$ corresponding to $x \in Q \cap p A_0$ is the parabolic subgroup of $W$ generated
by the (reflections corresponding to the) walls of $p A_0$ which contain $x$.

Specializing at $p = kh+1$, Athanasiadis \cite[Theorem 4.2]{Athanasiadis} used the affine Weyl
group to give a bijection between points in $Q \cap (kh+1)A_0$ and dominant 
regions of $\Shi^k(\Phi)$.  (Athanasiadis gives a bijection onto dominant regions of the
$k$-extended Catalan arrangement, but this arrangement coincides with $\Shi^k(\Phi)$ in the 
dominant cone.)
Let $x \in Q \cap (kh+1) A_0$ and let $R$ be the corresponding 
dominant region of $\Shi^k(\Phi)$.  Athanasiadis proves that the set of roots corresponding to the 
walls of $(kh+1) A_0$ containing $x$ is $W$-conjugate 
(up to sign)
to the set of roots $\alpha$ corresponding to the 
floors of $R$ of the form  $H_{\alpha, k}$ (where the element $u \in W$ giving rise to this
conjugacy depends on $x$) \cite[Proof of Theorem 4.2]{Athanasiadis}.  Hence, the parabolic 
subgroups of $W$ generated by the reflections in these sets of roots are $W$-conjugate and 
the coset representation ${\bf 1}_{W_R}^W$ of $W$ is isomorphic to ${\bf 1}_{W'}^W$, where
$W'$ is the parabolic subgroup of $W$ generated by reflections in the roots corresponding to 
walls of $(kh+1)A_0$ which contain $x$.

Combining the results of Haiman and Athanasiadis, we conclude that  
the finite torus $Q / (kh+1) Q$ is $W$-isomorphic to 
$\bigoplus_{R} {\bf 1}_{W_R}^W$, where the direct sum is over dominant regions $R$ of 
$\Shi^k(\Phi)$.  This is equivalent to the statement of the proposition by the isomorphisms in
(\ref{nonnesting-reformulation}).
\end{proof}

In contrast to the noncrossing case,
the proof
of Proposition~\ref{nonnesting-character} is uniform.  However, there is no known analog
of the positive root poset or Shi arrangement in noncrystallographic type, and hence no analog
of $\Park^{NN}_{\Phi}(k)$ when $\Phi$ is not crystallographic.  

We can take $W$-invariants in Proposition~\ref{nonnesting-character} to get the following additional
consequence of the Weak Conjecture (again derived uniformly modulo the Weak Conjecture).

\begin{corollary}
\label{nc-nn-corollary}
Suppose that the Weak Conjecture holds for $W = W(\Phi)$ and $W$ is crystallographic.
Then the number of $k$-$W$-noncrossing partitions equals the number of length $k$ geometric
multichains of nonnesting flats.
\end{corollary}

\section{Open problems}
\label{Open problems}

The most obvious open problem is to prove the Weak Conjecture for the exceptional
reflection groups EFH.  As previously mentioned, this is a priori an infinite problem since
the Fuss parameter $k$ could be any positive integer.  
More ambitiously, one could hope for
a type-uniform proof of the strong conjecture in all types.

\begin{problem}
\label{prove-main-conjecture}
Give a uniform proof of the Strong or Intermediate Conjectures.
\end{problem}

In light of Etingof's result \cite[Theorem 12.1]{ARR}, one could assume that the locus
$V^{\Theta}$ is reduced and still obtain a uniform proof of the Intermediate Conjecture.

The nonnesting parking space $\Park^{NN}_{\Phi}(k)$ carries an 
action of $W$, but the author is unaware of a 
natural action of the full cyclic group $W \times \ZZ_{kh}$ which gives
$\Park^{NN}_{\Phi}(k) \cong_{W \times \ZZ_{kh}} \Park^{NC}_W(k)$.  This inability to extend the 
action beyond $W$ in the nonnesting case persists even in the $k = 1$ case \cite{ARR}.  

\begin{problem}
\label{nc-nn-extended-structure}
Define a $W \times \ZZ_{kh}$-set structure on $\Park^{NN}_W(k)$ which 
gives $\Park^{NN}_{\Phi}(k) \cong _{W \times \ZZ_{kh}} \Park^{NC}_W(k)$.
\end{problem}

While this paper has focused exclusively on real reflection groups, there has 
the notions of Coxeter elements, absolute order, and noncrossing partitions make sense
in the broader context of well-generated complex reflection groups.  It is natural to ask 
how much of the theory presented in \cite{ARR} and this paper generalize to the complex setting.

\begin{problem}
Generalize the constructions in \cite{ARR} and in this paper to well-generated complex reflection
groups (or perhaps even to all reflection groups).
\end{problem}

When $W$ is a well-generated complex reflection group acting on $V$, one can define the 
intersection lattice $\LLL$ to be  $\{ V^w \,:\, w \in W \}$ as in the real case.
One also has a well-behaved notion of a Coxeter element and absolute order, and hence a
well-defined notion of a noncrossing flat in $\LLL$.  By taking equivalence classes of pairs
$(w, X_1 \leq \dots \leq X_k)$ with $w \in W$ and $X_i \in \LLL$ noncrossing, one could
define $\Park^{NC}_W(k)$ in the well-generated complex setting as before; this set carries an
action of $W \times \ZZ_{kh}$.  The Weak Conjecture extends to this setting.

Assuming $W$ is well-generated and $\Theta$ is a hsop of degree $kh+1$ carrying $V^*$, we 
still get a $W \times \ZZ_{kh}$-module structure on the quotient 
$\CC[V] / (\Theta)$
by letting $W$ act by linear substitutions and $\ZZ_{kh}$ scale by a primitive $kh^{th}$ root of unity
in degree $1$.  Results of Bessis and Reiner \cite{BessisR} show that 
Equation~\ref{hsopchar} still gives the character $W \times \ZZ_{kh} \rightarrow \CC$ of this module.
Writing $\Theta = (\theta_1, \dots, \theta_n)$, we can take a basis $(x_1, \dots, x_n)$ of $V^*$ such
that the linear map $x_i \mapsto \theta_i$ is $W$-equivariant and define $\Park^{alg}_W(k)$ to
be
$\CC[V] / (\Theta - \xx)$ as in the real case.  The proof of \cite[Proposition 2.11]{ARR} goes through
to show that $\Park^{alg}_W(k) \cong_{W \times \ZZ_{kh}} \CC[V] / (\Theta)$.  One can formulate
the Strong and Intermediate Conjectures in the same way.

Unfortunately, the rational Cherednik algebra theory which {\it uniformly} guarantees the existence 
of hsops $\Theta$ of degree $kh+1$ carrying $V^*$ is at present only available in the context 
of real reflection groups.  Moreover, Etingof's proof that the variety corresponding to 
$\Park^{alg}_W(k)$ is reduced makes use of root theoretic data associated to $W$ which
is unavailable in the complex setting.  However, for $W$ other than $\symm_n$
 contained in the infinite families 
$G(r, 1, n)$ and $G(e, e, n)$, one can obtain ad hoc hsops using powers of
coordinate functions such that the corresponding varieties are obviously reduced.

The definition of the nonnesting parking space $\Park^{NN}_{\Phi}(k)$ makes direct use of the 
crystallographic
root system
$\Phi$.  Since root systems are unavailable in the complex setting, it is entirely unclear how 
to generalize nonnesting parking functions to complex reflection groups.

\section{Acknowledgements} 
The author is grateful to Drew Armstrong, Christian Krattenthaler, 
and Vic Reiner for helpful conversations.

\end{document}